\documentclass[twocolumn,10pt]{article}
\usepackage[top=.6in, bottom=.6in, left=.4in, right=.4in]{geometry}
\pdfoutput=1
\usepackage{amsmath,amsfonts,amscd,amssymb}
\usepackage{graphicx}
\usepackage{epstopdf}
\usepackage{overpic}
\usepackage{cancel}
\usepackage{rotating}
\usepackage{url}
\usepackage{caption}
\usepackage{color}
\usepackage{rotating}
\usepackage{multirow}
\usepackage{wrapfig}
\usepackage{mathtools}
\usepackage{subeqnarray}
\usepackage{setspace}
\usepackage{booktabs}
\usepackage[bottom,flushmargin,hang,multiple]{footmisc}
\usepackage{lipsum}
\newcommand\blfootnote[1]{
  \begingroup
  \renewcommand\thefootnote{}\footnote{#1}
  \addtocounter{footnote}{-1}
  \endgroup
}

\usepackage{etoolbox}
\patchcmd{\thebibliography}
  {\settowidth}
  {\setlength{\itemsep}{-0.5pt plus -0.5pt}\settowidth}
  {}{}
\apptocmd{\thebibliography}
  {\footnotesize}
  {}{}

\usepackage{amsmath,amsfonts,amscd,amssymb,amsthm}
\usepackage{fancyhdr}
\usepackage{sidecap}
\usepackage{enumitem}
\usepackage{mathrsfs}
\usepackage{ifthen,tabularx,graphicx,multirow}
\usepackage{xargs}
\usepackage{tcolorbox}
\usepackage{enumerate}
\usepackage{comment}
\usepackage{algorithm}
\usepackage{booktabs}
\usepackage{varwidth}
\usepackage{graphicx}
\usepackage{epstopdf}
\usepackage{mathtools}
\usepackage{setspace}
\usepackage[bottom,flushmargin,hang,multiple]{footmisc}

\usepackage[pdfencoding=auto, psdextra]{hyperref}
\usepackage{algpseudocode}
\usepackage[capitalise,noabbrev]{cleveref}
\usepackage{lscape}
\usepackage{afterpage}
\usepackage{multirow}

\usepackage{arydshln}
\usepackage[font=footnotesize,labelfont=bf]{caption}
\usepackage{cite}

\newtheorem{theorem}{Theorem}
\newtheorem{lemma}[theorem]{Lemma}

\newtheorem{definition}[theorem]{Definition}

\theoremstyle{definition}
\newtheorem{example/}[theorem]{Example}

\newtheorem{remark}[theorem]{Remark}
\newtheorem{remarknn/}{Remark}

\newenvironment{example}
  {\begin{example/}}
  {\end{example/}}
	
\numberwithin{theorem}{section}

\newcommand*{\MH}{{\mathcal{M}_{\mathrm{H}}}}

\newcommand*{\Md}{{\mathcal{M}_{\mathrm{dec}}}}
\newcommand*{\spec}{{\mathrm{Sp}}}
\newcommand*{\dH}{{d_{\mathrm{H}}}}
\newcommand*{\dist}{{\mathrm{dist}}}
\newcommand*{\cl}[1]{{\mathrm{Cl}\left(#1\right )}}
\newcommand*{\dd}{{\,\mathrm{d}}}

\newcommand{\Uv}{\mathbf{U}}
\newcommand{\Vv}{\mathbf{V}}

\newcommand{\Xv}{\mathbf{X}}
\newcommand{\Yv}{\mathbf{Y}}

\title{Adversarial dynamical systems characterize when\\data-driven learning succeeds or fails}

\author{Matthew J. Colbrook$^{1*}$, Igor Mezić$^{2\dagger}$, Alexei Stepanenko$^{1}$\\
\footnotesize{$^1$ DAMTP, University of Cambridge, Cambridge, CB3 0WA, UK.\vspace{-1mm}}\\
\footnotesize{$^2$ University of California, Santa Barbara, CA 93106, USA.}
}

\date{}

\begin{document}

\twocolumn[
  \begin{@twocolumnfalse}
	\maketitle

\begin{abstract}
Many systems resist analytical modeling, making data-driven inference of dynamics important. Yet data-driven methods can fail to converge or generalize, leaving open a central question: \textit{When can system behavior be learned reliably from data, and when is such learning impossible?} We answer this question using adversarial dynamical systems to identify the boundary between accessible and inaccessible regimes. In Koopman operator learning, a leading framework for representing nonlinear dynamics through linear spectral objects, we design optimal data-driven spectral algorithms with convergence and certification guarantees under conditions arising broadly in physical systems. This yields a convergence theory for Koopman-operator approximations and resolves a longstanding open problem in Koopman spectral analysis. Conversely, by constructing adversarial systems, we prove matching impossibility results: without these conditions, no single-sequence limiting procedure can guarantee learning, regardless of data quality. These results sharply characterize when data-driven spectral learning can succeed and when it must fail. We validate the framework on oscillators, chaotic fluid flows and Arctic sea ice concentration forecasting. In the latter, we uncover hidden modes of Arctic sea ice decline, deliver long-range forecasts with geographic error bounds, and outperform state-of-the-art dynamical and deep learning models at substantially lower computational cost, enabling real-time deployment on standard CPUs.
\vspace{6mm}
\end{abstract}
\end{@twocolumnfalse}
]

\blfootnote{\scriptsize$^*$ m.colbrook@damtp.cam.ac.uk, $^\dagger$ mezic@ucsb.edu}

\setcounter{section}{1}

Models across science often involve systems that evolve
over time, known as dynamical systems. These systems have long been used to understand, predict, and control complex behavior across physics, chemistry, biology, and medicine. Yet in many fields such as climate science, neuroscience, and robotics, systems are often too complex for direct analysis, or their governing equations are unknown. Machine learning (ML) has transformed the analysis of complex data \cite{lecun2015deep}, with breakthroughs in protein structure prediction \cite{jumper2021highly} (see also the 2024 Nobel Prize in Chemistry \cite{abriata2024nobel}), image classification \cite{he2016deep}, and drug/material discovery \cite{zhavoronkov2019deep}. The emerging field of \textit{data-driven dynamical systems} seeks to combine ML with time-series data to uncover underlying structure and principles without requiring explicit models \cite{kevrekidis2003equation,mezic2005spectral,schmidt2009distilling,berry2015nonparametric,brunton2016discovering,kutz2016dynamic,giannakis2019data,brunton2020machine,giannakis2021learning,bury2021deep,das2021reproducing}.

Current ML techniques often struggle to converge or generalize, with limited guarantees of their trustworthiness. This hinders their effectiveness in critical applications and poses a central challenge: \textit{When can system behavior be learned reliably from data, and when is such learning impossible?} 
We address this with:

$\bullet$ \textbf{Adversarial dynamical systems:}
We present examples in data-driven dynamical systems for which no sequence of learning algorithms—probabilistic or otherwise—can solve, even with unlimited data.
By carefully altering a system's behavior in a way that respects both its structure and the data, we design adversarial systems that block reliable learning. These are not rare edge cases: success is fundamentally limited to 50\% and they arise in well-studied classes. Even for smooth systems on simple low-dimensional surfaces, tasks such as learning finite-dimensional linearized representations (e.g., via autoencoders) remain unsolvable.

This parallels adversarial attacks in machine learning, where small perturbations expose vulnerabilities and drive the development of robust methods. Similarly, our adversarial systems reveal fundamental learning limits across broad classes of dynamical systems and offer principles for trustworthy algorithms. They may also shed light on phenomena such as hallucinations in large language models (LLMs), as discussed below.

$\bullet$ \textbf{Optimal algorithms with learning guarantees:} Building on insights from these adversarial systems (see challenges \textbf{(C1)} and \textbf{(C2)} below), we develop provably optimal algorithms with guaranteed convergence and error bounds under broad conditions. Unlike traditional approaches that rely on sequences of algorithms, our methods reach fundamental limits and enable reliable extrapolation, crucial for trustworthy AI. Our models are also trained on CPUs at a fraction of the cost of deep learning approaches, far exceeding the scales at which recent claims of efficiency have been made \cite{gibneychina}.

$\bullet$ \textbf{A universal framework for classification:} We establish a rigorous yet practical mathematical foundation that clarifies when and why learning succeeds or fails. Matching lower and upper bounds on difficulty reveal the core challenges and offer a comprehensive classification of problem complexity.

\vspace{1mm}

These results offer a unified, rigorous, and practical basis for understanding when data-driven models can or cannot succeed, advancing the broader goal of trustworthy ML. They are applicable across fields from climate science and neuroscience to engineering and control (see Discussion), where reliable prediction and mechanistic insight are essential.

\begin{SCfigure*}[\sidecaptionrelwidth]
\centering
\includegraphics[width=11.7cm,trim={5mm 0 10mm 3mm},clip]{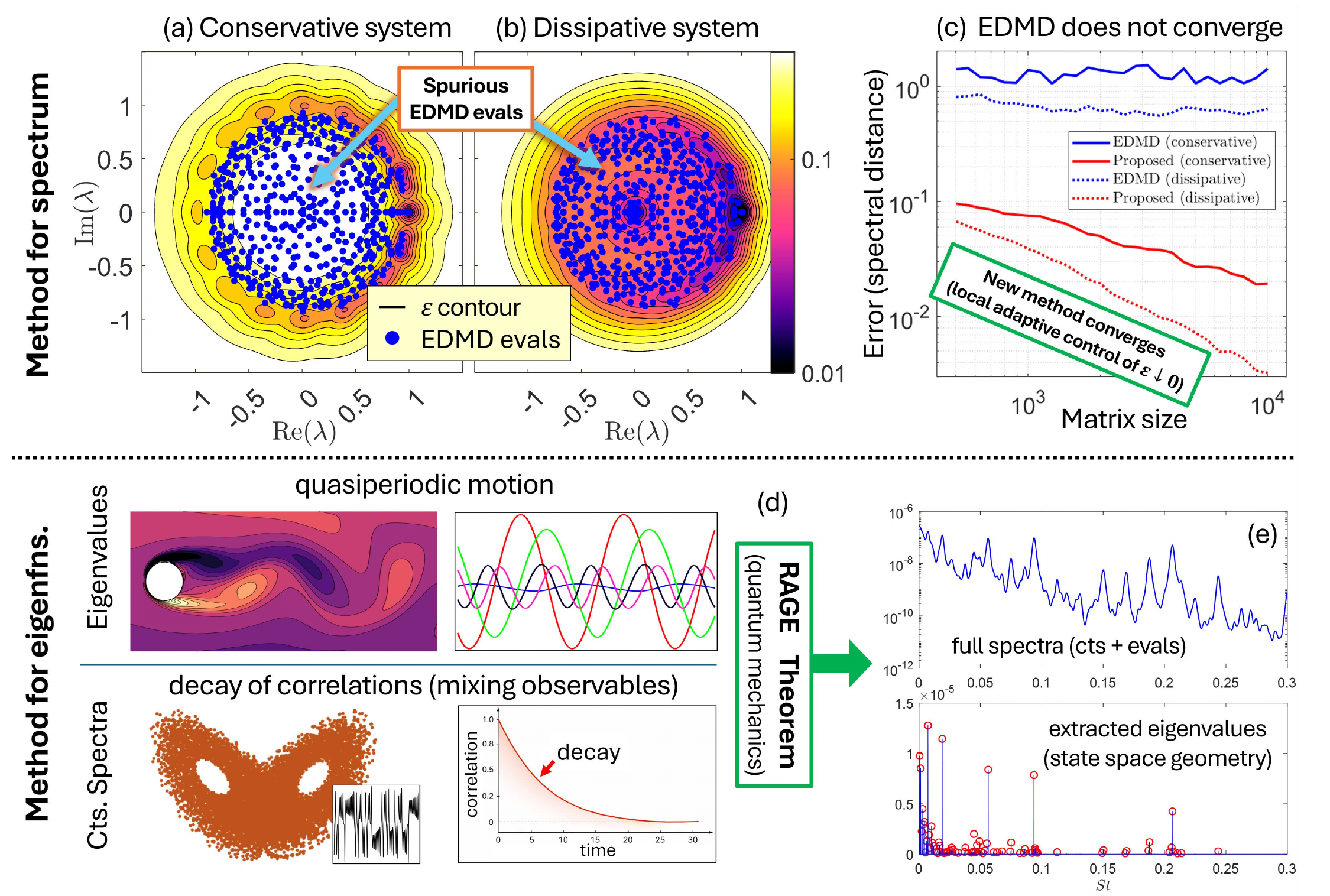}\vspace{-3mm}
\caption{\scriptsize\textbf{Convergent general-purpose methods for Koopman learning.}
\textit{Top:} The method (Supplementary Algorithm 1) is benchmarked on the Duffing oscillator in conservative (a) and dissipative (b) regimes. The state-of-the-art EDMD method fails to converge (c), generating spurious eigenvalues (blue). In contrast, our approach converges reliably by adaptively computing temporally coherent observables $\phi_\varepsilon$ and associated scalars $\lambda$ that approximately satisfy $\phi_\varepsilon(x_n)=\lambda^n\phi_\varepsilon(x_0)$ ($x_n$ is the state of the system at time $n$) up to a controlled tolerance $\varepsilon$. The contours of $\varepsilon$ show where approximate, near-eigenvalue behavior occurs, providing a more robust picture of dynamics. The tolerance $\varepsilon$ is adjusted locally based on data availability and the number of observables used, yielding certified error bounds. The reported error corresponds to the maximum $\varepsilon$ across outputs (averaged over 10 random realizations) and provably bounds the distance to the true spectrum (see Equation (16) of the Supplementary Information).
\textit{Bottom:}
The method (Supplementary Algorithm 6) is applied to the $Re=19000$ cavity flow to extract Koopman eigenvalues in a regime exhibiting signatures of continuous spectrum. Observables are separated according to their long-time behavior (d): quasiperiodic components correspond to discrete eigenvalues, whereas mixing components spread across the continuous spectrum and progressively lose finite-dimensional representation. This distinction is reflected in the full spectral distribution and in the extracted eigenvalues shown in (e).}\label{fig:EDMD_not_converge}
\end{SCfigure*}

We apply this framework to Koopman operators, a major research focus that addresses nonlinearity by acting on an \textit{infinite-dimensional} space of measurements rather than the system's state. Introduced nearly a century ago by Koopman and von Neumann \cite{koopman1931hamiltonian,koopman1932dynamical}, Koopman operators now play a central role in data-driven dynamical systems \cite{mezic2005spectral,budivsic2012applied,giannakis2019data,brunton2021modern,mezicAMS,colbrook2023multiverse}. Their spectral properties (e.g., eigenfunctions and eigenvalues) decompose complex behavior into simpler components like trends, oscillations, or decay, allowing for the use of linear methods in prediction, estimation, and control. This enables explainable, robust, and cost-efficient ML. Notable successes include robot control \cite{haggerty2023control}, climate analysis \cite{froyland2021spectral}, neural network training \cite{orvieto2023resurrecting}, disease modeling \cite{proctor2015discovering}, brain analysis \cite{brunton2016extracting}, non-autonomous systems \cite{froyland2024revealing}, and interpretable AI \cite{lusch2018deep}.

However, Koopman theory faces substantial practical challenges. Spectral approximation in infinite-dimensional settings is frequently non-convergent, even with perfect data \cite{kutz2016dynamic,korda2018convergence,colbrook2023multiverse}. The most widely used approach, Dynamic Mode Decomposition (DMD) \cite{schmid2010dynamic}, and its variants, including extended DMD (EDMD) \cite{williams2015data}, often generate spurious eigenvalues and fail to converge (\cref{fig:EDMD_not_converge}, top). Koopman operators are typically non-self-adjoint (non-Hermitian) and may possess continuous rather than purely discrete spectra \cite{mezic2005spectral}, limiting the applicability of classical spectral approximation techniques \cite{babuvska1991eigenvalue,chatelin2011spectral}. Although recent advances \cite{korda2020data,giannakis2021learning,das2021reproducing,colbrook2021rigorousKoop} address specific spectral properties under certain assumptions, they rely on multiple limiting procedures and do not provide a unified approach for convergence in all cases.

Koopman theory thus provides an ideal setting to explore our central question. We present a complete treatment of sharp algorithms and computational difficulty for Koopman spectra. In particular, we introduce general-purpose, provably convergent methods for learning Koopman spectral properties (Supplementary Algorithms 1–5, top panel of \cref{fig:EDMD_not_converge}) that avoid spurious eigenvalues through explicit local minimization of spectral errors. We further adapt tools from quantum mechanics to separate spectral components with distinct physical signatures (Supplementary Algorithms 6–7, bottom panel of \cref{fig:EDMD_not_converge}). These methods perform well across low- and high-dimensional systems, including challenging cases where the system's behavior spans a continuous range of frequencies, rather than discrete periodic behavior. This includes practical applications such as forecasting Arctic sea ice (\cref{fig:sea_ice_evals,fig:sea_ice_decay,fig:sea_ice_forecast2}).

This final example is motivated by Arctic amplification, where near-surface Arctic temperatures are rising faster than the global average \cite{rantanen2022arctic}. Sea ice loss has major impacts on polar bear habitats, Indigenous communities, shipping, and the Atlantic Meridional Overturning Circulation (AMOC). 
Concurrently, extreme weather events (e.g., wildfires, floods, heatwaves, and severe winters) have intensified in recent decades, affecting billions worldwide. The link between Arctic sea ice loss and Northern Hemisphere extreme weather remains an active area of research and debate \cite{cohen2020divergent,cohen2021linking,coumou2018influence}. While regional effects appear likely, identifying geographically significant regions and their influence is particularly challenging. Forecasting Arctic sea ice beyond two months remains a major challenge \cite{wayand2019year}.

Our algorithms uncover a family of hidden Koopman modes linked to sea ice decline and identify the associated geographic regions with error bounds. These patterns support accurate long-term predictions and reveal how different parts of the system are connected. By building data-driven decompositions from these and other validated modes, we achieve state-of-the-art Arctic sea ice forecasts (\cref{fig:sea_ice_forecast2}). Notably, such modes can influence tipping behavior, including patterns related to the AMOC \cite{PhysRevFluids.9.123801,lohmann2024multistability}.

The adversarial dynamics we construct may also shed light on hallucinations in LLMs. LLMs generate trajectories (sentences) over a state space (words) via one-step-ahead prediction, a process linked to Koopman operators \cite{mezic2023operator,zekri2024large,orvieto2023resurrecting}. Our adversarial systems often have Koopman operators with a continuous spread of frequencies, characteristic of chaotic dynamics. This enables short-term accuracy but causes long-term unpredictability due to sensitivity to initial conditions, mirroring how small prompt changes in LLMs can cause divergent outputs. Our results may therefore help explain inherent limitations of autoregressive architectures in maintaining long-term accuracy.

\subsubsection*{Multiple limits: Bad, good, or insufficient data?}

Learning systems from data requires two elements: quantities we can measure using sensors and their corresponding time series data. In ML, one often studies sequences of algorithms indexed by $n$, where $n$ might represent the size of the dataset or the complexity of the model (for example, the width or depth of a neural network). A simple illustration is estimating the average energy $E$ of an oscillating system based on measurements of its instantaneous energy $e(n)$ at each time step $n$. If the system is ergodic (meaning that averages along a single long-term trajectory reflect the behavior of the entire system), then $E$ can be estimated iteratively by:
\begin{equation}
\label{ergodic_energy_example}
\Gamma_{n+1}=\frac{e(1)+\cdots+ e(n+1)}{n+1}=\frac{n}{n+1}\Gamma_n+\frac{e(n+1)}{n+1}.
\end{equation}
This running average $\Gamma_n$ becomes more accurate as more data is collected. A classical result known as Birkhoff's ergodic theorem \cite{mane2012ergodic} guarantees that $\Gamma_n$ converges to the true average energy $E$ as $n\rightarrow\infty$, i.e., as we collect more data.

Learning often assumes that increasing the amount of data or model complexity will improve performance. However, our results show that for many key problems in dynamical systems, this assumption is false: no algorithm $\smash{\Gamma_n}$ can succeed by taking a single limit, regardless of how that limit is defined (e.g., sample size, model complexity, simulation length).

Instead, some problems only become solvable when multiple data limits are applied in a specific order. For example, computing long-term structures in a system, such as ergodic partitions \cite{mezic1994geometrical}, requires not just averaging over time (as in \cref{ergodic_energy_example}) but also subsequently computing the averages for a growing number of measured quantities (instead of just the energy $e$). These two types of data increase, longer observations and richer measurements, typically cannot be combined or chosen adaptively from the data itself. This leads to a hierarchy of difficulty, where each level corresponds to the number and type of data limits needed for reliable learning.

\subsubsection*{Dynamical setup}

A dynamical system describes how a system changes over time. Suppose we track a collection of variables, denoted $x$, that represent the system's state at a given moment. The set of all possible states is called the state space, written as $\mathcal{X}$, so that $x\in\mathcal{X}$. The system evolves in discrete time steps, meaning the next state $x_{n+1}$ depends on the current state $x_n$ via a rule:
\begin{equation}\setlength\abovedisplayskip{5pt}\setlength\belowdisplayskip{6pt}
x_{n+1} = F(x_n), \qquad n= 0,1,2,\ldots.
\end{equation}
Here, $F:\mathcal{X}\rightarrow\mathcal{X}$ is an \textit{unknown} function that may come from sampling a continuous-time process.

To study a system, we measure its physical properties like temperature or velocity, represented by functions $g:\mathcal{X}\rightarrow\mathbb{C}$, called ``observables''. One should think of $g(x)$ as a quantity we can measure from state $x$. The \textit{Koopman operator}, $\mathcal{K}_F$ (or $\mathcal{K}$), captures how these measurements evolve. It acts on an observable $g$ by composing it with the system's evolution:
\begin{equation}\setlength\abovedisplayskip{6pt}\setlength\belowdisplayskip{6pt}
[\mathcal{K}g](x) = [g\circ F](x)=g(F(x)),\quad x\in\mathcal{X}.
\label{eq:KoopmanOperator} 
\end{equation}
This means $\mathcal{K}g$ gives the value of the observable one step into the future: $[\mathcal{K}g](x_n)= g(x_{n+1})$. The key property of $\mathcal{K}$ is \textit{linearity}: for any observables $f$ and $g$ and scalars $\alpha$ and $\beta$, $\mathcal{K}(\alpha f+\beta g)=\alpha\mathcal{K}(f)+\beta \mathcal{K}(g)$. Linearity is powerful since it allows us to analyze the system through spectral properties of $\mathcal{K}$ (e.g., eigenvalues and eigenfunctions). The trade-off for this global linearization is that $\mathcal{K}$ acts on an \textit{infinite-dimensional} space of observables. One can think of $\mathcal{K}$ as an infinite matrix, corresponding to the infinite number of observables $g$.

The goal is to learn the spectral properties of the Koopman operator from \textit{snapshot data}, discrete sample pairs of the system's behavior:
\begin{equation}\setlength\abovedisplayskip{5pt}\setlength\belowdisplayskip{5pt}
\label{eq:snapshot_data}
\left\{\left(x^{(m)},y^{(m)}=F(x^{(m)})\right):m=1,\ldots,M\right\}.
\end{equation}
Here, each $y^{(m)}$ represents the state one time step ahead of $x^{(m)}$. Such data can arise from experiments or simulations, and observations of either long or short trajectories. We shall see examples of each of these throughout the paper.

\subsubsection*{Separation of variables \& spectra in nonlinear systems}
Spectral properties of Koopman operators contain valuable information about the system.
For example, a complex number $\lambda\in\mathbb{C}$ is called an almost eigenvalue if, for a tolerance $\varepsilon>0$, there exists a normalized observable $\phi_\varepsilon$ with $\|\mathcal{K} \phi_\varepsilon-\lambda \phi_\varepsilon\|\leq \varepsilon$. (Here, $\|\cdot\|$ measures an observable's energy.) These observables, called \textit{approximate eigenfunctions}, are physically relevant because, assuming $\|\mathcal{K}\|\leq 1$, they exhibit approximate temporal coherence \cite{giannakis2019data,valva2023consistent}:
\begin{equation}\setlength\abovedisplayskip{5pt}\setlength\belowdisplayskip{5pt}
\label{eq:approx_coherency}
\phi_\varepsilon(x_n)=\lambda^n \phi_\varepsilon(x_0)+ \mathcal{O}(n\varepsilon)\text{ as }\varepsilon\downarrow 0\text{ for }n=1,2,\ldots.
\end{equation}
That is, $\lambda$ approximately describes the oscillatory behavior and decay (or growth) of the measurement $\phi_\varepsilon(x)$ over time through its powers $\lambda^n$. Smaller values of $\varepsilon$ correspond to longer timescales where this approximation is valid. For $\varepsilon = 0$, $\phi_\varepsilon$ becomes an exact eigenfunction of $\mathcal{K}$ with eigenvalue $\lambda$.

Approximate eigenfunctions also encode key dynamical features of the system, such as the global stability of equilibria. The contours (or level sets) of these functions highlight key structures in the system's dynamics, such as regions that behave independently over long times, pathways along which the system evolves, and surfaces where states settle into long-term behavior at the same rate \cite{budivsic2012applied,mauroy2013isostables}. The \textit{approximate point spectrum} of $\mathcal{K}$, denoted $\spec_{\mathrm{ap}}(\mathcal{K})$, is the set of all scalars $\lambda$ for which $\varepsilon$ can be made arbitrarily small, and forms the most fundamental spectral property of Koopman operators.

Just as a matrix is diagonalized by its eigenvalues and eigenvectors, a nonlinear system can be ``diagonalized'' by its Koopman spectra. Koopman eigenfunctions act as fundamental components of behavior, revealing persistent patterns (like oscillations or trends) in the system. For a vector of observations ${\bf g} \in \mathbb{C}^N$, the Koopman mode associated with an eigenvalue is the projection of ${\bf g}$ onto the corresponding eigenspace. Under suitable conditions, this yields a spectral expansion for the time evolution of the observables \cite{mezic2005spectral}:
\begin{equation}\setlength\abovedisplayskip{4pt}\setlength\belowdisplayskip{4pt}
\label{KMD_example}
{\bf g}(x_n)=\sum_{j=1}^\infty \lambda_j^n \phi_j(x_0) {\bf g}_j.
\end{equation}
Here, ${\bf g}_j\in\mathbb{C}^N$ is the $j$th Koopman mode, associated with eigenvalue $\lambda_j$ and eigenfunction $\phi_j$, which one may think of as an expansion coefficient. This decomposition is conceptually and operationally similar to separation of variables: the eigenvalues describe time evolution through the powers $\lambda_j^n$, while the Koopman modes capture how each pattern is expressed in space and the regions where this dynamical behavior occurs.

When the observable ${\bf g}$ is real-valued, the eigenvalues, eigenfunctions, and Koopman modes in \cref{KMD_example} appear in complex-conjugate pairs. To illustrate how this affects time evolution, consider one such pair:
$
\lambda_1^n \phi_1(x_0) {\bf g}_1+
{\overline{\lambda_1}}^n \overline{\phi_1(x_0)} \overline{{\bf g}_1}.
$
Writing
$
\lambda_1= re^{i\theta}$ and $\phi_1(x_0) {\bf g}_1= \mathbf{R}e^{i\mathbf{\Theta}}, 
$
the pair becomes
\begin{equation}\setlength\abovedisplayskip{5pt}\setlength\belowdisplayskip{5pt}
\lambda_1^n \phi_1(x_0) {\bf g}_1+
{\overline{\lambda_1}}^n \overline{\phi_1(x_0)} \overline{{\bf g}_1}=
2\mathbf{R} r^n \cos(n\theta+\mathbf{\Theta}).
\end{equation}
This expression reveals key dynamical features: $r^n$ controls exponential growth ($r > 1$), decay ($r < 1$), or neutral evolution ($r = 1$); the angle $\theta$ sets the oscillation frequency; and the Koopman mode (e.g., through $\mathbf{R}$) encodes spatial structure.

Methods for learning $\spec_{\mathrm{ap}}(\mathcal{K})$ and spectral expansions from data face significant challenges, including spurious eigenvalues (\cref{fig:EDMD_not_converge}), missing critical spectral components, and numerical instabilities. Recent advances mitigate these issues \cite{colbrook2021rigorousKoop}, but there is a critical need to develop a deeper theoretical understanding of the conditions under which such spectral computations are feasible. Equally important is identifying scenarios where these computations are fundamentally impossible. We shall see that addressing these questions guides the development of robust and reliable methods.

\subsection*{Results}

\noindent
We first present three illustrative examples that demonstrate the reliability, advantages, and broad applicability of our general-purpose learning algorithms: (i) a classic oscillator where previous methods diverge but ours converges, (ii) a fluid flow with continuous spectrum where we extract meaningful eigenvalues, and (iii) an Arctic sea ice dataset where we uncover physically interpretable modes with error bounds and achieve state-of-the-art forecasts. We then examine the underlying theoretical foundations, showing that our algorithms optimally match fundamental limits revealed by adversarial dynamical systems. Indeed, analyzing the reasons behind these fundamental limitations enabled us to pinpoint the essential algorithmic properties required for reliable convergence. Finally, we provide a unified classification of the complexity of these learning problems, clarifying precisely when learning is possible and when it is fundamentally impossible.

\subsubsection*{Overcoming lack of convergence in current methods}
We begin with a simple low-dimensional system. \cref{fig:EDMD_not_converge} (top) examines the classical Duffing oscillator in two distinct regimes: the undamped case (conservative, panel (a)) and the damped case (dissipative, panel (b)).

EDMD, widely regarded as the state-of-the-art, constructs a finite $N\times N$ matrix approximation of the Koopman operator by projecting onto $N$ trial functions (a `dictionary' of observables). We use a common and effective choice of dictionary: radial basis functions centered via $k$-means clustering \cite{williams2015data}. In contrast to computing eigenvalues of an approximation, our method (Supplementary Algorithm 1) searches for approximate eigenfunctions where $\varepsilon$ in \cref{eq:approx_coherency} is locally minimized, with guaranteed convergence to the spectrum as $N$ increases. The number of observables $N$ grows with the number of snapshots $M$. Details of the experimental setup are in the Methods.

Panel (c) shows that EDMD does not converge as the matrix dimension $N$ increases, instead generating numerous spurious eigenvalues due to truncation from the infinite-dimensional Koopman operator to a finite observable subspace. In contrast, our algorithm (using the same data and dictionary) converges reliably across dynamical regimes. Importantly, even without access to ground truth, it provides a principled error bound and enables validation of the chosen observable dictionary, a critical practical consideration given that computations necessarily involve finitely many observables. In the Methods, we further demonstrate the same behavior across a broad range of examples.

\subsubsection*{Eigenvalue extraction in a chaotic flow}

\cref{fig:EDMD_not_converge} (bottom) shows the application of our algorithms to $Re=19000$ cavity flow, a regime exhibiting signatures of both discrete and continuous Koopman spectral components. In such cases, the Koopman representation involves an integral over a continuous spectrum in addition to the discrete sum in \cref{KMD_example} \cite{mezic2005spectral,colbrook2024rigged}. Discrete Koopman eigenvalues typically signal quasiperiodic motion, whereas broadband spectral content is indicative of chaotic dynamics \cite{swinney1978transition}. The combined spectral structure reflects the underlying geometry of the attractor in state space. Systems displaying both discrete and continuous spectral components are often described as skew-periodic \cite{broer1993mixed}, with one component evolving periodically and the other exhibiting phase-modulated chaotic behavior.

Our algorithm (see Methods) successfully extracts eigenvalues and separates spectral components using a two-limit approach: adaptively increasing a time lag for autocorrelations to detect spectrally localized observables and expanding projections onto increasing finite-dimensional subspaces. This is grounded in linking quasiperiodic dynamics and the corresponding eigenvalues (the so-called RAGE theorem \cite{fillman2017purely}, see \cref{RAGE_statement}). This foundation ensures that the separation of the continuous spectrum and eigenvalues is both mathematically rigorous and practically optimal. 

Physically, the evolution of Koopman spectra with increasing Reynolds number illustrates a transition, in which a chaotic state arises after one or two bifurcations (sudden qualitative changes in the system's behavior) from a stable steady flow. The Koopman spectrum offers a powerful tool for analyzing bifurcations, quantifying energy in both quasiperiodic and continuous components, and links to the geometry of the state space and the flow domain (Supplementary Figures 6--8).

\subsubsection*{Physical modes for Arctic sea ice concentration}

\begin{SCfigure*}[\sidecaptionrelwidth]
\centering
\includegraphics[width=13.3cm]{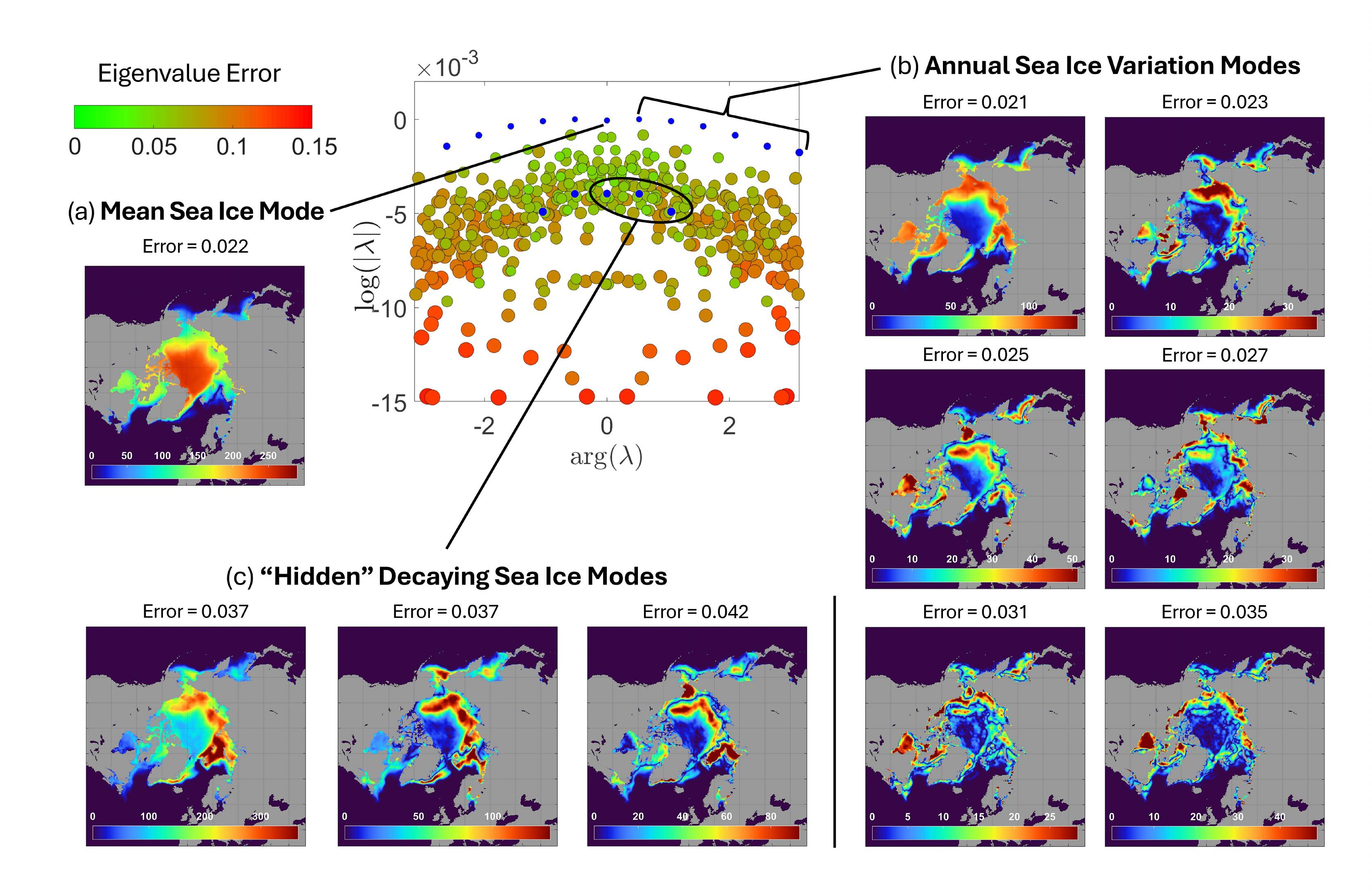}
\caption{\scriptsize\textbf{Verified eigenvalues and Koopman modes of Arctic sea ice (1979--2021) and detection of hidden modes for long-time forecasts.} For each EDMD eigenvalue, our proposed method computes an associated error bound (colorbar). The displayed size of each eigenvalue is proportional to this error. While many EDMD eigenvalues are spurious, we identify 17 reliable ones (shown in blue) with small associated errors. The Koopman modes (the ${\bf g}_j$ in \cref{KMD_example}, where we plot their absolute value) are categorized into three groups: (a) the mean sea ice concentration, corresponding to a stationary mode with eigenvalue \(\lambda = 1\); (b) annual variation in sea ice concentration, with \(\lambda = \exp(m\pi i / 6)\) for \(m = \pm 1, \pm 2, \ldots, \pm 5, 6\), representing periodic variation across the months; and (c) ``hidden'' decaying modes with \(|\lambda| < 1\) (revealed by our error bounds), representing long-term sea ice loss. The spatial structure of each mode indicates the geographic regions where these behaviors occur, and the corresponding Koopman eigenfunction time series shows a clear trend over the analysis interval. The hidden modes with nonzero $\mathrm{arg}(\lambda)$ can be interpreted as seasonal patterns (group (b)) modulated by the decaying mode with zero argument, providing insight into evolving Arctic sea ice dynamics.}\label{fig:sea_ice_evals}
\end{SCfigure*}

\begin{figure}
\centering
\includegraphics[width=0.93\linewidth,trim=4mm 4mm 4mm 4mm,clip]{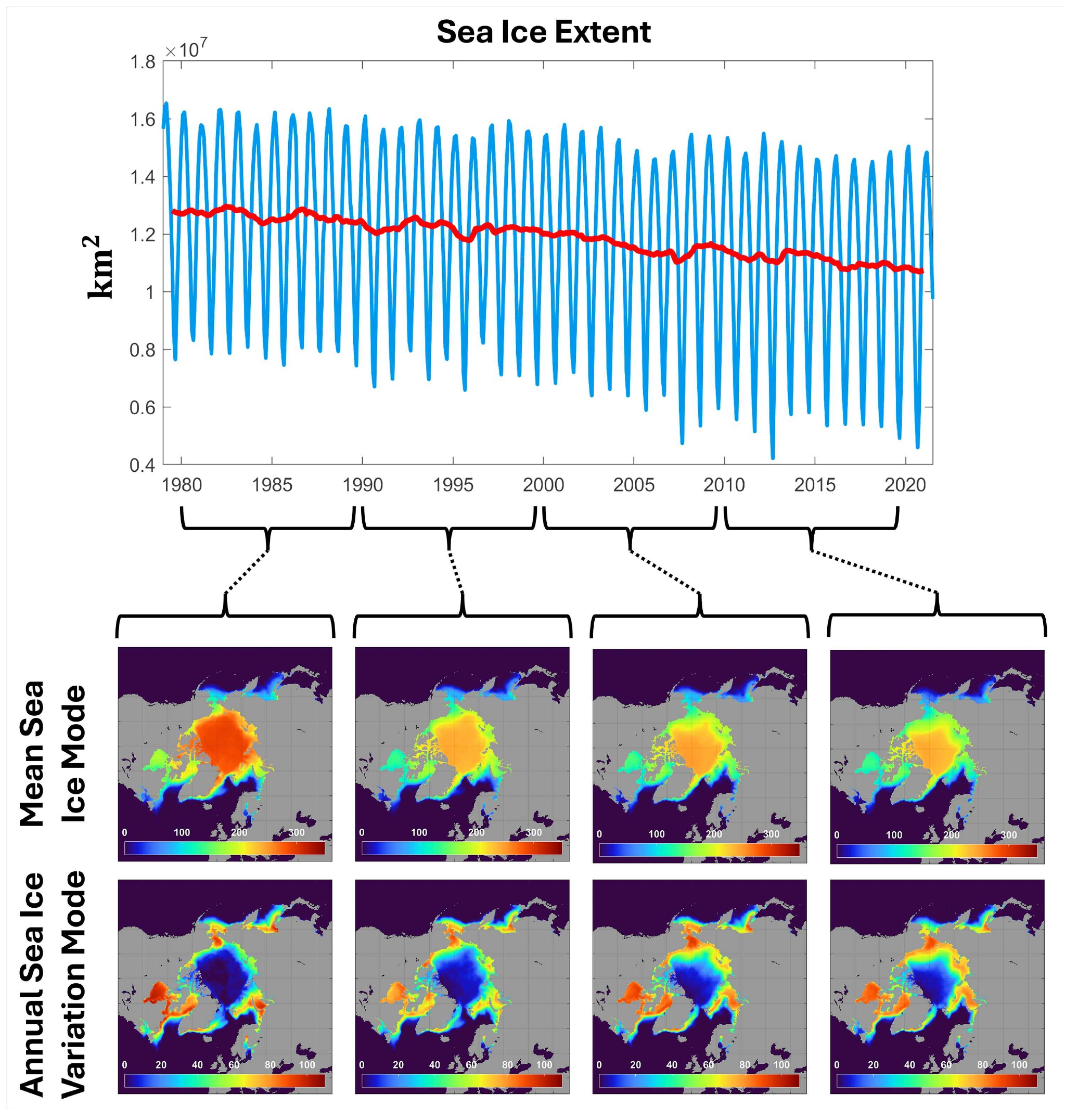}
\caption{\scriptsize\textbf{Sea ice decay and time-windowed Koopman modes.}
Top: Sea ice extent over the past several decades. The red curve shows the moving 12-month mean. Bottom: The absolute value of Koopman modes corresponding to the mean sea ice and annual variations over ten-year windows. The maximum error (relative residual) of these modes is 0.048. Mean and annual modes reveal declining sea ice, reduced winter extent, and amplified seasonal contrast in marginal seas, consistent with a slow decaying mode (\cref{fig:sea_ice_evals}).}
\label{fig:sea_ice_decay}
\end{figure}

We consider monthly Arctic sea ice concentration satellite data from 1979--2021. The data are defined on a $432 \times 432$ grid of 625 km$^2$ cells. Details on data collection and processing are provided in the Methods.
As discussed in the introduction, Arctic sea ice is a critical component of the climate system, notorious for its complex dynamics. In this setting, Koopman eigenvalues reveal the timescales of changes in sea ice cover, while the corresponding modes indicate the spatial pattern of these changes. Notably, Koopman methods require no model: measurements of the concentration suffice to compute the eigenvalues and modes.

We first analyze the entire dataset and approximate Koopman eigenvalues using EDMD (\cref{fig:sea_ice_evals}). Similar to \cref{fig:EDMD_not_converge}, our method provides error bounds for eigenpairs, shown in the plot. While many EDMD eigenvalues are spurious, several (shown in blue) have small errors and correspond to key dynamics: the annual mean sea ice ($\lambda=1$); yearly growth-melt cycle captured by a fundamental monthly mode ($\lambda\approx \exp(\pi i/6)$); and decaying modes ($|\lambda|<1$). The small errors of the hidden modes indicate strong coherency and forecasting power (see \cref{eq:approx_coherency}).

The decaying modes, ``hidden'' behind a sea of spurious modes but nonetheless revealed by our error bounds, are particularly interesting as they reflect dissipative dynamics. The corresponding Koopman eigenfunction time series (computed separately over 1980--1999 and 2000--2019 to assess robustness) suggests a connection to long-term sea-ice decline associated with climate change (Supplementary Figure 9). This connection is likely nuanced: dissipative Koopman eigenfunctions exhibit reduction (as well as oscillation) as time increases, whereas observed sea-ice loss rates over recent decades are known to be non-monotonic. Since the full sea-ice state is expressed as a sum of modes, the near-monotonic behavior of $\phi_{\mathrm{decay}}$ does not contradict non-monotonic trend patterns in \cref{fig:sea_ice_decay}.

\begin{figure*}[th!]
\centering
\includegraphics[width=0.24\linewidth]{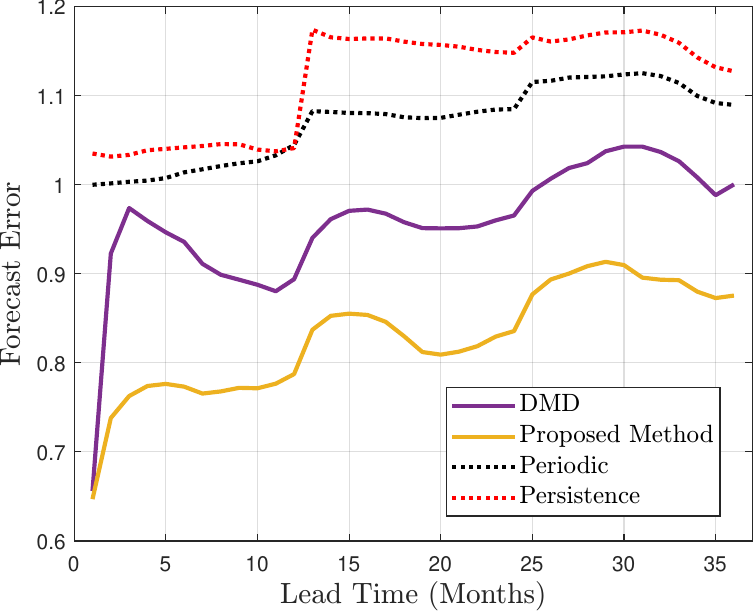}\hfill
\includegraphics[width=0.25\linewidth]{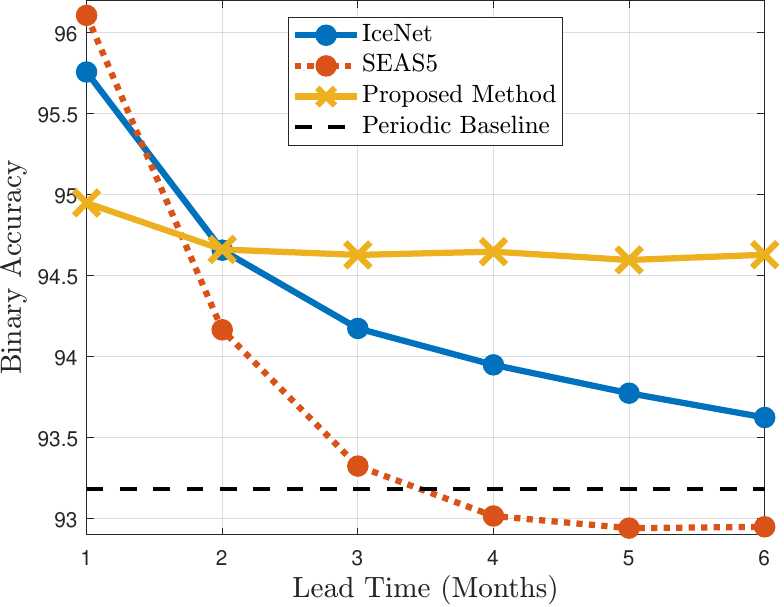}\hfill
\includegraphics[width=0.23\linewidth]{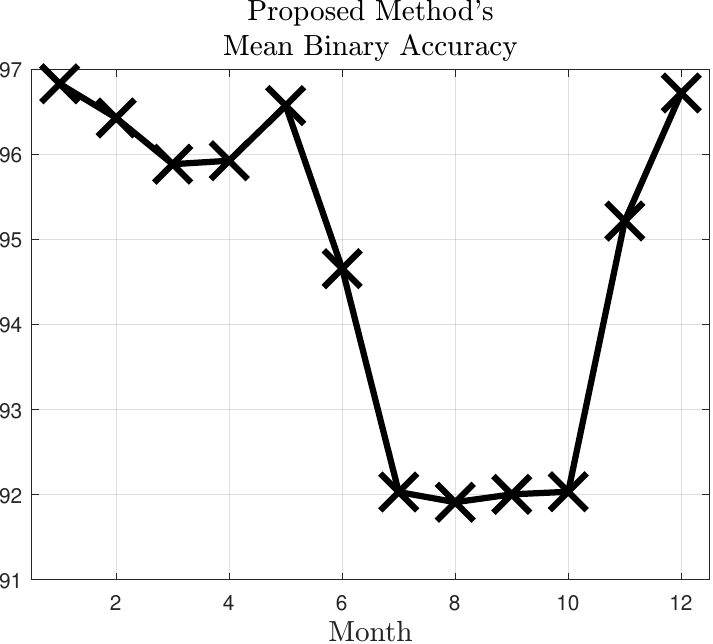}\hfill
\includegraphics[width=0.23\linewidth]{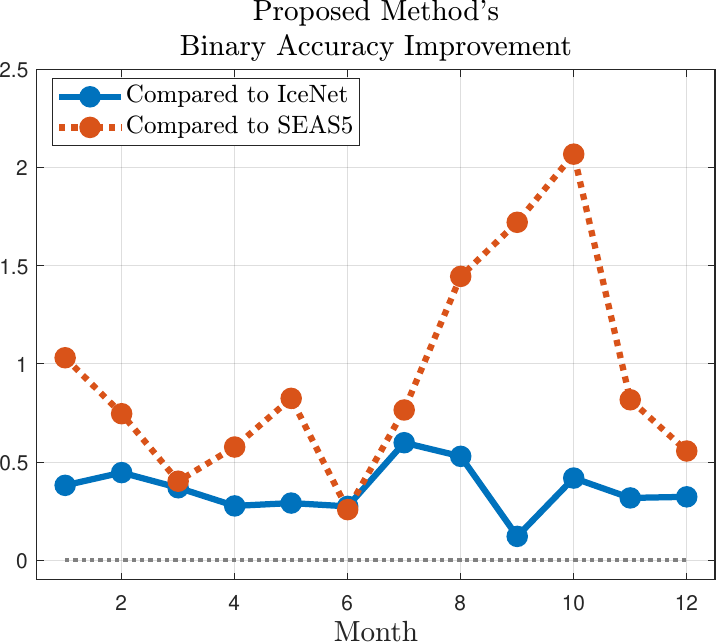}
\caption{\scriptsize\textbf{Comparison with benchmarks.}
Left: Forecast error for sea ice concentration at every grid point. Relative error of forecasted anomalies (relative to the periodic monthly climatological variance). We consider three-year forecasts initialized at each month from 2005 to 2015 and plot the average error for each lead time. The proposed method consistently outperforms DMD.
Middle left: Mean binary accuracy (see text for definition) over the test years 2012--2020, shown for IceNet, SEAS5, and our proposed method that avoids spurious Koopman eigenvalues. Our proposed method achieves better accuracy for lead times greater than one month, with very little increase in errors at larger lead times. Moreover, this is achieved using orders of magnitude fewer trainable parameters and substantially less computational cost.
Middle right: Mean binary accuracy of our proposed method over all lead times and years.
Right: Improvements in binary accuracy over IceNet and SEAS5.}
\label{fig:sea_ice_forecast2}
\end{figure*}

Hidden modes with nonzero $\mathrm{arg}(\lambda)$ can be interpreted as seasonal patterns modulated by this long-term trend. In other words, these represent seasonal oscillations that gradually decay over decades. Similar behavior has also been reported for Koopman eigenfunctions computed from sea-surface temperature data \cite{froyland2024revealing}. Our modes capture both the oscillation and the long-term decay. If $\phi_{\mathrm{decay}}$ is the decaying eigenfunction with zero complex argument and $\phi_{\mathrm{var}}$ an annual variation mode with argument $\pi/6$, then other decaying eigenfunctions are approximately given by $\phi_{\mathrm{decay}} \times [\phi_{\mathrm{var}}]^j$. This structure arises from the multiplicative property of Koopman eigenfunctions: if $\mathcal{K}g_1 = \lambda_1 g_1$ and $\mathcal{K}g_2 = \lambda_2 g_2$, then $\mathcal{K}(g_1 g_2) = \lambda_1 \lambda_2 g_1 g_2$, which follows from \cref{eq:KoopmanOperator}. For example, the canonical correlations between these decaying modes and the underlying modulated monthly dynamics are $0.9807$ and $0.9798$ for modes with arguments $\pi/6$ and $\pi/3$, respectively. In the climate science literature, such seasonally modulated product modes are sometimes referred to as ``combination modes'' \cite{stuecker2015nino,froyland2021spectral}.

The mean decay time, $-1/\log(|\lambda|)$, for these modes is 233 months, consistent with decay times observed in the Antarctic region \cite{hogg2020exponentially}. Notably, this decay rate was not observed in the Arctic region using DMD methods in \cite{hogg2020exponentially}, which we attribute to the challenges of extracting reliable eigenpairs in the absence of error bounds. Several studies indicate nonlinear trends in the decline of sea ice \cite{stroeve2012arctic}, which further advocates the use of Koopman operator techniques to disentangle the complex nonlinear dynamics. 

We compute Koopman modes (${\bf g}_j$ in \cref{KMD_example} where ${\bf g}$ is the vector of sea ice concentrations), which capture spatial and temporal patterns in sea ice concentration not easily discernible by conventional methods. Although the modes share units with the input data, they may take values above 100\% due to the non-orthogonal expansion in \cref{KMD_example}. Each mode highlights geographic regions where sea ice exhibits oscillatory, growing, or decaying behavior, as indicated by its corresponding eigenvalue. For example, modes with one-year oscillations reflect the seasonal cycle, eigenvalues near one capture the mean state, and multi-year or slowly varying modes indicate long-term trends, with the associated modes highlighting the regions where these changes occur.

The hidden decaying mode is concentrated in the Barents and Kara Seas, implying that the decrease in sea ice concentration over these timescales is localized to those regions. This is consistent with observations in \cite{hogg2020exponentially}. Indeed, there are links between sea ice reduction in these regions and extreme weather, such as severe winters in central Eurasia \cite{mori2014robust}. Importantly, our method provides error estimates, giving confidence that these identified modes are real and not artifacts (unlike many spurious modes given by EDMD).

These hidden modes also capture seasonally-modulated reemergence of correlations (Supplementary Figure 10), revealing `memory' in the climate system, whereby sea ice anomalies occurring during the growth season reemerge in the following melt season despite a loss of correlation in the intervening winter months \cite{blanchard2011persistence}. For studies of this phenomenon using kernel methods, see \cite{bushuk2017seasonality,bushuk2014reemergence}. Such long-lived modes and their geographic foci are critical, as they could influence large-scale climate patterns (e.g., ocean circulation changes) and have practical implications for shipping routes and climate resilience. Further work is needed to perform data-driven prediction of changes due to tipping points such as the greater mixing between the Barents Sea and North Atlantic \cite{lind2018arctic}. Modes in sea ice can significantly influence the onset or prevention of tipping behavior of AMOC patterns \cite{PhysRevFluids.9.123801,lohmann2024multistability}.

Sea ice decay is commonly assessed by sea ice extent, the area covered by grid cells with sea ice concentration exceeding $15\%$, shown in \cref{fig:sea_ice_decay}. We compute the Koopman modes corresponding to the mean sea ice concentration and the annual variation over ten-year periods (1980–1989, 1990–1999, 2000–2009, 2010–2019) and plot their absolute values. The mean mode reveals a clear decline in overall sea ice and reduced winter extent. Annual modes show a geographic shift in seasonal variability, with marginal seas (particularly the Beaufort, Kara, and nearby coastal regions) exhibiting increased amplitude, indicating stronger seasonal contrast in these regions. The overall decline in concentration suggests a slow decaying mode, consistent with the spectrum in \cref{fig:sea_ice_evals}.

\subsubsection*{Arctic sea ice forecasting}

We now address the problem of forecasting, focusing first on the challenging task of reconstructing sea ice concentration at active grid points. Forecasts are initialized monthly from 2005 to 2015, each with a 3-year horizon. Training data consists of observations from 1979 up to the month preceding initialization, yielding 132 distinct forecast trajectories. To obtain accurate forecasts, we truncate the Koopman mode decomposition using spectral approximations and error bounds computed by our algorithm (\cref{KMD_example2} in Methods). We compare our method to EDMD with delay embedding, which has previously shown strong performance in Arctic sea ice forecasting \cite{hogg2020exponentially}. We also benchmark against two baselines: the monthly climatology at each grid point (referred to as the periodic baseline) and the monthly persistence model. Forecast errors are measured relative to the climatological variance of the periodic baseline (see Methods) and averaged over all forecast trajectories. \Cref{fig:sea_ice_forecast2} (a) shows the results. Our approach outperforms EDMD, particularly at long lead times, consistent with its explicit minimization of $\varepsilon$ in \cref{eq:approx_coherency}.

We next consider a binary classification problem and compare our approach to IceNet \cite{andersson2021seasonal}, a deep learning model that provides state-of-the-art six-month sea ice forecasts. A point is classified as open water if the sea ice concentration is below $15\%$, the standard threshold for defining the ice edge. Following the setup in \cite{andersson2021seasonal}, we evaluate monthly forecasts from 2012 to 2020, with lead times from 1 to 6 months (chosen to align with the lead times for which IceNet produces forecasts). Binary accuracy is defined as the percentage of predicted classes matching the observations. We also compare against SEAS5 \cite{johnson2019seas5}, a leading dynamical model from the European Centre for Medium-Range Weather Forecasts. \cref{fig:sea_ice_forecast2} (b) shows the mean binary accuracy across lead times. Beyond one-month forecasts, the proposed verified Koopman model consistently outperforms both benchmarks, with significantly greater accuracy as lead time increases. This robustness arises from the method's direct minimization of $\varepsilon$, yielding longer coherency and predictive timescales (see \cref{eq:approx_coherency}).

This improvement is achieved at a fraction of the computational cost and with significantly fewer parameters than deep learning approaches. IceNet uses $4.4\times 10^7$ trainable weights and requires over a day to train on a NVIDIA Quadro P4000 GPU, while the Koopman approach relies on an observable space of at most 510 dimensions (260100 parameters) and is trained on a laptop in under a second.

We found no clear relationship between prediction accuracy and El Niño events. \cref{fig:sea_ice_forecast2} (c,d) presents average accuracy by calendar month and the improvement achieved by our method. Accuracy declines markedly during summer, consistent with the well-known ``spring predictability barrier,'' which affects forecasting models owing to the influence of melt-season ice thickness. These months also exhibit the largest performance gains relative to SEAS5. Finally, even as $\varepsilon \downarrow 0$ with increasing data, predictability remains fundamentally constrained by atmospheric chaos and observational noise \cite{olonscheck2019arctic}.

\vspace{3mm}

These practical advantages motivate a return to our central question: \textit{When can system behavior be learned reliably from data, and when is such learning impossible?} To address this, we present our theoretical results, with full proofs provided in the Supplementary Information.

\subsubsection*{Adversarial systems reveal challenges}
We first construct adversarial dynamical systems that expose two fundamental challenges in computing Koopman spectra. These establish \textit{lower bounds} on problem complexity. Even with infinite data, no algorithm can guarantee learning certain system behaviors, providing an explicit demonstration of inherent limits in data-driven dynamical system learning.

While it is sometimes possible to design algorithms that work for a single known system, this does not reflect the goals of general-purpose learning or complexity classification. Instead, our focus is on identifying the minimal assumptions under which an algorithm can reliably learn across a broad class of systems, a core aim in both computational complexity and ML. To illustrate, let $\Omega_{\mathbb{D}}$ be the class of systems that are continuous, measure-preserving\footnote{\scriptsize{A system is measure-preserving if it preserves a volume on the state space $\mathcal{X}$ during its evolution, e.g., an idealized frictionless system. Such systems are ubiquitous, including classical Hamiltonian systems \cite{arnold1989mathematical}, physical systems in equilibrium \cite{hill1986introduction}, and the post-transient behavior of many systems \cite{mezic2005spectral}.}}, and invertible on the unit disk in two dimensions. Let $\Omega_{[0,1]}$ be the class of smooth, invertible $F$ on the interval $[0,1]$ with uniform bounds on their derivatives (not necessarily measure-preserving).

Our results (Theorems 2.6 and 2.9 of the Supplementary Information) show that for either of these classes (denoted collectively as $\Omega$) no deterministic algorithms $\Gamma_n$ exist that, using snapshot data, converge to $\spec_{\mathrm{ap}}(\mathcal{K}_F)$ for all $F$ in the class $\Omega$ as $n\rightarrow\infty$. Furthermore, for any probabilistic learning algorithms, the probability of convergence cannot exceed 50\%. These impossibility results are \textit{universal}, applying to any type of algorithm and regardless of what $n$ represents. Hence, for \textit{any} algorithm, simply increasing the number of data points $M\rightarrow\infty$ will not lead to convergence, as this would correspond to an instance of the sequence $\Gamma_n$. These constructive results reveal fundamental challenges that occur across data-driven dynamical systems:
\begin{itemize}[leftmargin=0.15cm]\setlength{\itemsep}{-1pt}
\item[]\textbf{(C1)}: For systems in $\Omega_{\mathbb{D}}$, the challenge lies in determining when enough data has been collected to approximate the action of $\mathcal{K}$ on a given observable, e.g., by approximating the averages in \cref{ergodic_energy_example}. The convergence rate is problem dependent \cite{kachurovskii1996rate}: no universal rate exists \cite[page 14]{krengel2011ergodic}.
\item[]\textbf{(C2)}: For systems in $\Omega_{[0,1]}$, the difficulty stems from the non-normality of $\mathcal{K}$ (non-orthogonality of its eigenfunctions). This is a well-known challenge in spectral approximation more generally \cite{trefethen2005spectra}, and manifests itself in the Koopman context as the difficulty in distinguishing data corresponding to transient dynamics from post-transient dynamics.
\end{itemize}
Moreover, these challenges:
\begin{itemize}[leftmargin=0.23cm]\setlength{\itemsep}{0pt}
\item cover randomized algorithms, e.g., random trajectory sampling, or training with probability distributions over data, as in stochastic gradient descent and other ML methods;
\item hold whatever the distribution of data;
	\item hold for any type of computer, e.g., digital computation (Turing machines)  or exact arithmetic (BSS machines);
	\item hold even if we consider smoother $F$ and allow our algorithms to sample the derivatives of $F$ as well as \cref{eq:snapshot_data}.
\end{itemize}

The system classes for which we construct adversaries include widely studied examples such as measure-preserving flows and smooth interval exchange maps. The mechanisms can also be embedded in higher dimensions and other state spaces: the impossibility result holds for any class of systems satisfying \textbf{(C1)} or \textbf{(C2)}. These mechanisms are not limited to Koopman spectral estimation and reflect general challenges in data-driven dynamical systems.

\subsubsection*{A universal algorithm with learning guarantees}

Learning from the above results, we now show that Koopman spectra can be computed from trajectory data, provided two key conditions are satisfied. Together with the other algorithms discussed below, these show \textit{upper bounds} on problem complexity. This result also resolves the fundamental open problem of data-driven computation of $\spec_{\mathrm{ap}}(\mathcal{K}_F)$.

To address challenge \textbf{(C1)}, we assume the system is measure-preserving (this can be relaxed in many dissipative cases). To address \textbf{(C2)}, we require some control over the smoothness of the dynamical map $F$. Specifically, we assume $F$ has a known modulus of continuity $\alpha$. This function controls the distance between $F(x)$ and $F(y)$ by the distance between states $x$ and $y$. Although such a function always exists, the impossibility result above shows that without knowledge of $\alpha$, one cannot compute Koopman spectra in a single limit.

Let $\Omega_{\mathcal{X}}^{\alpha,m}$ be the class of systems satisfying both conditions. For such systems, we have developed deterministic learning algorithms $\Gamma_n$ that reliably approximate the system's dynamics using snapshot data (Theorem 2.3 and Algorithm 1 of the Supplementary Information). These algorithms converge to $\spec_{\mathrm{ap}}(\mathcal{K}_F)$ for all $F$ in the class $\Omega_\mathcal{X}^{\alpha,m}$ as $n\rightarrow\infty$. They also provide explicit error bounds that verify the accuracy of the approximation.

Our analysis of \textbf{(C1)} and \textbf{(C2)} directly yields a provably convergent algorithm, in contrast to EDMD, which does not converge in general. The central idea (see Methods) is to use the modulus of continuity to adaptively select the dictionary size $N$ as a function of the available data $M$, within an averaging framework similar to \cref{ergodic_energy_example}. Unlike EDMD, we do not compute eigenvalues of a finite-dimensional matrix. Instead, the adaptive procedure constructs a correlation matrix from which we evaluate the error metric shown in \cref{fig:EDMD_not_converge}. We then identify local minimizers of this metric and, using the measure-preserving structure of the system (or resolvent bounds, e.g., in dissipative settings), relate it to the distance between a point $z\in\mathbb{C}$ and the spectrum. This distance is computed together with an associated approximate eigenfunction satisfying the coherency condition \cref{eq:approx_coherency}, with $\varepsilon$ equal to the estimated spectral distance.

\cref{fig:EDMD_not_converge} demonstrates the convergence of this algorithm applied to the Duffing oscillator (see also the further systems in the Methods). Moreover, the error bounds achieved by the algorithm enabled improved sea ice forecasts in \cref{fig:sea_ice_forecast2}.

\subsubsection*{To infinity and beyond}
Surprisingly, spectral properties can still be learned in the presence of \textbf{(C1)} and \textbf{(C2)} by adjusting the approach. Instead of requiring a single data limit as $n \to \infty$, we consider separate \textit{successive} limits for key parameters. Each is tied to a different type of data, such as increasing dataset size, measurement resolution, or dictionary complexity.

For instance, for EDMD (see Methods) without the modulus of continuity, one must first take the number of data samples $M$ to infinity to approximate correlations and only then the number of observables $N$ to infinity (even then, EDMD may fail to converge\footnote{Convergence of EDMD eigenvalues can be established along a subsequence, provided the weak limit of the associated eigenfunctions is non-zero \cite{korda2018convergence}. Verifying this condition, however, entails multiple limiting procedures in the sense of the SCI described below.}). The convergence rate of time-averaged quantities such as correlations depends sensitively on the system's mixing properties and the regularity of the observable.\footnote{\scriptsize As another example, Bandtlow et al. \cite{bandtlow2023edmd} construct an EDMD-based method that provably converges with $N = N(M)$ for analytic expanding maps, using analyticity and spectral gaps to unify the two limits.} While polynomial rates hold for certain strongly mixing systems \cite{kachurovskii1996rate,mezic2002ergodic}, Birkhoff's theorem guarantees only convergence without any uniform rate and time averages can converge arbitrarily slowly \cite[Chapter 2]{krengel2011ergodic}. This non-uniformity prevents the formulation of a universal convergence rate or a general-purpose $N(M)$ strategy, unless we can control \textbf{(C1)}.

To further understand how assumptions on a system's structure influence our ability to learn from data, we establish results under three scenarios (see Theorem 2.3 of the Supplementary Information):

-- First, let $\Omega_{\mathcal{X}}^{\alpha}$ denote systems whose dynamics have a known smoothness quantified by a modulus of continuity $\alpha$. For these systems, there exist learning algorithms that depend on two parameters, $n_1$ and $n_2$, and reliably use snapshot data to approximate the dynamics. These algorithms converge in two successive data limits (${n_1\rightarrow\infty}$ and then ${n_2\rightarrow\infty}$) to the spectrum $\spec_{\mathrm{ap}}(\mathcal{K}_F)$ for all systems $F$ in $\Omega_\mathcal{X}^{\alpha}$.

-- Second, let $\Omega_{\mathcal{X}}^{m}$ denote measure-preserving systems (this assumption can be relaxed using resolvent bounds) with continuous dynamics. Algorithms again exist that converge to $\spec_{\mathrm{ap}}(\mathcal{K}_F)$, but require two successive data limits for all systems in this class. The necessity of these successive limits arises naturally: in \cref{EDMD1,EDMD2}, challenge \textbf{(C1)} implies that convergence rates of the underlying sums cannot be uniformly controlled \cite{krengel2011ergodic}, a phenomenon well known in mixing systems, among others. One must therefore first take the large-data limit ($M \to \infty$), followed by the dictionary limit ($N \to \infty$); in general, these limits cannot be combined or reversed. Our analysis further shows that, for \textit{any} algorithm (not only EDMD), no universal rule linking $M$ and $N$ guarantees convergence when both grow simultaneously.

-- Finally, consider the most general scenario, $\Omega_{\mathcal{X}}$, consisting of all continuous systems without further assumptions.\footnote{\scriptsize{While the continuity of $F$ can be relaxed, it is often assumed because discontinuities can lead to pathologies \cite{mane2012ergodic}.}}
Remarkably, reliable algorithms still exist, but they require three successive data limits to converge to the spectrum $\spec_{\mathrm{ap}}(\mathcal{K}_F)$ for all $F$ in $\Omega_\mathcal{X}$: increasing the number of snapshots ($M\rightarrow\infty$), expanding the dimensionality of the subspaces ($N\rightarrow\infty$), and refining coherence estimates into spectra (take the regularization parameter $\varepsilon \downarrow 0$ in \cref{eq:approx_coherency}). Without assumptions on the underlying system, no simpler two-step algorithm can achieve guaranteed convergence (Theorem 2.15 of the Supplementary Information). \textbf{(C1)} and \textbf{(C2)} each cost a limit, illustrating the intrinsic difficulty of learning general dynamical systems.

\subsubsection*{Unified complexity classifications of learning}

\begin{figure}
\centering
\includegraphics[width=1\linewidth]{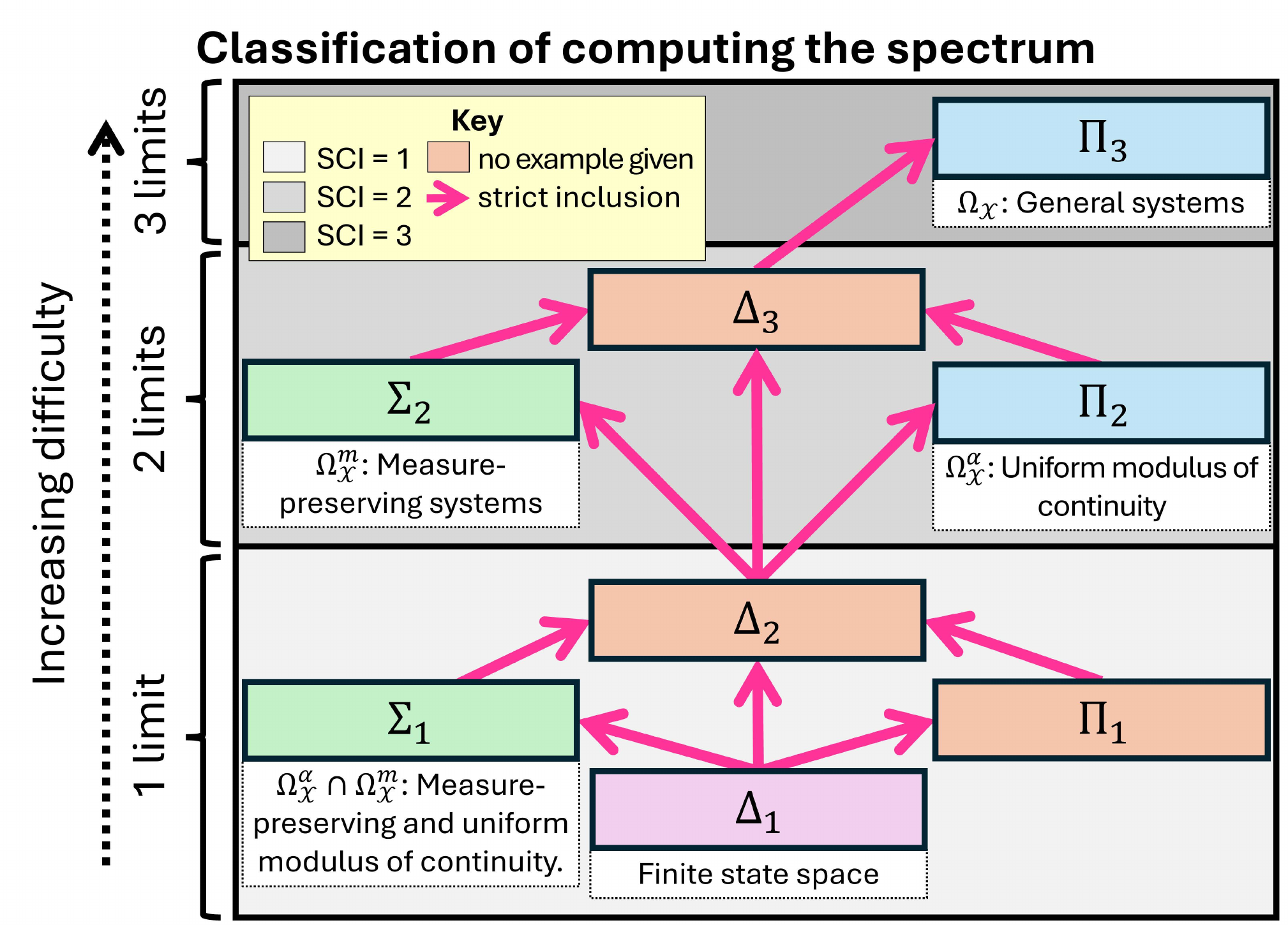}
\caption{\scriptsize\textbf{Classifications for learning Koopman spectra from trajectory data.} Each SCI level indicates that solving the problem requires more complex, layered procedures. Each `limit' reflects an extra data refinement. Our results establish both upper bounds (via convergent algorithms) and lower bounds (through adversarial dynamical systems) on the computational complexity of these problems. The class $\Delta_{m+1}$ comprises problems with $\mathrm{SCI}\leq m$. The $\Sigma$ and $\Pi$ classes characterize how verification occurs in the final limit when learning Koopman operator spectra: $\Sigma$ corresponds to convergence from within, and $\Pi$ to convergence from above. For finite state spaces, the problem lies in $\Delta_1$, since the Koopman operator reduces to a finite-dimensional matrix whose spectrum can be computed by a single convergent algorithm with explicit rates.}
\label{fig:SCI1}
\end{figure}

The Solvability Complexity Index (SCI) \cite{Hansen_JAMS,ben2015can} formalizes how many limits are needed to learn properties from data. For the full class $\Omega_\mathcal{X}$, computing $\spec_{\mathrm{ap}}(\mathcal{K}_F)$ requires three separate limits, which cannot be reduced or combined: no method can succeed with only two. For structured subclasses like $\Omega_{\mathcal{X}}^{\alpha,m}$, $\Omega_{\mathcal{X}}^{\alpha}$, or $\Omega_{\mathcal{X}}^{m}$, the above shows that fewer limits suffice.

\begin{table*}[th!]
\centering\setlength{\tabcolsep}{2pt}
\caption{\scriptsize\textbf{Convergence results for Koopman operators in the SCI hierarchy.} This table summarizes convergence results from the Koopman literature, interpreted through the lens of the Solvability Complexity Index (SCI). ``N/C'' denotes non-convergence without additional strong assumptions (e.g., requiring observables to lie within a finite-dimensional invariant subspace), ``n/a'' indicates the method is not applicable to the spectral problem, and  ``m.p.'' stands for measure-preserving systems. A superscript $^*$ indicates the SCI bound improves by one if errors in approximations of correlations (see \cref{EDMD1,EDMD2}) are controlled, e.g., through known variational bounds on $F$. This reduction depends on system properties \cite{kachurovskii1996rate}. The column $\mathcal{K}g$ refers to approximating $\mathcal{K}$'s action on observables $g$. Note that computing the spectral measure (distribution of the system's behavior across different frequencies) may not yield the spectrum due to spurious eigenvalues. Upper bounds typically assume access to a dictionary with projections converging strongly to the identity; this construction is either specified (e.g., the entries ``compactification methods'' and ``diffusion maps'' have the distinct advantage of learning a well-conditioned dictionary) or required as an additional input by the user (e.g., EDMD and generator EDMD) which may or may not increase the SCI. In our upper bounds in \cref{fig:SCI1}, we construct a dictionary. \textit{N.B. These results are upper bounds on the SCI. Many of these bounds are not sharp, meaning they overestimate the number of limits required.}}\label{tab:SCI_table}\scriptsize
\begin{tabular}{|l|c|l|l|l|l|}
\hline
\multicolumn{1}{|c|}{\multirow{2}{*}{\textbf{Algorithm}}} & \multicolumn{1}{c|}{\multirow{2}{*}{\textbf{Comments/Assumptions}}} & \multicolumn{4}{c|}{\textbf{Spectral Problem's Corresponding SCI Upper Bound}}\\ \cline{3-6} 
\multicolumn{1}{|c|}{}& \multicolumn{1}{c|}{}& \multicolumn{1}{c|}{$\mathcal{K}g$} & \multicolumn{1}{c|}{\textit{Spectrum}} & \multicolumn{1}{c|}{\textit{Spectral Measure (if m.p.)}} & \multicolumn{1}{c|}{\textit{Spectral Type (if m.p.)}} \\
\hline
Extended DMD \cite{williams2015data} & general $L^2$ spaces & $\mathrm{SCI}\leq 2^*$ & N/C & N/C & n/a\\\hdashline[0.5pt/1pt]
\multirow{2}{*}{Residual DMD \cite{colbrook2021rigorousKoop}} & \multirow{2}{*}{general $L^2$ spaces} & \multirow{2}{*}{$\mathrm{SCI}\leq 2^*$} & \multirow{2}{*}{$\mathrm{SCI}\leq 3^*$} & \multirow{2}{*}{$\mathrm{SCI}\leq 2^*$} & varies, see \cite{colbrook2019computing}\\
&  & & & & e.g., a.c. density: $\mathrm{SCI}\leq 2^*$\\
\hdashline[0.5pt/1pt]
\multirow{2}{*}{Measure-preserving EDMD \cite{colbrook2023mpedmd}} & \multirow{2}{*}{m.p. systems} & \multirow{2}{*}{$\mathrm{SCI}\leq 1$} & \multirow{2}{*}{N/C} & $\mathrm{SCI}\leq 2^*$ (general) & \multirow{2}{*}{n/a}\\
&  & & & $\mathrm{SCI}\leq 1$ (delay-embedding) & \\
\hdashline[0.5pt/1pt]
Hankel DMD \cite{arbabi2017ergodic} & m.p. ergodic systems & $\mathrm{SCI}\leq 2^*$ & N/C & N/C & n/a\\\hdashline[0.5pt/1pt]
Christoffel--Darboux kernel \cite{korda2020data} & m.p. ergodic systems & $\mathrm{SCI}\leq 3$ & n/a & $\mathrm{SCI}\leq 2$ & e.g., a.c. density: $\mathrm{SCI}\leq 2$\\\hdashline[0.5pt/1pt]
\multirow{2}{*}{Generator EDMD \cite{klus2020data}} & cts.-time, samples $\nabla F$ & \multirow{2}{*}{$\mathrm{SCI}\leq 2$} & \multirow{2}{*}{N/C} & \multirow{2}{*}{$\mathrm{SCI}\leq 2$}
& \multirow{2}{*}{n/a}\\
& (otherwise additional limit) &&&&\\
\hdashline[0.5pt/1pt]
Compactification \cite{das2021reproducing} & cts.-time, m.p. ergodic systems & $\mathrm{SCI}\leq 4$ & N/C & $\mathrm{SCI}\leq 4$ & n/a\\\hdashline[0.5pt/1pt]
Resolvent compactification \cite{valva2023consistent} & cts.-time, m.p. ergodic systems & $\mathrm{SCI}\leq 5$ & N/C & $\mathrm{SCI}\leq 5$ & n/a\\\hdashline[0.5pt/1pt]
Diffusion maps \cite{berry2015nonparametric} & cts.-time, m.p. ergodic systems & $\mathrm{SCI}\leq 3$ & n/a & n/a & n/a \\
\hline
\end{tabular}
\end{table*}

This framework provides a systematic way to assess the complexity of data-driven problems. Applied to Koopman operators, \cref{fig:SCI1} summarizes the difficulty of learning their spectra based on our results. Each classification includes both an upper bound (a convergent algorithm), and a lower bound, established by constructing adversarial dynamical systems to prove that fewer limits are insufficient. We can also relate existing algorithms to the SCI hierarchy by summarizing convergence results from the Koopman literature and their corresponding implicit upper bounds in \cref{tab:SCI_table}. Each algorithm relies on specific system assumptions and, in some cases, uses more data limits than necessary for convergence.

Beyond \cref{tab:SCI_table}, Ulam's method is widely used to approximate eigenvalues of transfer operators, which describe the evolution of probability densities and are dual to Koopman operators. It typically requires two successive limits: one for Monte Carlo approximation and another for increasing matrix size \cite{froyland2007ulam}, but convergence is not always reliable \cite[Section 2.6]{blank2002ruelle}. Adding a third limit via noise smoothing can improve convergence \cite{blank2002ruelle}, though adaptive noise selection sometimes reduces this back to two. Similar SCI classifications apply to other data-driven dimensionality reduction methods \cite{klus2018data}. For multiple limits in control theory, see \cite[Theorem 3]{peitz2019koopman}.

These examples illustrate that multiple-limit phenomena are central to many data-driven methods in dynamical systems. While methods provide SCI upper bounds, a key challenge is determining whether these bounds are optimal. This raises fundamental questions: \textit{Can convergence be achieved with fewer limits? If not, what assumptions make it easier?} To address these, we have derived lower bounds that show how system properties and the quantities being computed shape optimal algorithm design. When upper and lower bounds match, the algorithm is provably \textit{optimal} for the problem.

\vspace{3mm}

\textit{In short, we now have a characterization of when data-driven spectral learning can succeed and when it cannot.}

\subsubsection*{Learning eigenpairs and latent spaces is hard}

\begin{figure}
\centering
\includegraphics[width=1\linewidth]{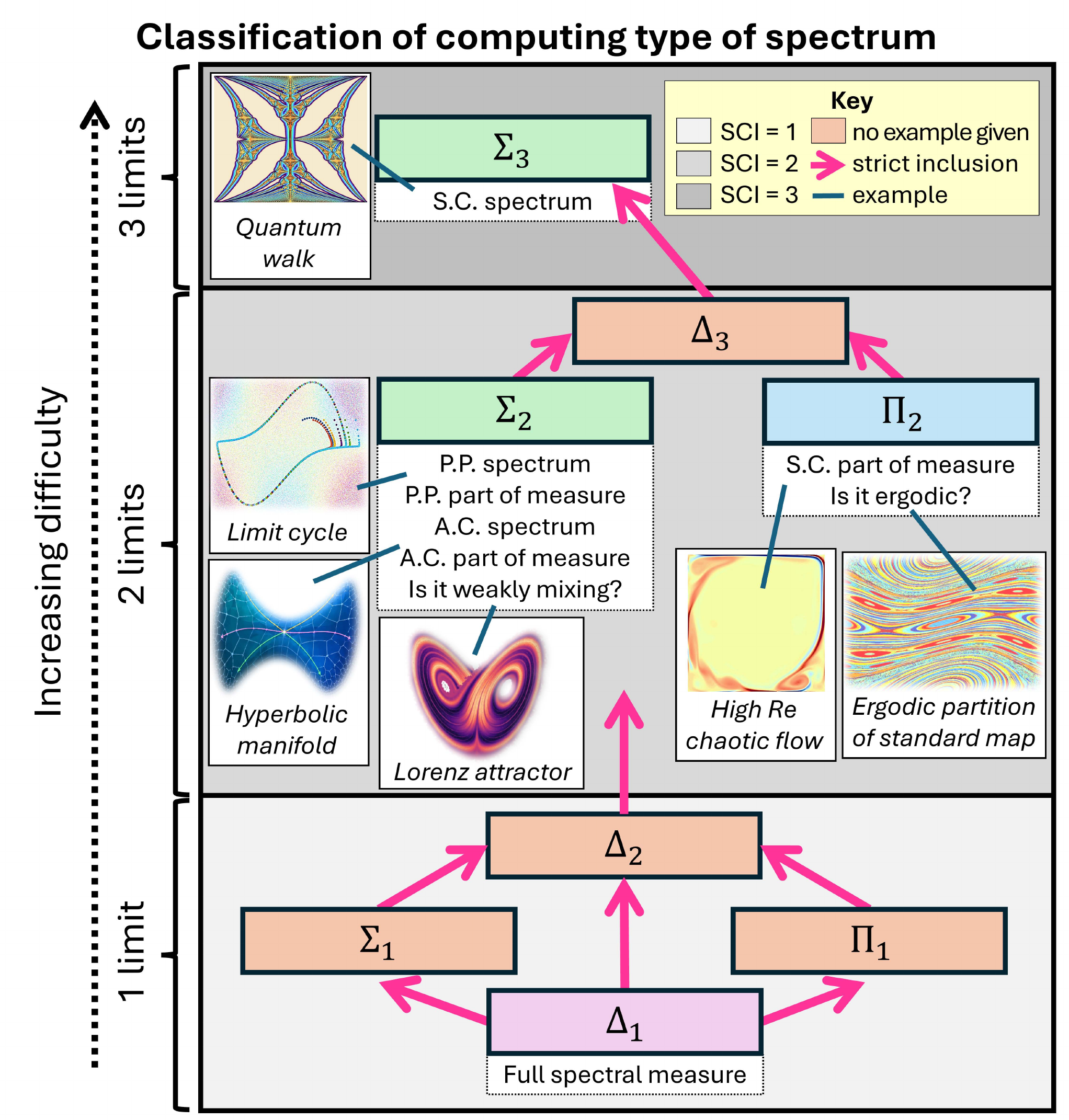}
\caption{\scriptsize\textbf{Classifications for learning spectral types (eigenvalues (p.p.), absolutely continuous (a.c.) and singular continuous (s.c.)) of Koopman operators for measure-preserving invertible systems from trajectory data.} Each classification comprises an upper bound (convergent algorithms) and a lower bound (established by constructing adversarial dynamical systems). The SCI is the number of data limits needed to solve a problem. For each learning objective, we provide a representative system to emphasize that the classification fundamentally depends on the nature of the underlying dynamics. These examples illustrate how different dynamical behaviors influence the complexity of learning spectral types from data. The various classes $\Delta_m$, $\Sigma_m$, and $\Pi_m$ are described in \cref{fig:SCI1}.\vspace{-5mm}}
\label{fig:SCI2}
\end{figure}

As a final problem, we consider the complexity of determining the spectral type, i.e., eigenvalues versus continuous components (bottom panel of \cref{fig:EDMD_not_converge}), of measure-preserving systems. Spectral types distinguish between recurrent patterns (periodic or quasiperiodic oscillations) and more chaotic or mixing behavior, based on how the system spreads energy over frequencies \cite[page 45]{katznelson2004introduction}. They play a key role in applications such as fluid mechanics \cite{mezic2013analysis}, anomalous transport \cite{zaslavsky2002chaos}, and analysis of trajectory invariants and exponents \cite{kantz2004nonlinear}, and are particularly important in reduced-order modeling \cite{mezic2004comparison,mezic2005spectral}. We discussed their role in the above cavity flow problem.

Identifying eigenvalues and eigenfunctions of Koopman operators ($\varepsilon=0$ in \cref{eq:approx_coherency}) reveals coordinates in which nonlinear dynamics evolve linearly, enabling the expansion in \cref{KMD_example}. Non-unit eigenvalues ($\lambda \neq 1$) encode time-varying structure. Let $\Omega_p$ denote the class of smooth, invertible, measure-preserving systems on the torus with uniformly bounded derivatives. This class lies within $\Omega_{\mathcal{X}}^{\alpha,m}$; consequently, our universal one-limit algorithm applies and provably computes the full spectrum—without requiring any a priori distinction between discrete and continuous components.

Strikingly, even for this well-structured class, no single-limit learning algorithm (deterministic or probabilistic with success probability exceeding 50\%) based on trajectory data can determine whether the Koopman operator $\mathcal{K}_F$ admits a non-unit eigenvalue for $F$ in $\Omega_p$. Likewise, no such algorithm can converge to the set of eigenvalues (Theorem 2.12 of the Supplementary Information).

However, both problems become computable through two interdependent data limits (Supplementary Algorithms 6 and 7): adaptively adjusting the time lag in autocorrelation estimates according to data size, and increasing projections onto finite-dimensional subspaces defined by a dictionary. We demonstrated this strategy on a complex fluid flow (see also Methods). This two-limit procedure is provably optimal: no single-limit algorithm can resolve these questions. The corresponding SCI classifications are summarized in \cref{fig:SCI2}.

These results help explain the challenges in finding finite-dimensional representations, such as autoencoders or latent spaces, in which the dynamics appear linear \cite{lusch2018deep}. In particular, it is fundamentally impossible to determine or verify the existence of Koopman eigenvalues using only a single limit. A constructive alternative is to compute approximate eigenfunctions associated with $\spec_{\mathrm{ap}}(\mathcal{K}_F)$. This problem admits a single-limit algorithm with verification (\cref{fig:SCI1}), providing a sharp and practically implementable route for spectral analysis. We applied this to the Arctic sea ice dataset above.

\subsection*{Discussion}

We developed a powerful technique, the construction of adversarial dynamical systems (see \S~\textit{Adversarial systems} and \cref{fig:proof1}), to establish impossibility results in data-driven dynamical systems. We identified conditions (\textbf{(C1)}, \textbf{(C2)}) under which no sequence of randomized algorithms (e.g.,  based on randomly sampled trajectories) can succeed with probability greater than 50\%. These failures are not rare pathological cases, but reflect fundamental barriers intrinsic to the problem.
These barriers, in turn, guided the development of a suite of algorithms (Supplementary Algorithms 1–7) that achieve optimal performance and are provably convergent. The resulting methods produce trustworthy, verifiable outputs and include what we term multi-data-limit methods (see \cref{fig:SCI1,fig:SCI2}). By linking dynamical systems theory with the foundations of computation, we establish a computational complexity framework for data-driven dynamical systems. This framework clarifies the inherent limitations of using finite data to analyze complex dynamics while simultaneously identifying algorithmic possibilities with broad applicability. It provides a synthesis in ML between foundational theory and concrete application.

We demonstrated the power of our framework by solving a long-standing problem: computing Koopman spectra from data without spurious eigenvalues or missing components. For example, we show that there exists a constructive computational procedure for learning $\mathrm{Sp}_{\mathrm{ap}}(\mathcal{K}_F)$ from data for general continuous systems. Our algorithms succeed where existing methods (e.g., EDMD) struggle, recovering accurate spectra in both low- and high-dimensional systems, including cases with continuous spectra and a real-world application to Arctic sea ice. 

Arctic amplification has accelerated sea ice loss, with major consequences for ecosystems, local communities, and extreme weather. A key challenge is identifying the geographic patterns driving these changes. Our algorithms uncover hidden Koopman modes linked to sea ice decline (\cref{fig:sea_ice_evals}), offering dynamic and spatial insight supported by \textit{verification}. These decaying modes are physically meaningful and may guide future measurements. Using spectral approximations and error bounds to truncate spurious components, we built forecast models that achieve state-of-the-art accuracy at a fraction of the cost of deep learning models (\cref{fig:sea_ice_forecast2} and surrounding discussion). Applications include optimizing shipping routes, reducing environmental risks, and informing early-warning systems. Understanding sea ice loss is also vital given potential links to extreme events such as wildfires, floods, heatwaves, and cold spells. The Koopman framework further enables systematic comparisons across climate models and supports the construction of interpretable response formulas and reaction coordinates near critical transitions, with broad implications for climatology \cite{lucarini2023theoretical}.

Our approach extends beyond Koopman operators, provided appropriate domain-specific ``sudden change'' lemmas can be established (see Methods). In each case, an analog of the ``ball of learnability'' applies. This perspective is relevant to a range of methods, including SINDy \cite{brunton2016discovering}, neural ODEs \cite{chen2018neural}, Fourier neural operators \cite{LiKALBSA21}, LSTMs \cite{vlachas2018data}, and PDE-net \cite{long2018pde}, as well as to other areas of scientific computing with ML. For example, recent work has shown that linear elliptic PDEs can be learned from input-output pairs \cite{boulle2023elliptic}, analogous to snapshot-based learning. Whether similar approaches extend to hyperbolic or nonlinear PDEs remains open, but our proof techniques may offer insight into this challenge.

This paper initiates a broader effort to explore the limits of robust learning and to develop a theory of necessary and sufficient conditions. Several promising directions remain. 
First, we have assumed access to full-state observations in \cref{eq:snapshot_data}, whereas many applications involve only partial measurements of the state $x$. Extending our upper and lower bounds to such settings is a natural next step.
Second, we focused on discrete-time systems, reflecting how data is typically collected. It would be valuable to study continuous-time dynamics and the sampling conditions needed for reliable learning. Our results should carry over under generic time discretizations, but a formal analysis remains open.
Third, Koopman-based methods have shown promise in control problems across domains such as 
power grids \cite{netto2018robust},
robotics \cite{haggerty2023control},
fluid dynamics \cite{arbabi2018data}, 
chemistry \cite{narasingam2019koopman}, and
biology \cite{hasnain2020steady}, where the linearity of $\mathcal{K}$ enables tractable control design \cite{strasser2026overview}. Our development of provably convergent, error-bounded algorithms for Koopman spectral properties opens the door to significant advances in nonlinear control. Fourth, it is natural to seek lower bounds complementing the upper bounds in \cref{tab:SCI_table} for problems that do not rely on spectral computations. These may have different lower bounds than those we established, but we anticipate that the adversarial dynamical systems framework can be extended to cover them.

ML in data-driven dynamical systems is rapidly expanding, and this momentum shows no signs of slowing. Across nearly every area, key challenges are being reexamined through big data and deep learning. With this surge of interest and innovation, it is crucial for the community to grasp not only what is possible but also what is fundamentally impossible. This prevents the pursuit of unattainable algorithms or methods, safeguards against potentially catastrophic errors, and reveals the conditions under which learning is feasible: upper and lower bounds inform and sharpen one another. This process led us to our convergent algorithms. Such classifications are essential if we are to fully harness the power of ML in dynamical systems.

\subsection*{Methods}

\subsubsection*{Computing Koopman modes and EDMD}
For a dictionary of observables $\{g_j\}_{j=1}^N$ and snapshot pairs $(x^{(m)},y^{(m)})$ in \cref{eq:snapshot_data}, EDMD approximates two correlation matrices
$
G_{ij}=\langle g_j,g_i\rangle$ and $A_{ij}=\langle \mathcal{K}g_j,g_i\rangle$ ($i,j=1,\ldots,N$)
from data averages using $\mathcal{K}g_j(x^{(m)})=g_j(y^{(m)})$, analogously to \cref{ergodic_energy_example}:\vspace{-2mm}
\begin{align}\setlength\abovedisplayskip{3pt}\setlength\belowdisplayskip{3pt}
G_{ij}\approx \hat{G}_{ij} &=\frac{1}{M}\sum_{m=1}^M g_j(x^{(m)})\overline{g_i(x^{(m)})},\quad i,j=1,\ldots,N,\label{EDMD1}\\
A_{ij}\approx \hat{A}_{ij} &=\frac{1}{M}\sum_{m=1}^M g_j(y^{(m)})\overline{g_i(x^{(m)})},\quad i,j=1,\ldots,N.\label{EDMD2}\vspace{-2mm}
\end{align}
(Here, the $L^2$ inner product $\langle \cdot,\cdot \rangle$ is a generalized dot product that measures how strongly two observables are correlated.) The Koopman approximation is the $N\times N$ matrix $\hat{G}^{-1}\hat{A}$ and EDMD computes its eigenvalues.

We compute the Koopman modes ${\bf g}_j$ appearing in \cref{KMD_example} as follows. Given a collection of $n$ approximate eigenfunctions (computed, e.g., by EDMD or our algorithms), let $\Psi\in\mathbb{C}^{M\times n}$ denote the matrix corresponding to their time series across the training data $x^{(m)}$ ($m=1,\ldots,M$) in \cref{eq:snapshot_data}. For a vector of observations ${\bf g} \in \mathbb{C}^N$, let ${\bf O}\in\mathbb{C}^{M\times N}$ be their collected data over the same training data. The Koopman modes ${\bf g}_j$ are computed by
\begin{equation}\setlength\abovedisplayskip{6pt}\setlength\belowdisplayskip{6pt}
\begin{pmatrix}
{\bf g}_1& {\bf g}_2 &\cdots & {\bf g}_n
\end{pmatrix}^\top=\Psi^\dagger{\bf O},
\end{equation}
where $\dagger$ denotes the Moore--Penrose inverse (solution of the least squares problem) and $\top$ the matrix transpose.

In \cref{fig:sea_ice_evals,fig:sea_ice_decay}, we have plotted the absolute values of Koopman modes since these describe spatial presence of the dynamics described by eigenvalues (as explained in the discussion surrounding \cref{KMD_example}). The error bars for these Koopman modes and eigenvalues correspond to $\|\mathcal{K} \phi_\varepsilon-\lambda \phi_\varepsilon\|$ (computed using the function in \cref{h_fun}) for the corresponding approximate eigenfunction $\phi_\varepsilon$ (normalized so $\|\phi_\varepsilon\|=1$). These provide an error bound for coherency in \cref{eq:approx_coherency}.

\subsubsection*{Upper bounds: A suite of convergent algorithms}

There are several algorithms, each corresponding to different classes of dynamical systems in \cref{fig:SCI1}, and full pseudocode is provided in the Supplementary Information. Recall from the Duffing oscillator example that we consider a dictionary of observables $\{g_j\}_{j=1}^N$ and that we have access to the snapshot data in \cref{eq:snapshot_data}.

We assume the state space is a compact metric space \((\mathcal{X}, d_{\mathcal{X}})\) and consider observables \(g\) in \(L^2(\mathcal{X}, \omega)\), the space of functions whose squared values can be averaged (integrated) over \(\mathcal{X}\) using a probability measure \(\omega\), which defines how different parts of the state space are weighted. We also assume that $F$ is nonsingular with respect to $\omega$ (if $\omega(E)=0$ then $\omega(\{x:F(x)\in E\})=0$) and $\mathcal{K}_F$ is bounded. This is the standard setting in Koopman analysis, often referred to as the \textit{spectral study} of the system. Notably, our methods apply to \textit{any} compact metric space \(\mathcal{X}\) and probability measure \(\omega\). To assess how well algorithms capture the spectrum of the Koopman operator, we use the Hausdorff metric, which measures how far two spectral sets are from each other by quantifying the greatest distance one must travel from a point in one set to reach the other. This metric ensures that the computed spectra converge accurately, avoiding errors such as including false eigenvalues or missing important spectral regions.

We first describe the algorithm for $\Omega_{\mathcal{X}}^{\alpha,m}$ (Supplementary Algorithm 1). Traditional methods for analyzing the dynamics of complex systems (e.g., EDMD) often rely on estimating the Koopman operator by computing eigenvalues of finite-dimensional matrices. However, these approximations can be unreliable and fail to converge, as shown in \cref{fig:EDMD_not_converge,fig:other_examples}. Our approach avoids this pitfall by taking a different route. Instead of directly computing eigenvalues of an approximate matrix, we use trajectory data from the system to construct \textit{three} matrices that capture the essential correlations between observables and their evolution. These matrices approximate inner products involving the Koopman operator $\mathcal{K}_F$ and its adjoint $\mathcal{K}_F^*$ (which encodes how observables change in \textit{reverse time}).

We form a correlation matrix based on how different observables behave across the dataset:
\begin{equation}\setlength\abovedisplayskip{6pt}\setlength\belowdisplayskip{6pt}
\hat{L}_{ij} =\frac{1}{M}\sum_{m=1}^M g_j(y^{(m)})\overline{g_i(y^{(m)})},\quad i,j=1,\ldots,N,
\end{equation}
which estimates how the time-evolved observables overlap. It approximates $\langle \mathcal{K}g_j,\mathcal{K}g_i\rangle$ in an analogous fashion to how \cref{EDMD1,EDMD2} approximate $\langle g_j,g_i\rangle$ and $\langle \mathcal{K}g_j,g_i\rangle$, respectively. Using these matrices, we define a function 
\begin{equation}\setlength\abovedisplayskip{5pt}\setlength\belowdisplayskip{5pt}\label{h_fun}
\smash{h_{N}(z,F)=\sqrt{\sigma_{\mathrm{inf}}(\hat{L}-\overline{z}\hat{A}-z\hat{A}^*+|z|^2\hat{G})}.}
\end{equation}
Here, $\sigma_{\mathrm{inf}}$ denotes the smallest singular value of a matrix.
This method has a key advantage: If $F$ has a known modulus of continuity, we adaptively select $N$ based on the number of data points $M$ (amounting to controlling errors in the approximations of the correlations) to ensure that
\begin{equation}\setlength\abovedisplayskip{6pt}\setlength\belowdisplayskip{6pt}
\lim_{N\rightarrow\infty}h_{N}(z,F)=\sigma_{\mathrm{inf}}(\mathcal{K}_F-zI),
\end{equation}
with convergence from above. Here, $I$ denotes the identity operator. Moreover, this specific construction guarantees convergence, in contrast to approaches using $\smash{\sigma_{\mathrm{inf}}(\hat{A}-z\hat{G})}$, which lack this property. It establishes the rigorous error bounds demonstrated, for example, in \cref{fig:EDMD_not_converge,fig:other_examples}.

We evaluate this function over an initial grid of $z$-points that is adaptively refined with increasing $N$ (a good choice is $\{z\in\frac{1}{n}\mathbb{Z}+\frac{i }{n}\mathbb{Z}:|z|\leq n\}$). The computed function $h_{N}(z,F)$ yields an upper bound $d_{N,z}$ on $\mathrm{dist}(z,\spec_{\mathrm{ap}}(\mathcal{K}_F))$, provided the non-normality of $\mathcal{K}_F$ is suitably controlled (e.g., if the system is measure-preserving, see also Equation (16) of the Supplementary Information for the general resolvent-bounded case). We then find local minimizers of $h_{N}(z,F)$ in a radius $d_{N,z}$ neighborhood of each grid point. These minimizers provide approximations to the spectrum near each point. The union of these local approximations is guaranteed to converge to the true spectrum as $N\rightarrow\infty$. By computing the singular vectors corresponding to the smallest singular values, we simultaneously obtain approximate eigenfunctions \(\phi_\varepsilon\), which satisfy the condition \(\|\mathcal{K}_F \phi_\varepsilon - \lambda \phi_\varepsilon\| \leq \varepsilon\) (with $\varepsilon=\sigma_{\mathrm{inf}}(\mathcal{K}_F-\lambda I)$ for a spectral parameter $\lambda$) and hence the coherency bound in \cref{eq:approx_coherency}. In addition to providing error bounds, the algorithm is local, trivially parallelizable, and stable.

For the class $\Omega_{\mathcal{X}}^{\alpha}$, we cannot directly convert $h_{N}(z,F)$ into the distance bound $d_{N,z}$ as described above. This limitation reflects the challenge \textbf{(C2)} and requires an additional limiting process to achieve the conversion (Supplementary Algorithm 3). Similarly, for the class $\Omega_{\mathcal{X}}^{m}$, we are unable to adaptively choose $N$ based on the number of data points $M$, due to the challenge \textbf{(C1)}. As a result, we must evaluate $h_{N}(z,F)$ using another limiting procedure (Supplementary Algorithm 4). In the most general case, $\Omega_{\mathcal{X}}$, both challenges arise simultaneously, necessitating three successive limits (Supplementary Algorithm 5). This phenomenon of several successive limits occurs in all algorithms for Koopman operators that provably converge. In particular, the above argument using the matrices $L$ is a generalization of the ResDMD algorithm \cite{colbrook2021rigorousKoop}.

To separate eigenvalues from continuous spectra, we use a mathematical result known as the RAGE theorem. It states for unitary $\mathcal{K}$ that we can isolate the contribution of eigenfunctions by averaging time-evolved observables after projecting onto an increasing sequence of finite-dimensional subspaces. Specifically, if $\{\mathcal{P}_n\}$ are finite-rank orthogonal projections converging to the identity, then for any observable \(g\), the projection $\mathcal{P}_{\mathrm{pp}}$ onto the closure of the space spanned by eigenfunctions of $\mathcal{K}$ satisfies:
\begin{equation}\label{RAGE_statement}\setlength\abovedisplayskip{6pt}\setlength\belowdisplayskip{6pt}
\|\mathcal{P}_{\mathrm{pp}}g\|^2=\lim_{n\rightarrow\infty}\lim_{L\rightarrow\infty}
\frac{1}{2L+1}\sum_{\ell=-L}^L\left\|\mathcal{P}_n\mathcal{K}^{\ell}g\right\|^2.
\end{equation}
This formula separates observables into parts that are quasiperiodic over time and parts corresponding to continuous spectra (decay of correlations). This provides a principled way to identify and extract the parts of an observable associated with eigenfunctions. The procedure is implemented in Supplementary Algorithms 6 and 7.

\cref{RAGE_statement} involves two consecutive limits: one over the increasing rank of the projections or number of observables (\(n\rightarrow\infty\)) and another over the time window length (\(L\rightarrow\infty\)). 
This method enables accurate spectral computations using only finite data, without requiring an explicit approximation of the Koopman operator (which must be formed for EDMD). We also point out that this technique is substantially different from partitions of the full spectral measure (distribution of the system's behavior across different frequencies) over intervals, which have previously been computed using the ergodic theorem, e.g., \cite{korda2020data}. Moreover, the RAGE theorem does not require the system to be ergodic.

\subsubsection*{Duffing oscillator}
 
The Duffing oscillator is the system of equations:
\begin{equation}\setlength\abovedisplayskip{6pt}\setlength\belowdisplayskip{6pt}
\frac{dx}{dt}=y,\quad \frac{dy}{dt}=-\gamma y +x(1-x^2).
\end{equation}
The state \({\bf{x}}=(x, y) \in \mathbb{R}^2\) evolves in a two-dimensional state space. To analyze this system in discrete time, we use a time step of \(\Delta t = 0.3\). We consider two cases based on the parameter \(\gamma\). \textit{Conservative case (\(\gamma = 0\)):} This is a Hamiltonian system: it conserves energy and, by Liouville's theorem, preserves phase-space area.
\textit{Dissipative case (\(\gamma = 0.3\)):} In this scenario, the system has energy dissipation, leading to two stable spirals at \((\pm 1, 0)\) and a saddle point at the origin. These two cases highlight distinct dynamical behaviors. In the Hamiltonian case, trajectories exhibit long-term behavior without dissipation, while in the dissipative case, trajectories eventually converge to stable attractors. EDMD struggles to reliably approximate the Koopman operator spectrum in both cases, illustrating its non-convergence under varying dynamical conditions.

We first generate an initial dataset by uniformly sampling $10^4$ initial conditions ${\bf{x}}_0=(x_0,y_0)$ in the square \([-2, 2]^2\) and recording trajectories of 5 time steps for each sample. To construct a dictionary of observables, we apply \(k\)-means clustering to this initial dataset and use the resulting cluster centers $\{c_j\}_{j=1}^N$ to define $N$ radial basis functions:
\begin{equation}
g_j({\bf{x}})=\exp(-2\|{\bf{x}}-c_j\|_{l^2}/\sigma),\quad j= 1,\ldots, N.
\end{equation}
The trajectory length 5 is chosen to reduce the condition number of the resulting basis (Supplementary Figure 4)
and $k$-means clustering is a standard way to ensure that the centers are well distributed, preventing over-concentration and improving the representativeness and conditioning of the dictionary. The scaling parameter $\sigma$ is set to the average  $l^2$-norm of the snapshot data after it is shifted to mean zero, which we have found empirically to work well across a range of examples. Other dictionary choices are certainly possible.
The only requirement is that, as we use more basis functions, the projections acting on any fixed observable converge to that observable.

Once the dictionary is constructed, we further resample \(M\) snapshots of the system.
This corresponds to Monte Carlo integration, where the snapshots approximate the correlations required in our analysis. Supplementary Figure 5 illustrates the convergence of this approach as the number of snapshots increases, as well as the convergence of  $h_{N}(z,F)$ with increasing $N$.

The level curves of $\varepsilon$ in \cref{fig:EDMD_not_converge} were computed using $N=500$ basis functions and $M=50000$ snapshots, and Supplementary Algorithm 2. To compute the spectrum, we used the above adaptive procedure (Algorithm 1 in the S.I). We repeated experiments 10 times with different random seeds for the trajectory data to ensure robustness.

\subsubsection*{Further examples comparing EDMD and algorithms}

\begin{SCfigure*}[\sidecaptionrelwidth]
\centering
\includegraphics[width=12cm,trim={4mm 1mm 35mm 4mm},clip]{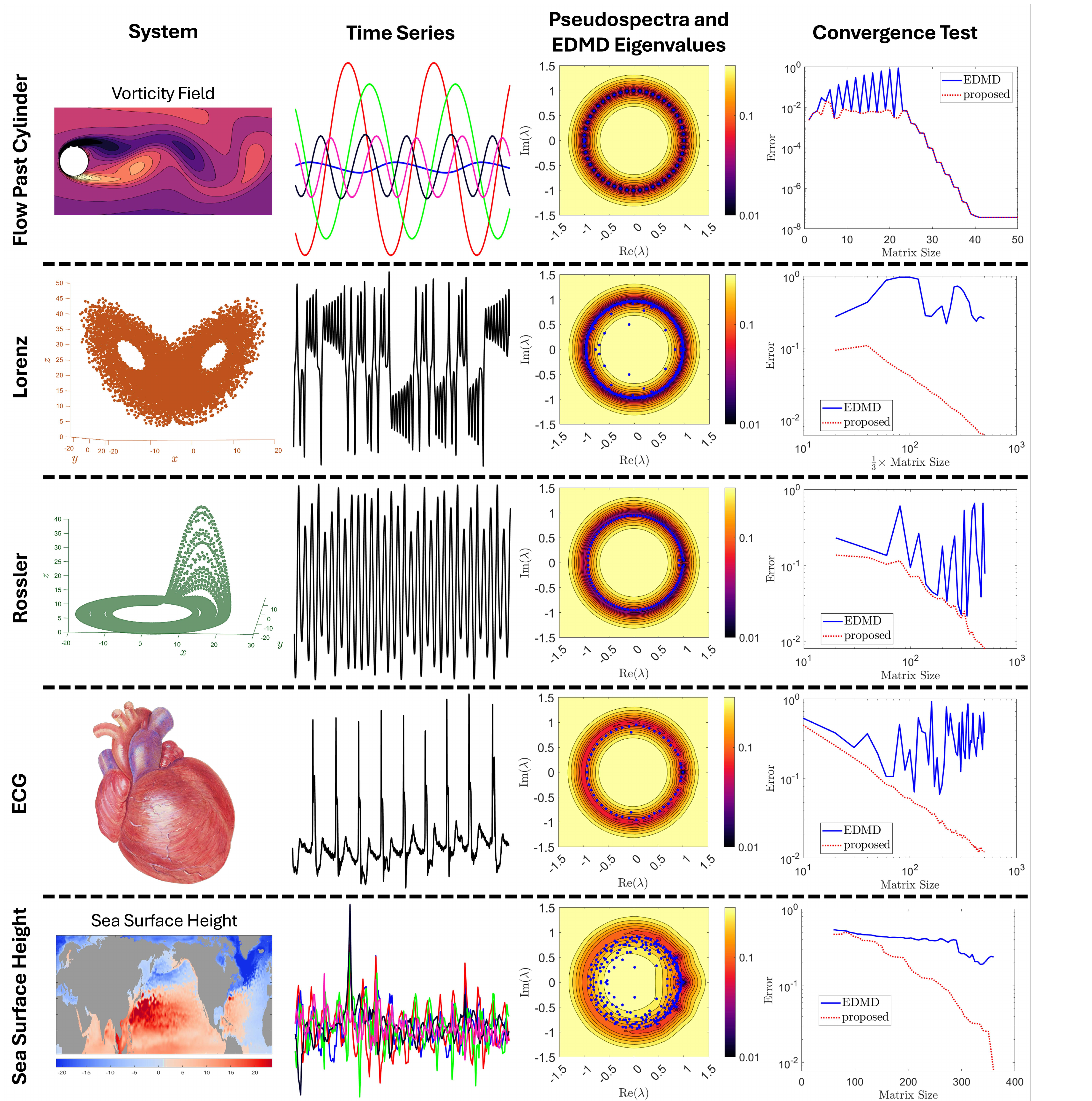}
\caption{\scriptsize\textbf{Spectral analysis of a range of analytic and real-world dynamical systems using our proposed algorithms.} 
Each row corresponds to a different system: (i) periodic flow past a cylinder at $Re=100$, (ii) Lorenz system, (iii) Rössler system, (iv) electrocardiogram (ECG) data (Image: \emph{Heart} by H. G. Wetselaar, Leiden University Libraries / Europeana, Public Domain), and (v) monthly mean sea surface height in the Northern Hemisphere (1950–present). Columns show: the system or dataset; a representative time series of observables; the computed level curves of $\varepsilon$ for $\varepsilon$-approximate eigenfunctions (color scale) termed ``pseudospectra'' overlaid with EDMD eigenvalues (blue dots); and a convergence comparison between EDMD (blue) and our method (red, using Supplementary Algorithm 1). The pseudospectrum shows where approximate, near-eigenvalue behavior occurs, providing a more robust picture of dynamics. In this final column, the error metric is the same as in \cref{fig:EDMD_not_converge}. Except for the periodic flow, EDMD yields spurious eigenvalues and lacks convergence. In contrast, our approach produces qualitatively accurate and convergent spectral approximations. Notably, the level curves of $\varepsilon$ reveal distinct spectral structures across systems: continuous spectra for chaotic systems (Lorenz, R\"ossler), spectral clustering near $\lambda = 1$ for ECG, and non-normal features and seasonal modes around $\lambda \approx \exp(2\pi i/12)$ in sea surface height data. Further experimental details are provided in the Supplementary Information.}
\label{fig:other_examples}
\end{SCfigure*}

Our algorithms are further applied to a variety of analytic and real-world systems, as shown in \cref{fig:other_examples}, using trajectory data with qualitatively different characteristics. Further details about each system and experiment are provided in the Supplementary Information, and code for all examples is publicly available. The examples span a broad range of dynamics, including periodic flow past a cylinder and canonical chaotic systems such as the Lorenz and Rössler attractors, which are two of the simplest systems that exhibit chaotic motion. To demonstrate applicability to more realistic scenarios, we also include data collected from an electrocardiogram (ECG), and monthly mean sea surface height in the Northern Hemisphere from 1950 to the present. Data sources are listed in the Supplementary Information. In all cases, our algorithms and EDMD use the same data and the same dictionary to ensure a fair comparison.

In each case, we computed both the pseudospectra (sublevel sets of $\varepsilon$ for approximate eigenfunctions) using Supplementary Algorithm 2 and the EDMD eigenvalues (blue dots). Except for the periodic flow case, EDMD produces spurious eigenvalues, as it did for the Duffing oscillator in \cref{fig:EDMD_not_converge}. \cref{fig:other_examples} demonstrates this phenomenon over a range of dictionary choices, data collections, and types of dynamical systems.

The pseudospectra exhibit different structures. For the cylinder flow, the multiplicative structure of the eigenvalues (see the discussion of combination modes in the Arctic sea ice example) is clearly visible in the first row. The Lorenz and Rössler systems display continuous spectra concentrated on the unit circle. The ECG data yields a tight spectral cluster near $\lambda = 1$ and near the frequency of the recorded heartbeat. The pseudospectra of the sea surface height data show strong non-normal features (the pseudospectra differ from spectral distances), capturing transient dynamics and nonstationary trends in data, along with dominant spectral regions corresponding to seasonal variations around $\lambda \approx \exp(2\pi i / 12)$.

In all cases, our spectral approximation algorithm (Supplementary Algorithm 1) converges as expected, whereas EDMD does not converge (except for the periodic flow case).

\subsubsection*{High Reynolds number fluid flows}

As an example of detecting eigenfunctions using an optimal two-limit procedure (Supplementary Algorithm 6), we consider two-dimensional lid-driven cavity flow \cite{arbabi2017study}. This system involves the motion of an incompressible, viscous fluid at a high Reynolds number (\(Re\)), which characterizes the dominance of inertial forces over viscous forces, leading to complex flow patterns. This setup provides a challenging example for identifying eigenfunctions of the underlying dynamics.

The domain is the cavity $B=[-1,1]^2$,  with stationary solid boundaries on all sides except the top. The top boundary moves with a regularized velocity profile $u_{\mathrm{top}} = (1-x^2)^2$. This standard boundary condition ensures both continuity and incompressibility, even at the corners of the top boundary. 
Using the streamfunction $\psi$, the incompressible Navier--Stokes equations for this flow can be reformulated as:
\begin{gather}\setlength\abovedisplayskip{6pt}\setlength\belowdisplayskip{6pt}
\frac{\partial }{\partial t} \nabla^2 \psi+ \frac{\partial \psi}{\partial y}\frac{\partial}{\partial x}\nabla^2 \psi - \frac{\partial  \psi}{\partial x}\frac{\partial}{\partial y}\nabla^2 \psi = \frac{1}{Re}\nabla^4 \psi,\\\psi|_{\partial B}=0,\, \frac{\partial \psi}{\partial n}(y\!=\!-1)=\frac{\partial \psi}{\partial n}(x\!=\!\pm 1)=0,\, \frac{\partial \psi}{\partial n}(y\!=\!1)=u_{\mathrm{top}}.
\end{gather}
These equations have a unique solution, and the flow dynamics converge to a universal attractor as time progresses \cite{temam2012infinite}. As discussed in the Results, the spectral structure of the Koopman operator tells us about the geometry of this attractor. To compute $\psi$, we use a Chebyshev spectral collocation method with an adaptive grid resolution \cite{trefethen2000spectral}, which depends on \(Re\), ensuring accurate computation for a range of flow conditions. For our analysis, we utilize $M=20000$ snapshots of the flow sampled at time intervals of $0.1$s to capture its dynamics, after an initial burn-in time to ensure data is collected from the attractor.

We apply \cref{RAGE_statement} to analyze the mean-subtracted total kinetic energy as our observable \(g\). To construct the projections \(\mathcal{P}_n\), we use time-delay embedding, with \(n\) time delays. This approach captures the temporal structure of the system by embedding the observable in a higher-dimensional space. Specifically, the terms \(\|\mathcal{P}_n \mathcal{K}^\ell g\|^2\) are calculated as follows: First, the Koopman operator \(\mathcal{K}^\ell\) is applied to the trajectory data of \(g\), which acts by shifting the time series by \(\ell\) steps. Then, the Moore--Penrose inverse is used to apply the projection \(\mathcal{P}_n\), ensuring consistency with the chosen time-delay embedding. Finally, the squared norm is computed by averaging as in \cref{EDMD2}, which allows us to extract long-term statistical averages from the trajectory data.

\begin{figure*}
\centering
\includegraphics[width=1\textwidth]{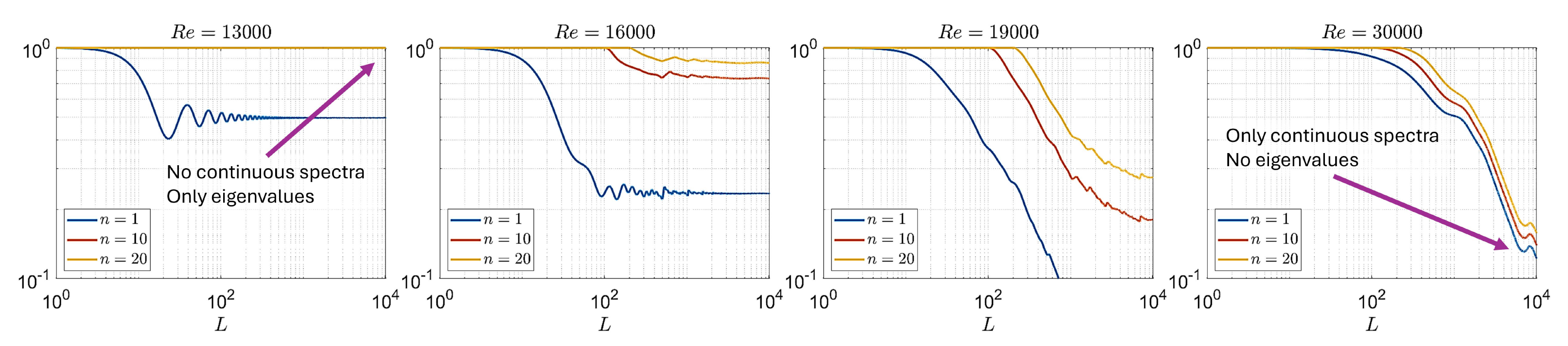}
\caption{\scriptsize\textbf{Application of \cref{RAGE_statement} to the kinetic energy of a lid-driven cavity flow.} The plots show $\frac{1}{2L+1}\sum_{\ell=-L}^L\|\mathcal{P}_n\mathcal{K}^{\ell}g\|^2$ (for normalized $g$), which in the double limit $\lim_{n\rightarrow\infty}\lim_{L\rightarrow\infty}$ converge to the fraction of $g$ made up of eigenfunctions (Supplementary Algorithm 6). This double limit procedure is used to prove upper bounds. As we move from left to right, the Reynolds number of the flow increases, and the spectral type becomes more dominated by continuous spectra. The structure of the spectrum reveals a sequence of bifurcations described in the main text.
}
\label{fig:RAGE_figure}
\end{figure*}

\cref{fig:RAGE_figure} shows the results for various choices of $Re$. We observe the double limit $\lim_{n\rightarrow\infty}\lim_{L\rightarrow\infty}$ at play and the structure of the spectrum reveals a sequence of bifurcations as the Reynolds number increases. For $Re \leq 10 000$, the flow converges to a steady laminar solution, corresponding to a fixed point in the state space. Just above $Re=10 000$, this steady solution becomes unstable, and the system transitions to a time-periodic flow, which remains stable up to $Re = 15 000$. The fundamental frequency of the periodic flow decreases with increasing Reynolds number. At $Re \geq 15 000$, a second bifurcation occurs, and the flow becomes quasiperiodic. The basic frequencies of the quasiperiodic motion also decrease with Reynolds number, until around $Re = 18 000$, where a third bifurcation is observed: the portion of kinetic energy associated with the continuous spectrum rises sharply to a few percent. This continuous component continues to grow, and by $Re \geq 22 000$, the state space dynamics exhibit no quasiperiodic structure: the Koopman spectrum consists solely of continuous components, indicating fully chaotic dynamics.
Our results are consistent with the findings of \cite{arbabi2017study,colbrook2024rigged}, where the authors computed spectral measures for this system.

The bottom of \cref{fig:EDMD_not_converge} illustrates the extraction of the eigenvalues for $Re=19000$. The top part of panel (e) shows the smoothed spectral measure (distribution of the kinetic energy across different frequencies) with a second-order smoothing kernel (taken from \cite{colbrook2024rigged})
 and smoothing parameter  $\epsilon=10^{-3}$ \cite{colbrook2024rigged}. The bottom part of panel (e) displays the extracted eigenvalues, computed using \cref{RAGE_statement} (Supplementary Algorithm 6). These are plotted against the Strouhal number, a dimensionless quantity that measures how often vortices are shed from a body relative to the speed of the flow and the size of the body.

\subsubsection*{Arctic sea ice}

Satellites have measured sea ice for decades using passive microwave sensors, which are then used to compute sea ice concentration using retrieval algorithms. We used data from the European Organisation for the Exploitation of Meteorological Satellites' (EUMETSAT) Ocean and Sea Ice Satellite Application Facilities (OSI-SAF) data record, comprising retrieval algorithms OSI-450 (1979–2015) and OSI-430-b (2016 onwards). These algorithms have been shown to be more accurate than other retrieval algorithms. Some portions of the data surrounding the North Pole are missing, and we use bilinear interpolation to fill these gaps. There are also missing data for three months in the 1980s due to satellite malfunctions, and we use linear interpolation to fill these gaps.

To construct our approximation of the Koopman operator, we consider observables $\mathbf{x}\in\mathbb{R}^{97877}$ representing sea ice concentration at each sea grid point. We then apply time-delay embedding: with $\tau-1$ delays, the augmented system state at time $n$ is given by (we use $(\cdot)$ to indicate time dependence instead of a subscript since $\mathbf{x}$ is a vector):
\begin{equation}\setlength\abovedisplayskip{6pt}\setlength\belowdisplayskip{6pt}
\widehat{\bf{x}}(n)=
\begin{pmatrix}
\mathbf{x}(n)\\
\mathbf{x}(n-1) \\
  \vdots\\
	\mathbf{x}(n-\tau+1)
\end{pmatrix}\in\mathbb{R}^{97877\cdot \tau}.
\end{equation}
That is, we include the sea ice concentration at the current and $\tau-1$ previous time steps. We choose $\tau=6$ to help capture semiannual patterns, though other choices also yield good performance.
As our dictionary, we use Gaussian radial basis functions of the form
$
\exp(-\tfrac{100}{81}\|{\widehat{\bf{x}}}-{\widehat{\bf{x}}}^{(m)}\|_{l^2}^2/\sigma^2),
$
where the centers ${\widehat{\bf{x}}}^{(m)}$ correspond to the $M$ snapshot data of the augmented system state (in contrast to the Duffing oscillator where we used $k$-means to choose centers). As in the Duffing oscillator example, the scaling parameter $\sigma$ is set to the average  $l^2$-norm of the snapshot data after centering it to have zero mean. For the three-year forecast in the left panel of \cref{fig:sea_ice_forecast2}, we use a Lorentzian radial-basis-function dictionary; for the binary accuracy problem, we use a poly-exponential dictionary with up to 60 adaptively chosen time delays.

A key strength of our algorithm is its rigorous error bounds (e.g., \cref{fig:EDMD_not_converge,fig:other_examples}), which enable direct evaluation of dictionary performance by tracking the error—without needing held-out data or forecasts. This allows us to verify that our chosen dictionary yields small errors when approximating the Koopman spectrum. Future work will explore integrating neural network embeddings \cite{lusch2018deep} with our error-bound framework.

To compute the periodic benchmark in \cref{fig:sea_ice_forecast2}, let $I_k$ denote the time indices for month $k$ ($=1,2,\ldots,12$) in the training data when we make a forecast. For forecasting month $k$, $\mathbf{x}_{\mathrm{per}}$ is the average of $\mathbf{x}$ over the times in $I_k$. The difference $\mathbf{x}-\mathbf{x}_{\mathrm{per}}$ then represents the sea ice anomaly. Error metrics are computed over active grid points defined separately for each calendar month. This region expands in winter and contracts in summer according to the maximum observed sea ice extent for that month, following \cite{andersson2021seasonal}. We define $\sigma^2$ as the average value of $\|\mathbf{x}-\mathbf{x}_{\mathrm{per}}\|_{l^2}^2$, evaluated over the active grid points, at a lead time of one month and averaged over the period 2005–2015. A one-month lead time is used to avoid the trivial zero-error baseline at initialization.
Letting $\mathbf{x}$ be the true sea ice concentration and $\mathbf{x}_{\mathrm{rec}}$ the forecast, we define the relative error shown in \cref{fig:sea_ice_forecast2} (a) as
$
\mathrm{Error} = {\|\mathbf{x}-\mathbf{x}_{\mathrm{rec}}\|^2_{l^2}}\big/\sigma^2,
$
where the $l^2$ norm is computed over the active grid points (and averaged over the initializations in the years 2005--2015).
This normalization removes the trivially predictable seasonal cycle, enabling a more meaningful assessment of forecast skill. We also include a monthly persistence model as a baseline for comparison, in which the sea ice concentration is forecast by repeating its value from the previous year for the same calendar month.

To compute forecasts, we get rid of spurious modes in the decomposition in \cref{KMD_example} (where ${\bf g}$ is the vector of sea ice concentrations $\bf x$). The evolution is predicted forward in time for $\bf x$ using approximate eigenfunctions $\phi_\varepsilon^{(j)}$ for errors $\varepsilon$ below a threshold $\varepsilon_0$:
\begin{equation}\setlength\abovedisplayskip{6pt}\setlength\belowdisplayskip{6pt}
\label{KMD_example2}
{\bf x}(n)=\sum_{\varepsilon\leq \varepsilon_0} \lambda_j^n \phi_\varepsilon^{(j)}(0) {\bf g}_j.
\end{equation}
Here, the parameter $j$ ranges from $1$ to $M$ (corresponding to the space generated by the radial basis functions). The errors associated with the approximate eigenfunctions $\phi_\varepsilon$ are ordered by $\varepsilon$, and we identify the ``elbow'' in the error curve to determine a principled truncation point (Supplementary Figure 12). The Koopman modes ${\bf g}_j$ correspond to the vector ${\bf x}$ of observables. After this model has been built, forecasts are produced by increasing $n$. The DMD forecast in \cref{fig:sea_ice_forecast2} (b) is computed in the same manner, but now does not get rid of spurious modes and uses the augmented state space $\widehat{\bf{x}}$ as the dictionary of observables. IceNet and SEAS5 data in \cref{fig:sea_ice_forecast2} are taken directly from \cite{andersson2021seasonal}.

\subsubsection*{Lower bounds: The method of adversarial systems}

\begin{figure}
\centering
\includegraphics[width=1\linewidth,trim={2mm 2mm 2mm 2mm},clip]{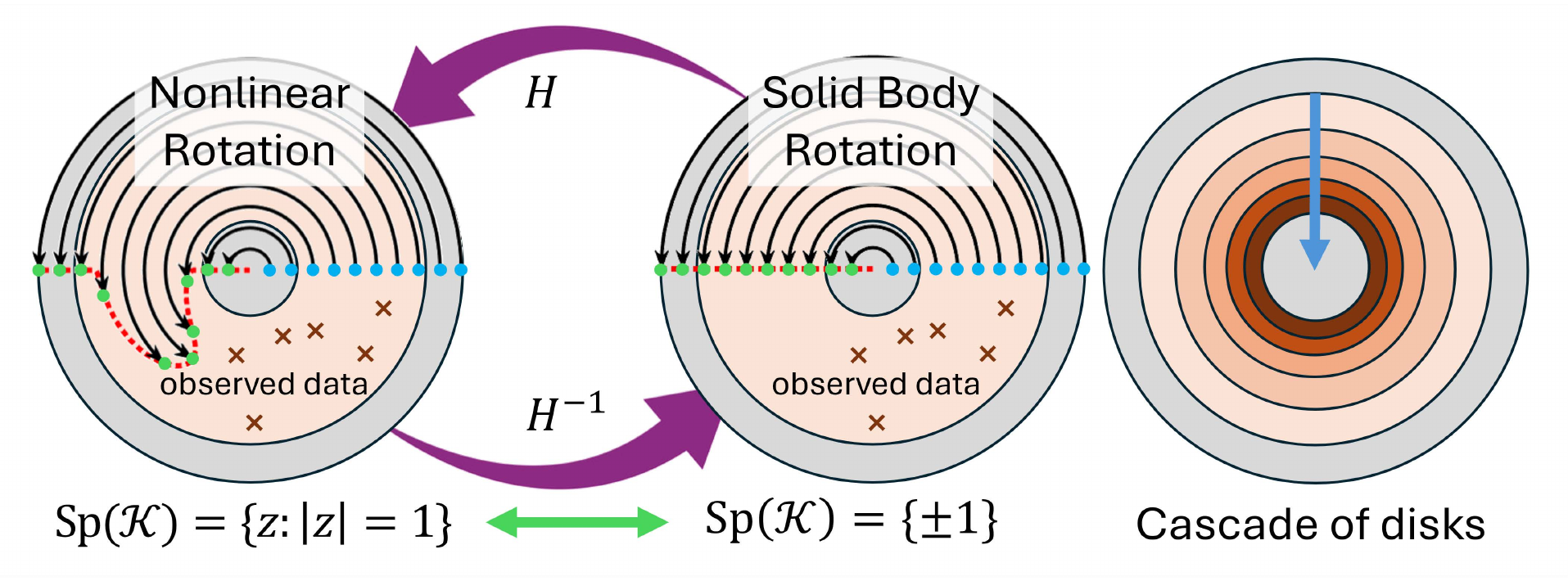}
\caption{\scriptsize\textbf{Proof idea of the impossibility result for $\Omega_{\mathbb{D}}$.} 
At each stage, we modify the system consistently with the observed data (``x''), ensuring it is related to a rotation, thereby drastically altering the spectrum (see green arrow). This alteration is executed such that the cascade of dynamical systems converges to an underlying limit, providing the adversarial family of systems.}
\label{fig:proof1}
\end{figure}

We establish each lower bound (impossibility result) using simple examples of state spaces \(\mathcal{X}\) (e.g., disk, interval, or torus) where \(\omega\) represents the standard Lebesgue measure.  The techniques are general and can be extended to other state spaces \(\mathcal{X}\) and function spaces. Snapshot data is subject to noise and finite precision. To establish robust results, we use a measurement device $\mathcal{T}_F$ that allows arbitrarily accurate sampling:
\begin{equation}\setlength\abovedisplayskip{6pt}\setlength\belowdisplayskip{6pt}
\mathcal{T}_F=\left\{\hat{y}_{j,n}\in\mathcal{X}:d_\mathcal{X}(F(\hat{x}_j),\hat{y}_{j,n})\leq 2^{-n},n\in\mathbb{N}\right\},
\end{equation}
where \(\{\hat{x}_j\}_{j=1}^\infty\) is a dense subset of the metric space \((\mathcal{X}, d_{\mathcal{X}})\), corresponding physically to ``measurement points''. This formulation assumes that measurements can approximate the mapping \(F\) with arbitrarily high precision. This is a strong assumption, and hence it allows us to derive correspondingly strong impossibility results, highlighting fundamental limits even under idealized conditions. Limitations under idealized conditions imply that they hold under more realistic ones. (It can also be significantly relaxed for our upper bounds.)

To establish the lower bounds, we construct families of adversarial dynamical systems. These systems are carefully designed to embed sudden changes in the spectral properties of the Koopman operator $\mathcal{K}_F$ directly into the dynamics, while remaining consistent with the observed trajectory data.

\cref{fig:proof1} illustrates this idea for the class $\Omega_{\mathbb{D}}$. The construction uses a homeomorphism (a continuous, reversible transformation that stretches or bends a space without tearing, gluing, or creating holes) to deform the map while preserving the sampled data. This sudden change lemma (Lemma 2.7 in the Supplementary Information) alters the spectral behavior of the Koopman operator in a controlled way. By applying this construction recursively on a cascade of nested disks, we effectively force any proposed algorithm to fail to converge. Each adversarial construction is tailored to a specific sudden change lemma, with full details provided in the Supplementary Information. This strategy bridges computational techniques with classical ergodic theory, offering a framework for analyzing dynamical systems. Moreover, the method is flexible and can be adapted to tackle a broader range of problems beyond those considered in this study.

To prove a learning problem cannot be solved in one limit (i.e., it has $\mathrm{SCI} > 1$), we assume, for contradiction, that a convergent sequence of algorithms exists. We then construct an adversarial family that causes the algorithm to fail, ensuring convergence occurs with probability no greater than 50\%. This involves tricking the algorithm into oscillating between two different outputs as more trajectory data is collected (see the green arrow in \cref{fig:proof1}). To prove that a problem cannot be solved in two limits (i.e., $\mathrm{SCI} > 2$), we embed complex combinatorial problems into the system's dynamics. These problems, involving sets of numbers with an inherent complexity in their description, are embedded into the dynamics, lifting the lower bound from combinatorics to dynamics.

\vspace{2mm}
\noindent{}\textbf{Data availability:} The observational sea ice concentration data are provided by OSI-SAF (\textcolor[rgb]{0,0,1}{\url{https://osi-saf.eumetsat.int/products/sea-ice-products}}). The results of IceNet and SEAS5 in \cref{fig:sea_ice_forecast2} are reported in \cite{andersson2021seasonal}. All other datasets are produced by the code listed below.\vspace{2mm}

\noindent{}\textbf{Code availability:} Code for the examples of this paper can be found at \textcolor[rgb]{0,0,1}{\url{https://github.com/MColbrook/Adversarial-Dynamical-Systems}}.\vspace{2mm}


\noindent{}\textbf{Acknowledgements}\\
M.C. would like to thank Gary Froyland, Nathan Kutz, and Geoffrey Vasil for discussions. We would also like to thank Benjamin Erichson, James Hogg, Paula Lorenzo Sánchez and Valerio Lucarini for discussions about the sea ice example. We would like to thank Hassan Arbabi for providing the cavity flow dataset.\vspace{2mm}

\noindent{}\textbf{Funding statement}\\
M.C. was supported by INT/UCAM grant LEAG/929. I.M. was supported by ONR grant N00014-21-1-2384 and AFOSR award number FA9550-22-1-0531.\vspace{2mm}

\noindent{}\textbf{Author contributions}\\
M.J.C. conceived and designed the research, developed the methodology, proved the theorems, performed the numerical experiments, interpreted the results, and wrote and revised the manuscript and Supplementary Information. I.M. contributed through proposing dynamical systems useful in the adversarial lemmas and examples, discussion of the results and exposition, including connections with ergodic theory, contributions to the text, and review and editing of the manuscript. A.S. contributed to early-stage discussions and preliminary work on the transition lemmas.\vspace{2mm}

\noindent{}\textbf{Competing interests statement}\\
We declare no competing interests. \vspace{2mm}

\onecolumn

\clearpage

\setcounter{equation}{0}
\renewcommand{\theequation}{S\arabic{equation}}
\renewcommand{\theHequation}{supp.\arabic{equation}}

\crefname{equation}{Supplementary Equation}{Supplementary Equations}
\Crefname{equation}{Supplementary Equation}{Supplementary Equations}

\setcounter{figure}{0}
\renewcommand{\figurename}{Supplementary Figure}
\renewcommand{\thefigure}{S\arabic{figure}}
\renewcommand{\theHfigure}{supp.\arabic{figure}}

\crefname{figure}{Supplementary Figure}{Supplementary Figures}
\Crefname{figure}{Supplementary Figure}{Supplementary Figures}

\setcounter{table}{0}
\renewcommand{\tablename}{Supplementary Table}
\renewcommand{\thetable}{S\arabic{table}}
\renewcommand{\theHtable}{supp.\arabic{table}}

\crefname{table}{Supplementary Table}{Supplementary Tables}
\Crefname{table}{Supplementary Table}{Supplementary Tables}

\setcounter{algorithm}{0}
\floatname{algorithm}{Supplementary Algorithm}
\renewcommand{\thealgorithm}{S\arabic{algorithm}}

\crefname{algorithm}{Supplementary Algorithm}{Supplementary Algorithms}
\Crefname{algorithm}{Supplementary Algorithm}{Supplementary Algorithms}

\section*{Appendix of Supplementary Material}

This section provides further details for constructing our adversarial dynamical systems and classifying problems. We introduce the Solvability Complexity Index (SCI), which, as discussed in the main text, provides a general framework across data-driven dynamical system problems, and consider Koopman operators and their spectral properties. We end this section by discussing the setup for the theorems.

\subsection{Koopman operators -- linearising nonlinear systems}

Throughout, we consider discrete-time dynamical systems:
\begin{equation}
\label{eq:DynamicalSystem} 
x_{n+1} = F(x_n), \qquad n= 0,1,2,\ldots.
\end{equation}
Here, $x\in\mathcal{X}$ denotes the state of the system, and the metric space $(\mathcal{X},d_{\mathcal{X}})$ denotes the state space. Often, $\mathcal{X}\subset\mathbb{R}^d$, though this is not required in what follows. The function $F:\mathcal{X} \rightarrow \mathcal{X}$ governs the evolution of the dynamical system and is generally nonlinear. (Here, we mean that $F$ is linear if $\mathcal{X}$ is a vector space and $F$ is a linear map. Otherwise, we call $F$ nonlinear.) We lift the system \eqref{eq:DynamicalSystem} into a (typically infinite-dimensional) vector space of observable functions using a Koopman operator to deal with the nonlinearity.
A Koopman operator \cite{koopman1931hamiltonian,koopman1932dynamical} is defined on a Banach space $\mathcal{J}$ of functions $g: \mathcal{X} \rightarrow \mathbb{C}$, where the functions $g\in\mathcal{J}$ are referred to as \textit{observables} and measure the state of the system.
Koopman operators allow us to study the evolution of observables in $\mathcal{J}$ through a linear framework.
The Koopman operator is defined via the composition formula:
\begin{equation} 
[\mathcal{K}g](x) = [g\circ F](x)=g(F(x)), \qquad g\in \mathcal{D}(\mathcal{K}),
\label{eq:KoopmanOperator2} 
\end{equation}
where $\mathcal{D}(\mathcal{K}) \subset \mathcal{J}$ is a suitable domain. This definition means that $[\mathcal{K}g](x_n)= g(F(x_n))=g(x_{n+1})$ represents the measurement of the state one time-step ahead of $g(x_n)$, and hence that $\mathcal{K}$ effectively captures the dynamic progression of the system.

The critical property of the Koopman operator $\mathcal{K}$ is its \textit{linearity}. This linearity holds irrespective of whether the map $F$ in \cref{eq:DynamicalSystem} is linear or nonlinear. Consequently, the spectral properties of $\mathcal{K}$ become a powerful tool in analyzing the dynamical system's behavior. The Koopman operator is not defined uniquely by the dynamical system in \eqref{eq:DynamicalSystem}, but fundamentally depends on the space $\mathcal{J}$. Throughout this paper, we focus on the choice
$$
\mathcal{J}=\mathcal{D}(\mathcal{K})=L^2(\mathcal{X},\omega)\quad \text{with inner product}\quad \langle g_1,g_2 \rangle=\int_{\mathcal{X}} g_1(x)\overline{g_2(x)}\ \mathrm{d}\omega(x)\quad \text{and norm}\quad \|g\|=\sqrt{\langle g,g \rangle},
$$
for some positive measure $\omega$. We do not assume that this measure is invariant. For Hamiltonian systems, a common choice of $\omega$ is the standard Lebesgue measure, for which the Koopman operator is unitary on $L^2(\mathcal{X},\omega)$. For other systems, we can select $\omega$ according to the region where we wish to study the dynamics, for example, by using a Gaussian density. In many applications, $\omega$ corresponds to an unknown ergodic measure on an attractor.

In going from a pointwise definition in \eqref{eq:KoopmanOperator2} to the space $L^2(\mathcal{X},\omega)$, care is needed since $L^2(\mathcal{X},\omega)$ consists of equivalence classes of functions. We assume that the map $F$ is nonsingular with respect to $\omega$, meaning that
$$
\omega(E)=0\quad \text{implies that}\quad \omega(\{x:F(x)\in E\})=0.
$$
This condition ensures that the Koopman operator is well-defined since $g_1(x)=g_2(x)$ for $\omega$-almost every $x$ implies that $g_1(F(x))=g_2(F(x))$ for $\omega$-almost every $x$. The pushforward measure is defined as $F\#\omega(E)=\omega(F^{-1}(E))$, and the fact that $F$ is nonsingular with respect to $\omega$ is equivalent to saying that $F\#\omega$ is absolutely continuous with respect to $\omega$. We assume that $\mathcal{K}$ is a bounded linear operator on the Hilbert space $L^2(\mathcal{X},\omega)$. This assumption is equivalent to saying that the Radon--Nikodym derivative $dF\#\omega/d\omega$ lies in $L^\infty(\mathcal{X},\omega)$. The above Hilbert space setting is standard in the Koopman literature, though our results can be extended to other function spaces such as those studied in \cite{mezic2020spectrum}. Once $(\mathcal{X},\omega)$ are specified, we let $\mathcal{K}_F$ denote the corresponding Koopman operator on the corresponding Hilbert space $L^2(\mathcal{X},\omega)$.

Since $\mathcal{K}_F$ acts on an \textit{infinite-dimensional} function space, we have exchanged the nonlinearity in \eqref{eq:DynamicalSystem} for an infinite-dimensional linear system. This means that the spectral properties of $\mathcal{K}_F$ can be significantly more complex than those of a finite matrix, making them more challenging to compute. A message of this paper is that, in most cases, unless strong assumptions are made regarding the system, the spectral properties of $\mathcal{K}_F$ are impossible to compute in a single limit, even if we had a perfect measurement device to sample trajectories of the dynamical system.

\subsubsection{Koopman spectrum}

If $g\in L^2(\mathcal{X},\omega)$ is an \textit{eigenfunction} of $\mathcal{K}_F$ with \textit{eigenvalue} $\lambda$, then $g$ exhibits perfect coherence\footnote{Coherence here is meant in the sense of an observation for which all the points in state space exhibit the same, (complex) exponential, time-dependence.} since
\begin{equation}
\label{eq:perfectly_coherent}
g(x_n)=[\mathcal{K}_F^ng](x_0)=\lambda^n g(x_0)\quad \forall n\in\mathbb{N}.
\end{equation}
The oscillation and decay/growth of the observable $g$ are dictated by the complex argument and absolute value of the eigenvalue $\lambda$, respectively. In infinite dimensions, the appropriate generalization of the set of eigenvalues of $\mathcal{K}_F$ is the \textit{spectrum}:
$$
\spec(\mathcal{K}_F)=\left\{z\in\mathbb{C} :(\mathcal{K}_F - zI)^{-1}\text{ does not exist as a bounded operator}\right\}\subset\mathbb{C}.
$$
Here, $I$ denotes the identity operator. In contrast to finite matrices, the spectrum $\spec(\mathcal{K}_F)$ may contain points that are not eigenvalues. This phenomenon occurs because there are more ways for $(\mathcal{K}_F - zI)^{-1}$ to not be a linear bounded operator in infinite dimensions than in finite dimensions. For example, the standard Lorenz system on the Lorenz attractor has a Koopman operator with no eigenvalues except $\lambda=1$ \cite{luzzatto2005lorenz}. The \textit{approximate point spectrum} is
$$
\spec_{\mathrm{ap}}(\mathcal{K}_F)=\left\{\lambda\in\mathbb{C}:\exists\{g_n\}_{n\in\mathbb{N}}\subset L^2(\mathcal{X},\omega)\text{ such that }\|g_n\|=1,\lim_{n\rightarrow\infty}\|(\mathcal{K}_F-\lambda I)g_n\|=0\right\}\subset\spec(\mathcal{K}_F)\subset\mathbb{C}
$$
and for $\epsilon>0$, the \textit{approximate point pseudospectrum} is
$$
\spec_{\mathrm{ap},\epsilon}(\mathcal{K}_F)=\left\{\lambda\in\mathbb{C}:\exists\{g_n\}_{n\in\mathbb{N}}\subset L^2(\mathcal{X},\omega)\text{ such that }\|g_n\|=1,\lim_{n\rightarrow\infty}\|(\mathcal{K}_F-\lambda I)g_n\|\leq\epsilon\right\}\subset\spec_\epsilon(\mathcal{K}_F)\subset\mathbb{C}.
$$
An observable $g$ with $\|g\|=1$ and $\|(\mathcal{K}_F-\lambda I)g\|\leq\epsilon$ for $\lambda\in\mathbb{C}$ is known as an $\epsilon$-approximate eigenfunction. Such observables satisfy
$$
\|\mathcal{K}_F^ng-\lambda^n g\|= \mathcal{O}(n\epsilon)\text{ as }\epsilon\downarrow0\quad \forall n\in\mathbb{N}.
$$
In other words, $\lambda$ describes an \textit{approximate} coherent oscillation and decay/growth of the observable $g$ with time. The approximate eigenfunctions and $\spec_{\mathrm{ap}}(\mathcal{K}_F)$ encode information about the underlying dynamical system \cite{mezicAMS}. For example, the level sets of certain eigenfunctions determine ergodic partitions \cite{mezic1994geometrical,budivsic2012applied}, invariant manifolds \cite{mezic2015applications}, isostables \cite{mauroy2013isostables}, and the global stability of equilibria \cite{mauroy2016global}.

In this paper, we will focus on the computation of $\spec_{\mathrm{ap}}(\mathcal{K}_F)$. We anticipate that further foundational results can be proven on the computation of other spectral properties of Koopman operators, such as spectral type. For example, see \cref{thm:torus} regarding the detection of non-trivial eigenfunctions.

Two special classes of Koopman operators are defined as follows:
\begin{itemize}
	\item Measure-preserving systems: The dynamical system preserves $\omega$ if and only if $\mathcal{K}_F$ is an isometry, that is $\mathcal{K}_F^*\mathcal{K}_F=I$.
	\item Measure-preserving invertible systems: The dynamical system preserves $\omega$ and is invertible modulo $\omega$-null sets \cite[Chapter 7]{eisner2015operator} if and only if $\mathcal{K}_F$ is unitary, that is $\mathcal{K}_F^*\mathcal{K}_F=\mathcal{K}_F\mathcal{K}_F^*=I$. 
\end{itemize}
The Wold--von Neumann decomposition \cite[Theorem I.1.1]{nagy2010harmonic} states that any isometry on a Hilbert space is unitarily equivalent to the direct sum of a unitary operator and a direct sum of unilateral shifts. Hence, if $\mathcal{K}_F$ is an isometry but not unitary, then the spectrum of $\mathcal{K}_F$ is the unit disk and $\spec_{\mathrm{ap}}(\mathcal{K}_F)=\mathbb{T}=\{z\in\mathbb{C}:|z|=1\}$. If $\mathcal{K}_F$ is unitary, then the spectrum of $\mathcal{K}_F$ is equal to $\spec_{\mathrm{ap}}(\mathcal{K}_F)$ and is a subset of $\mathbb{T}$.

\subsection{The Solvability Complexity Index -- classifying the difficulty of problems}

We now outline the fundamentals of the Solvability Complexity Index Hierarchy. This tool allows us to precisely classify the difficulty of computational problems and prove that algorithms are optimal, realizing the boundaries of what is possible. We give the definitions of the hierarchy before specializing to the computational setup of this paper, where we give a precise formulation of a \textit{perfect measuring device}, which we use to prove our lower bounds (impossibility results).

\subsubsection{Computational problems and general algorithms}

Before classifying the difficulty of computational problems, we must precisely define a computational problem. Precision here is essential since altering the information an algorithm is permitted to use can significantly affect the difficulty of a problem or even whether the problem can be solved at all. The following definition of a computational problem is deliberately general, designed to encompass all types of problems encountered in computational mathematics. For example, as well as spectral problems, the SCI hierarchy has been applied to other areas of mathematics, including PDEs \cite{colbrook2022computing,becker2020computing}, the limits of AI, and Smale's 18th problem \cite{colbrook2022difficulty}, and optimization \cite{SCI_optimization}.

\begin{definition}[Computational problem]
\label{def:comp_prob}
The basic objects of a computational problem are:
\begin{itemize}
	\item A \textit{primary set}, $\Omega$, that describes the input class;
	\item A \textit{metric space} $(\mathcal{M},d)$;
	\item A \textit{problem function} $\Xi:\Omega\rightarrow\mathcal{M}$;
	\item An \textit{evaluation set}, $\Lambda$, of functions on $\Omega$.
\end{itemize}
The problem function $\Xi$ is the object we want to compute, with the notion of convergence captured by the metric space $(\mathcal{M},d)$. The evaluation set $\Lambda$ describes the information that algorithms can read. We require that $\Lambda$ separates elements of $\Omega$ to the degree of separation achieved by $\Xi$:
\begin{equation}
\label{eq:Lambda_determines_it}
\text{if }A,B\in\Omega\text{ with }\Xi(A)\neq \Xi(B),\text{ then } \exists f\in\Lambda\text{ with }f(A)\neq f(B).
\end{equation}
In other words, any $\Xi(A)\in\mathcal{M}$ is uniquely determined by the set of evaluations $\{f(A):f\in\Lambda\}$ (otherwise it is impossible to recover $\Xi$ from $\Lambda$). We refer to the collection $\{\Xi,\Omega,\mathcal{M},\Lambda\}$ as a \textit{computational problem}.
\end{definition}

\begin{example}
In the above setting of dynamical systems, we fix the metric space $(\mathcal{X},d_{\mathcal{X}})$ and the measure $\omega$. The primary set $\Omega$ could be a class of functions $F:\mathcal{X}\rightarrow\mathcal{X}$, each of which induces a dynamical system with bounded $\mathcal{K}_F$ on $L^2(\mathcal{X},\omega)$. The problem function could describe a spectral property of $\mathcal{K}_F$. For example, we could consider the computation of the approximate point spectrum with $\Xi(F)=\spec_{\mathrm{ap}}(\mathcal{K}_F)$. Since $\mathcal{K}_F$ is bounded, $\Xi(F)$ is a compact, non-empty subset of $\mathbb{C}$ \cite[Proposition VII.6.7]{conway2019course}. Hence, we let $(\mathcal{M},d)$ be the Hausdorff metric space, $(\MH,\dH)$, which is the collection of non-empty compact subsets of $\mathbb{C}$ equipped with the Hausdorff metric:
$$
\dH(X,Y) = \max\left\{\sup_{x \in X} \inf_{y \in Y} |x-y|, \sup_{y \in Y} \inf_{x \in X} |x-y| \right\},\quad X,Y\in\MH.
$$
Convergence to the spectrum in this metric means our algorithms converge without spectral pollution (persistent spurious eigenvalues) or spectral invisibility (missing parts of the spectrum). As our evaluation set, we could consider maps $f_x(F)=F(x)$ for $x\in\mathcal{X}$ or a subset of such $x$. For example, if $\mathcal{X}\subset\mathbb{R}^d$, then $f_x$ is a real vector-valued function.
\end{example}

With the definition of a computational problem $\{\Xi,\Omega,\mathcal{M},\Lambda\}$ established, we now define what we mean by an algorithm. An algorithm is a function
$
\Gamma:\Omega\to \mathcal{M}
$
that, unlike the problem function $\Xi$, utilizes the evaluation set $\Lambda$ in some manner. The specifics of how $\Lambda$ is used (or even which sets of $\Lambda$ are permitted) depend on the computational model. We adopt a general definition for proving lower bounds (impossibility results). This approach not only yields stronger results but also significantly simplifies the proofs. Specifically, we aim to establish lower bounds that are valid in \textit{any model of computation}.

\begin{definition}[General algorithm]
\label{def:Gen_alg}
A general algorithm for a computational problem $\{\Xi,\Omega,\mathcal{M},\Lambda\}$ is a map $\Gamma:\Omega\to \mathcal{M}$ with the following property. For any $A\in\Omega$, there exists a non-empty finite subset of evaluations $\Lambda_\Gamma(A) \subset\Lambda$ such that if $B\in\Omega$ with $f(A)=f(B)$ for every $f\in\Lambda_\Gamma(A)$, then $\Lambda_\Gamma(A)=\Lambda_\Gamma(B)$ and $\Gamma(A)=\Gamma(B)$.
\end{definition}

\cref{def:Gen_alg} outlines the most fundamental properties of any reasonable \textit{deterministic} computational device. It says that $\Gamma$ can only use a finite amount of information, though it can adaptively choose this information as it processes the input. Moreover,  the output of $\Gamma$ depends consistently on the information it has accessed. Specifically, if $\Gamma$ sees the same information for two different inputs, it must behave identically for those two inputs. A general algorithm has no restrictions on the operations allowed. It is more powerful than a Turing machine \cite{turing1937computable} or BSS machine\footnote{One should think of a Blum--Shub--Smale (BSS) machine as akin to an algorithm that deals with exact arithmetic. They are a model of computation designed to work over any ring or field, most notably over the real numbers. This distinguishes them from Turing machines, which are based on discrete values. A BSS machine extends the concept of computation to include computation with real numbers and other continuous data.} \cite{BCSS} and serves two main purposes:
\begin{itemize}
\item[(i)] {\it A focus on what really matters:} \cref{def:Gen_alg} significantly simplifies the process of proving lower bounds. The non-computability results we present stem from the intrinsic non-computability of the problems themselves, not from the type of operations allowed being too restrictive. Specifically, the limitation lies in the algorithmic input $\Lambda$ being inadequate for solving the problem.
\item[(ii)] {\it Classifications with the strongest possible lower and upper bounds:} The generality of \cref{def:Gen_alg} implies that a lower bound established for general algorithms also applies to any computational model. Furthermore, the algorithms we provide can be executed using only arithmetic operations, both in the Turing and BSS models. Therefore, we derive the strongest possible lower and upper bounds simultaneously.
\end{itemize}

\subsubsection{Towers of algorithms}

Having established precise definitions for a computational problem and a general algorithm, we now introduce the concept of a tower of algorithms. This captures the observation in the main text that algorithms for data-driven Koopmanism depend on several parameters that must be taken to successive limits to ensure convergence.

\begin{definition}[Tower of algorithms]
\label{def:tower_funct}
Let $k\in\mathbb{N}$.
A tower of algorithms of height $k$ for a computational problem $\{\Xi,\Omega,\mathcal{M},\Lambda\}$ is a collection of functions
$$
\Gamma_{n_k},\,
\Gamma_{n_k, n_{k-1}},\, \ldots,\,\Gamma_{n_k, \ldots, n_1}:\Omega \rightarrow \mathcal{M}, \quad n_k,\ldots,n_1 \in \mathbb{N},
$$
where $\{\Gamma_{n_k, \ldots, n_1}\}$ are general algorithms (\cref{def:Gen_alg}) and for every $A \in \Omega$, the following convergence holds in $(\mathcal{M},d)$:
$$
\Xi(A) = \lim_{n_k \rightarrow \infty} \Gamma_{n_k}(A), \quad\Gamma_{n_k}(A) =
  \lim_{n_{k-1} \rightarrow \infty} \Gamma_{n_k, n_{k-1}}(A),\quad\ldots\quad
\Gamma_{n_k, \ldots, n_2}(A) =
  \lim_{n_1 \rightarrow \infty} \Gamma_{n_k, \ldots, n_1}(A).
$$
We shall use the term ``tower'' even if $k=1$.
\end{definition}

When we prove upper bounds (i.e., provide algorithms that solve a problem), we can specify the type of tower by imposing conditions on the functions $\{\Gamma_{n_k, \ldots, n_1}\}$ at the lowest level. In essence, the type is the toolbox allowed:
\begin{itemize}
\item A \textit{general tower}, denoted by $\alpha=G$, refers to \cref{def:tower_funct} with no further restrictions.
\item An \textit{arithmetic tower}, denoted by $\alpha=A$, refers to \cref{def:tower_funct} where each $\Gamma_{n_k, \ldots, n_1}(A)$ can be computed using $\Lambda$ and finitely many arithmetic operations and comparisons. More precisely, if $\Lambda$ is countable, each output $\Gamma_{n_k, \ldots, n_1}(A)$ is a finite string of numbers (or encoding) that can be identified with an element in $\mathcal{M}$, and the following function is recursive:
$
(n_k, \ldots, n_1,\{f(A)\}_{f \in \Lambda})\mapsto \Gamma_{n_k, \ldots, n_1}(A).
$
\end{itemize}

We can now define the Solvability Complexity Index (SCI).

\begin{definition}[Solvability Complexity Index]
\label{def:complex_ind}
A computational problem $\{\Xi,\Omega,\mathcal{M},\Lambda\}$ has Solvability Complexity Index $k\in\mathbb{N}$ with respect to type $\alpha$, written $\mathrm{SCI}(\Xi,\Omega,\mathcal{M},\Lambda)_{\alpha} = k$, if $k$ is the smallest integer for which there exists a tower of algorithms of type $\alpha$ and height $k$ that solves the problem. If no such tower exists, then $\mathrm{SCI}(\Xi,\Omega,\mathcal{M},\Lambda)_{\alpha} = \infty.$ If there exists an algorithm $\Gamma$ of type $\alpha$ with $\Xi = \Gamma$, then $\mathrm{SCI}(\Xi,\Omega,\mathcal{M},\Lambda)_{\alpha} = 0$.
\end{definition}

The SCI induces the SCI hierarchy as follows.

\begin{definition}[SCI hierarchy]
\label{def:1st_SCI}
Consider a collection $\mathcal{C}$ of computational problems and let $\mathcal{T}_\alpha$ be the collection of all towers of algorithms of type $\alpha$. We define the following subclasses of $\mathcal{C}$: 
\begin{align*}
\Delta^{\alpha}_0 &= \{\{\Xi,\Omega,\mathcal{M},\Lambda\} \in \mathcal{C} :   \mathrm{SCI}(\{\Xi,\Omega,\mathcal{M},\Lambda\})_{\alpha} = 0\},\\
\Delta^{\alpha}_{1} &= \{\{\Xi,\Omega,\mathcal{M},\Lambda\}  \in \mathcal{C}  : \exists \{\Gamma_n\}_{n=1}^\infty \in \mathcal{T}_\alpha\text{ s.t. } \forall A\in\Omega, d(\Gamma_n(A),\Xi(A)) \leq 2^{-n}\},\\
\Delta^{\alpha}_{m+1} &= \{\{\Xi,\Omega,\mathcal{M},\Lambda\}  \in \mathcal{C} :   \mathrm{SCI}(\{\Xi,\Omega,\mathcal{M},\Lambda\})_{\alpha} \leq m\},  \quad\text{ for } m \in \mathbb{N}.
\end{align*}
\end{definition}

In summary, a $\Delta^{\alpha}_{m+1}$ problem can be computed in $m$ successive limits, and a $\Delta^{\alpha}_{1}$ problem can be computed in one limit with complete error control. The $2^{-n}$ in the definition of $\Delta^{\alpha}_{1}$ is arbitrary; replacing it with any sequence converging to zero that is computable from $\Lambda$ does not alter the definition.

\subsubsection{Inexact input}
So far, we have only discussed algorithms with exact input from the evaluation set $\Lambda$. However, in practice, we may only have access to input of a certain accuracy. Suppose we are given a computational problem $\{\Xi, \Omega, \mathcal{M}, \Lambda\}$ with evaluation set $\Lambda = \{f_j:\Omega\rightarrow\mathcal{M}_\Lambda\}_{j \in \mathcal{I}}$, for some index set $\mathcal{I}$ and metric space $(\mathcal{M}_\Lambda,d_\Lambda)$. Obtaining $f_j$ may be a computational task in its own right. For instance, $f_j(A)$ could be the number $\cos(\sin(1))$ or an inner product that is approximated using quadrature. Alternatively, it may be the case that we can only measure $f_j(A)$ to a certain accuracy due to effects such as noise. In the context of Koopman operators, any physical measurement device will have a non-zero measurement error.

We may view the problem of obtaining $f_j(A)$ as a problem in the SCI hierarchy. A $\Delta_1$-classification corresponds to access to $f_{j,n}: \Omega \rightarrow \mathcal{M}_\Lambda$
such that 
\begin{equation}\label{eq:Lambda_limits2}
 d_\Lambda(f_{j,n}(A),f_j(A)) \leq 2^{-n} \quad \forall A \in \Omega.
 \end{equation}
We want algorithms that can handle all possible choices of such inexact input. We can make this precise by replacing the class $\Omega$ by the class of suitable evaluation functions $\{f_{j, n_1}(A)\}_{j, n_1 \in \mathcal{I} \times \mathbb{N}}$ that satisfy \cref{eq:Lambda_limits2}. This viewpoint is well-defined since \cref{eq:Lambda_determines_it} holds. The following definition captures the notion of algorithms being robust to noise in input data.

\begin{definition}[Computational problems with $\Delta_1$-information]
\label{def:inexact_def_need}
Given a computational problem $\{\Xi,\Omega,\mathcal{M},\Lambda\}$, the corresponding computational problem with $\Delta_{1}$-information is denoted by 
$
\{\Xi,\Omega,\mathcal{M},\Lambda\}^{\Delta_{1}} = \{\tilde \Xi,\tilde \Omega,\mathcal{M},\tilde \Lambda\},
$
and defined as follows:
\begin{itemize}[leftmargin=0.3cm]
	\item The primary set $\tilde \Omega$ is the class of tuples
	$
	\tilde A = \left\{f_{j,n_1}(A):j\in\mathcal{I},n_1\in\mathbb{N}\right\},
	$
	where $A\in\Omega$, $\{f_j\}_{j \in \mathcal{I}}$ and \cref{eq:Lambda_limits2} holds;
	\item The problem function is $\tilde \Xi(\tilde A) = \Xi(A)$, which is well-defined by \cref{eq:Lambda_determines_it};
	\item The evaluation set is $\tilde \Lambda = \{\tilde f_{j,n_1}\}_{j,n_1 \in \mathcal{I} \times \mathbb{N}}$, where $\tilde f_{j,n_1}(\tilde A) = f_{j,n_1}(A)$.
\end{itemize}
The SCI hierarchy given $\Delta_1$-information is then defined in an obvious manner.
\end{definition}

\subsubsection{Refinements that capture error control}

When performing numerical computations, particularly in many spectral applications, determining the accuracy of the results is essential. The importance of error bounds extends beyond science and engineering and holds in pure mathematics, especially when using spectral problems in computer-assisted proofs. For instance, even when a discretization method for computing spectra through eigenvalues converges, typically, only a subset of the numerically computed eigenvalues is reliable. Note that such a problem is not computable in the classical Turing sense but instead verifiable. Most infinite-dimensional spectral problems do not lie in $\Delta_1$ \cite{colbrook2020PhD}, but many lie in the following refinements that capture error control \cite{colbrook2019compute,colbrook3}. Classifying when error bounds can or cannot be obtained is a fundamental challenge in dealing with infinite-dimensional spectral problems.

Sufficient structure in $(\mathcal{M},d)$ enables two types of verification or error control: convergence from above and below. In this paper, there are two metric spaces that we use for computational problems $\{\Xi,\Omega,\mathcal{M},\Lambda\}$:
\begin{itemize}
	\item  If $\mathcal{M} = \{0,1\}$ with the discrete topology, we call the problem a decision problem and denote this space by $\Md$. For an input $A\in\Omega$, we interpret the output $\Xi(A)=1$ as ``Yes'' and the output $\Xi(A)=0$ as ``No''. 
	\item The Hausdorff metric space, $(\MH,\dH)$, is suitable for computing spectra of bounded operators. It is the collection of non-empty compact subsets of $\mathbb{C}$ equipped with the Hausdorff metric:
\begin{equation}\label{eq:Hausdorff}
\dH(X,Y) = \max\left\{\sup_{x \in X} \inf_{y \in Y} |x-y|, \sup_{y \in Y} \inf_{x \in X} |x-y| \right\},\quad X,Y\in\MH.
\end{equation}
As noted above, we are interested in the Hausdorff metric since convergence to the spectrum in this metric means that our algorithms converge without spectral pollution or spectral invisibility.
\end{itemize}

We now define notions of error control for these two metric spaces. When $\mathcal{M}$ is a totally ordered set with relation $\leq$, such as $\mathbb{R}$ or $\mathbb{N}$, convergence from above or below is straightforward to define.

\begin{definition}[$\Sigma$ and $\Pi$ classes for totally ordered sets]
\label{def:tot_ord}
Consider a collection $\mathcal{C}$ of computational problems and let $\mathcal{T}_\alpha$ be the collection of all towers of algorithms of type $\alpha$. Suppose $\mathcal{M}$ is a totally ordered set. Set $\Sigma^{\alpha}_0 = \Pi^{\alpha}_0 = \Delta^{\alpha}_0$ and for $m \in \mathbb{N}$, define
\begin{equation*}
\begin{split}
\Sigma^{\alpha}_{m} &{=} \left\{\{\Xi,\Omega,\mathcal{M},\Lambda\} \in \Delta_{m+1}^\alpha :  \exists \ \{\Gamma_{n_{m},\ldots, n_1}\} \in \mathcal{T}_\alpha \text{ s.t. }\Gamma_{n_{m}}(A) \uparrow \Xi(A)  \forall A \in \Omega\right\}, \\
\Pi^{\alpha}_{m} &{=} \left\{\{\Xi,\Omega,\mathcal{M},\Lambda\} \in \Delta_{m+1}^\alpha :  \exists \ \{\Gamma_{n_{m},\ldots, n_1}\} \in \mathcal{T}_\alpha \text{ s.t. }\Gamma_{n_{m}}(A) \downarrow \Xi(A)   \forall A \in \Omega\right\},
\end{split}
\end{equation*}
where $\uparrow$ and $\downarrow$ denote convergence from below and above, respectively. In other words, we have convergence from below or above in the final limit of the tower of algorithms.
\end{definition}

The following two examples discuss $(\mathcal{M},d)=\mathbb{R}$ and $\Md$, respectively.

\begin{example}[Spectral radius of normal operators]
\label{exam_spec_rad}
Let $\Omega$ be the class of bounded normal operators on $l^2(\mathbb{N})$, $\Lambda=\{A\mapsto \langle Ae_j,e_i \rangle:i,j\in\mathbb{N}\}$, and consider the spectral radius problem function
$
\Xi(A)=\sup_{z\in\spec(A)}|z|.
$
For normal operators, $\sup_{z\in\spec(A)}|z|=\|A\|$. Hence, we let $\Gamma_{n}(A)$ be an approximation of $\|\mathcal{P}_nA\mathcal{P}_n^*\|$ to accuracy $1/n$ from below, where $\mathcal{P}_n$ is the orthogonal projection onto $\mathrm{span}\{e_1,\ldots,e_n\}$. $\Gamma_{n}(A)\uparrow \|A\|$ and hence $\{\Xi,\Omega,\mathbb{R},\Lambda\}\in\Sigma_1^A$. This classification means that for any finite $n$, we obtain a lower bound for the value $\Xi(A)$. However, we may not know how close $\Gamma_{n}(A)$ is to $\|A\|$ and one can also show that $\{\Xi,\Omega,\mathbb{R},\Lambda\}\notin\Delta_1^G$. To see why, suppose for a contradiction that $\{\Xi,\Omega,\mathbb{R},\Lambda\}\in\Delta_1^G$. Hence, there exists a general algorithm $\Gamma$ such that $|\Gamma(A)-\|A\||\leq 1$ for all $A\in\Omega$. We may choose $A=\mathrm{diag}(0,\ldots,0,3,3,\ldots)$ such that the number of zeros in the diagonal of $A$ ensures that $\Gamma(A)=\Gamma(0)$ (this follows from the consistency requirement in the definition of a general algorithm). But then $1\geq |\Gamma(A)-3|=|\Gamma(0)-3|\geq 3-|\Gamma(0)-0|\geq 2$, a contradiction. The point is that we cannot compute an upper bound on $\|A\|$ from a finite amount of information, and hence, we cannot get full error control in the metric space $\mathbb{R}$.
\end{example}

\begin{example}[Is the spectral radius strictly larger than one?]
Consider the setup of \cref{exam_spec_rad}, but now let $\Xi$ be the decision problem, `Is $\sup_{z\in\spec(A)}|z|>1$?' Let $\Gamma_{n}(A)$ be as before, then if $\Xi(A)=1$ (yes), $\Gamma_{n}(A)>1$ for some $n$, otherwise $\Gamma_{n}(A)\leq 1$ for all $n$. Note that we have used the fact that $\{\Gamma_n\}$ is a $\Sigma_1^A$-tower for the problem in \cref{exam_spec_rad}. It follows that
$$
\tilde{\Gamma}_n(A)=\begin{cases}
1,\quad&\text{if }\Gamma_{n}(A)>1,\\
0,\quad &\text{otherwise.}
\end{cases}
$$
provides a $\Sigma_1^A$-tower for $\{\Xi,\Omega,\Md,\Lambda\}$. Again, one can show that $\{\Xi,\Omega,\Md,\Lambda\}\notin\Delta_1^G$. If we changed the decision problem to `Is $\sup_{z\in\spec(A)}|z|\leq 1$?', we would obtain a $\Pi_1^A$ classification instead.
\end{example}

More generally, the classes $\Sigma_1^A$ and $\Pi_1^A$ in \cref{def:tot_ord} allow \textit{verification}. For example, suppose that we have a problem function $\Xi:\Omega\rightarrow\mathbb{R}$ and we wish to verify a theorem $\Xi(A)<0$ for some $A\in\Omega$. If there exists a $\Pi_1^A$-tower $\{\Gamma_n\}$ for the problem and the theorem is true for the given $A$, then $\Gamma_n(A)<0$ for sufficiently large $n$. We can compute $\Gamma_n(A)$ for various $n$ and as soon as $\Gamma_n(A)<0$, we know that $\Xi(A)\leq\Gamma_n(A)<0$ and have verified the theorem. Note, however, that we cannot use a $\Pi_1^A$-tower to negate such a theorem (but can, instead, use a $\Sigma_1^A$-tower if it exists). We can only verify one way using a $\Pi_1^A$-tower. Similar remarks hold for $\Sigma_1^A$.

While \cref{def:tot_ord} is straightforward, it does not carry over to the more complicated Hausdorff metric. To define convergence of $\Gamma_n(A)$ to $\Xi(A)$ in the Hausdorff metric ``from below'', a first attempt may be to require that $\Gamma_n(A)\subset\Xi(A)$. However, this is severely restrictive. For example, when computing $\spec(A)\subset\mathbb{C}$, we can rarely ensure that a point $z$ is exactly in $\spec(A)$. Nevertheless, we can often ensure that $z$ is close to $\spec(A)$ and measure how close. Hence, it is natural to relax the condition $\Gamma_n(A)\subset\Xi(A)$ to
$
\sup_{z\in\Gamma_n(A)}\dist(z,\Xi(A))\leq 2^{-n}.
$
The exact form of the sequence $\{2^{-n}\}$ does not matter. What matters is that we can control the proximity of $\Gamma_n(A)$ to being contained within $\Xi(A)$ using a \textit{known} sequence that converges to zero as $n\rightarrow\infty$. Based on this discussion, the following provides the generalization of \cref{def:tot_ord}.

\begin{definition}[$\Sigma$ and $\Pi$ classes for Hausdorff metric]
\label{def_sig_pi_spec}
Consider a collection $\mathcal{C}$ of computational problems and let $\mathcal{T}_\alpha$ be the collection of all towers of algorithms of type $\alpha$. Suppose that $(\mathcal{M},d)$ is the Hausdorff metric. We set $\Sigma^{\alpha}_0 = \Pi^{\alpha}_0 = \Delta^{\alpha}_0$ and for $m \in \mathbb{N}$, we define
\begin{equation*}
\begin{split}
\Sigma^{\alpha}_{m}&= \Big\{\{\Xi,\Omega,\mathcal{M},\Lambda\} \in \Delta_{m+1}^\alpha  :  \exists \{\Gamma_{n_{m},\ldots,n_1}\} \in \mathcal{T}_\alpha,  \{X_{n_{m}}(A)\}\subset\mathcal{M} \text{ s.t. }\forall A \in\Omega\\
&\quad\quad\quad\quad\quad\quad\quad\quad\quad\quad\Gamma_{n_{m}}(A)  \subset X_{n_{m}}(A),\lim_{n_{m}\rightarrow\infty}\Gamma_{n_{m}}(A)=\Xi(A), d(X_{n_{m}}(A),\Xi(A))\leq 2^{-n_{m}}\Big\},\\
\Pi^{\alpha}_{m} &= \Big\{\{\Xi,\Omega,\mathcal{M},\Lambda\} \in \Delta_{m+1}^\alpha  :  \exists \{\Gamma_{n_{m},\ldots,n_1}\} \in \mathcal{T}_\alpha,  \{X_{n_{m}}(A)\}\subset\mathcal{M} \text{ s.t. }\forall A \in \Omega\\
&\quad\quad\quad\quad\quad\quad\quad\quad\quad\quad\Xi(A)  \subset X_{n_{m}}(A),\lim_{n_{m}\rightarrow\infty}\Gamma_{n_{m}}(A)=\Xi(A),
 d(X_{n_{m}}(A),\Gamma_{n_{m}}(A))\leq 2^{-n_{m}} \Big\}.
\end{split}
\end{equation*}
\end{definition}

These classes capture convergence from below or above, up to a small error parameter $2^{-n}$. It is precisely the classes $\Sigma_1^{\alpha}$ and $\Pi_1^{\alpha}$ that allow computations with verification, used, for example, in computer-assisted proofs. For example, to build a $\Sigma_1^{\alpha}$ algorithm in the case of the Hausdorff metric, it is enough to construct a convergent tower $\{\Gamma_n\}$ such that $\Gamma_{n}(A) \subset \Xi(A)+B_{{E_n}}(0)$ with some computable $E_n$ that converges to zero.

\subsubsection{Randomized algorithms}

We also consider sequences of probabilistic general algorithms. Randomized algorithms are commonly used in machine learning and optimization. For example, in the context of Koopman operators, Monte Carlo methods that randomly sample the snapshots are common. We consider the following definition.

\begin{definition}[Coin model]
Let $\mathcal{C} = \{0,1\}^\mathbb{N}$, which we interpret as the set of possible outcomes of a countably infinite number of coin flips. We equip $\mathcal{C}$ with the product topology, where each flip's outcome set, $\{0,1\}$, is given the discrete topology. Let $p_j$ denote the projection map $p_j(a) = a_j$ for $a \in \mathcal{C}$ and $j \in \mathbb{N}$. We define $\mathbb{P}$ as the standard probability measure corresponding to an unbiased coin. That is, for any $n \in \mathbb{N}$ and any $a_1, \ldots, a_n \in \{0,1\}$, we have $\mathbb{P}(\{b\in\mathcal{C}:b_j=a_j,j=1,\ldots,n\})=2^{-n}.$
\end{definition}

When an algorithm is given an input $A$ from an input class $\Omega$, we want it to be able to act on both the outcome of coin flips and $A$. This motivates the following definition of a probabilistic computational problem.

\begin{definition}[Probabilistic computational problem]
\label{def:prob_comp_prob}
Let $\{\Xi,\Omega,\mathcal{M},\Lambda\}$ be a computational problem, as defined in \cref{def:comp_prob}. The corresponding probabilistic computational problem is $\{\Xi^\mathbb{P},\Omega^\mathbb{P},\mathcal{M}^\mathbb{P},\Lambda^\mathbb{P}\}$, where:
\begin{itemize}[leftmargin=0.7cm]
	\item $\Omega^\mathbb{P}=\Omega\times\mathcal{C}=\{(A,a):A\in\Omega,a\in\mathcal{C}\}$;
	\item $\Xi^\mathbb{P}(A,a)=\Xi(A)$ for $(A,a)\in\Omega^\mathbb{P}$;
	\item $\mathcal{M}^\mathbb{P}=\mathcal{M}\cup\{\mathrm{NH}\}$, where $\mathrm{NH}$ is added as an isolated point to $(\mathcal{M},d)$;
	\item $\Lambda^\mathbb{P}=\{f:f(A,a)=f(A),f\in\Lambda\}\cup\{p_j:p_j(A,a)=p_j(a),j\in\mathbb{N}\}$.
\end{itemize}
\end{definition}

The set $\Omega^\mathbb{P}$ is interpreted as an infinite sequence of coin flips for each input $A \in \Omega$. The evaluation set $\Lambda^\mathbb{P}$ consists of the original evaluation set and the ability to read the outcomes of the coin flips. The point $\mathrm{NH}$ added to $\mathcal{M}$ is interpreted as ``non-halting'', meaning that the algorithm never finishes computing to produce an output.

\begin{example}[Why we include $\mathrm{NH}$]
Suppose an algorithm repeatedly flips an unbiased coin. After each flip, if it sees ``heads'', it outputs the answer $1$ and halts. Otherwise, it flips the coin again. With probability $1$, this algorithm will halt and output $1$. However, there is an outcome -- an infinite sequence of tails -- upon which the algorithm does not halt.
\end{example}

\begin{definition}[Probabilistic general algorithm]
\label{def:prob_Gen_alg}
Given a computational problem $\{\Xi,\Omega,\mathcal{M},\Lambda\}$, a probabilistic general algorithm is a map $\Gamma:\Omega^\mathbb{P}\to \mathcal{M}\cup\{\mathrm{NH}\}$ with the following property. For any $(A,a)\in\Omega^\mathbb{P}$, there exists a non-empty subset of evaluations $\Lambda_\Gamma(A,a) \subset\Lambda^\mathbb{P}$ such that:
\begin{itemize}[leftmargin=0.7cm]
	\item If $(B,b)\in\Omega^\mathbb{P}$ with $f(A,a)=f(B,b)$ for every $f\in\Lambda_\Gamma(A,a)$, then $\Lambda_\Gamma(A,a)=\Lambda_\Gamma(B,b)$ and $\Gamma(A,a)=\Gamma(B,b)$;
	\item If $\Gamma(A,a)\neq \mathrm{NH}$, then $\Lambda_\Gamma(A,a)$ is finite.
\end{itemize}
We will refer to a sequence of probabilistic general algorithms as an SPGA.
\end{definition}

One can easily show that if $\Gamma$ is a probabilistic general algorithm for a computational problem $\{\Xi,\Omega,\mathcal{M},\Lambda\}$, then for any fixed $A\in\Omega$, the map $a\mapsto \Gamma(A,a)$ is measurable (and hence a random variable). Hence, we obtain an interpretation of a probabilistic general algorithm as a randomized algorithm. For $A\in\Omega$, we consider the possible outputs $\{\Gamma(A,a):a\in\mathcal{C}\}$ with the probability measure $\mathbb{P}$. For example, given a sequence of probabilistic general algorithms (SPGA) $\{\Gamma_n\}$, we drop $a\in\mathcal{C}$ from the notation when there is no ambiguity and define
$$
\mathbb{P}\left(\lim_{n\rightarrow\infty}\Gamma_n(A)=\Xi(A)\right)=\mathbb{P}\left(\left\{a\in\mathcal{C}:\lim_{n\rightarrow\infty}\Gamma_n(A,a)=\Xi^\mathbb{P}(A,a)\right\}\right).
$$
The conditions of a probabilistic general algorithm hold for any standard probabilistic machine that flips coins (e.g., probabilistic Turing \cite[Ch. 7]{arora2009computational} or BSS \cite[Ch. 17]{BCSS} machines).\footnote{One could also consider other, even continuous, probability distributions. In the case of BSS machines, machines that can pick numbers uniformly at random in $[0,1]$ are no more computationally powerful \cite[Section 17.5]{BCSS}. Hence, we do not consider such scenarios, which are also unrealistic in practice.} 
The critical point is that we have replaced consistency in the output of deterministic algorithms with consistency in the probability law of randomized algorithms.

\begin{example}[Consistency]\label{exa:consistency_prob}
Let $\Gamma$ be a probabilistic general algorithm for a computational problem $\{\Xi,\Omega,\mathcal{M},\Lambda\}$ and fix $A\in\Omega$. Let $S\subset\Lambda^\mathbb{P}$ and $B\in\Omega$ be such that $f(A)=f(B)$ for all $f\in S\cap \Lambda$. Then the consistency requirement in \cref{def:prob_Gen_alg} implies that
$$
\{a\in\mathcal{C}:\Lambda_{\Gamma}(A,a)\subset S\}=\{a\in\mathcal{C}:\Lambda_{\Gamma}(B,a)\subset S\}.
$$
Moreover, for any $a$ in this set, we must have $\Gamma(B,a)=\Gamma(A,a)$
\end{example}

\begin{example}[Random guess for decision problems]
Suppose $\{\Xi,\Omega,\mathcal{M},\Lambda\}$ is a decision problem with $\mathcal{M} = \Md = \{0,1\}$. We can define a probabilistic general algorithm by setting $\Gamma(A, a) = p_1(a)$. If we set $\Gamma_n = \Gamma$, we find that $\mathbb{P}(\lim_{n \rightarrow \infty} \Gamma_n(A) = \Xi(A)) = 1/2$. However, if we instead set $\Gamma_n(A, a) = p_n(a)$, then $\mathbb{P}(\Gamma_n(A) \text{ does not converge in } \Md) = 1$.
\end{example}

\begin{example}[Independent sequence of probabilistic algorithms]
For certain problems, we may want to consider an SPGA $\{\Gamma_n\}$, where the coin flips of $\Gamma_j$ and $\Gamma_k$ are independent for $j \neq k$. This can be easily achieved by expressing the set of positive integers as a disjoint union of countably infinite sets, $\mathbb{N} = \cup_{j=1}^\infty T_j$. We then restrict $\Lambda_{\Gamma_j}$ to be a subset of $\Lambda \cup \{p_k : k \in T_j\}$. Hence, we see that the definition of an SPGA captures computations that allow for independence or dependence between the different $\Gamma_n$'s.
\end{example}

\begin{example}[Random sampling of trajectories]
Suppose we sample snapshots randomly from $\mathcal{X}$ according to a probability distribution. Any random sampling produced by a digital computer and subsequent use of the data to produce an output is an example of a probabilistic general algorithm. Hence, the lower bounds we prove in this paper are universal.
\end{example}

A very useful property of probabilistic general algorithms is the following covering lemma. This tool enables us to reason similarly to deterministic algorithms with only an arbitrarily small loss in probability.

\begin{lemma}[Finite $\epsilon$-covering lemma]\label{lem:prob_covering_lemma}
Let $\Gamma$ be a probabilistic general algorithm for a computational problem $\{\Xi,\Omega,\mathcal{M},\Lambda\}$ and fix $A\in\Omega$. For any $\epsilon>0$, there exists a finite subset $S\subset\Lambda^\mathbb{P}$ such that
$$
\mathbb{P}(\Lambda_{\Gamma}(A,a)\subset S\text{ \rm and }\Gamma(A,a)\neq\mathrm{NH})\geq \mathbb{P}(\Gamma(A,a)\neq \mathrm{NH})-\epsilon.
$$
\end{lemma}

\begin{proof}
For any $S\subset\Lambda^\mathbb{P}$, define the set
$$
U_S=\{a\in\mathcal{C}:\Lambda_{\Gamma}(A,a)\subset S,\Gamma(A,a)\neq\mathrm{NH}\}.
$$
If $a \in U_S$, then there exists an open neighborhood of $a$ such that any $b$ in this neighborhood satisfies $\Lambda_{\Gamma}(A, a) = \Lambda_{\Gamma}(A, b)$. It follows that $U_S$ is open.
If $a\in\mathcal{C}$ has $\Gamma(A,a)\neq \mathrm{NH}$, then $\Lambda_{\Gamma}(A,a)$ is finite. It follows that
$$
\{a\in\mathcal{C}:\Gamma(A,a)\neq \mathrm{NH}\}=\bigcup_{S\subset\Lambda^\mathbb{P},S\text{ finite}}U_S.
$$
Let $U$ denote this set. Since $U$ is a union of open sets, it is open. We shall prove that there exists a countable subcover of $U$.

Let $V\subset \mathcal{C}$ be the set of eventually constant coin flips. That is, the set of all $a\in\mathcal{C}$ such that the sequence $\{p_j(a)\}_{j=1}^\infty$ is eventually constant. The set $V$ is countable and dense in $\mathcal{C}$. Since $U$ is open, $V\cap U=\{v_1,v_2,\ldots\}$ is dense in $U$. For each $j\in\mathbb{N}$ and $n\in\mathbb{N}$, consider the open ball $D_{1/n}(v_j)$. If there exists a finite set $S\subset\Lambda^\mathbb{P}$ with $D_{1/n}(v_j)\subset U_S$, let $U_{j,n}=U_S$ for a choice of such an $S$. Otherwise, let $U_{j,n}=\emptyset$. Let $a\in U$. There exists a finite set $S\subset\Lambda^\mathbb{P}$ and $\delta>0$ such that $D_\delta(a)\subset U_S$. We may choose $j,n\in\mathbb{N}$ such that $a\in D_{1/n}(v_j)\subset D_\delta(a)$. Hence, $U_{j,n}\neq \emptyset$ and $a\in U_{j,n}$. Since $a\in U$ was arbitrary, $U=\cup_{j,n=1}^\infty U_{j,n}$. If $S_1\subset S_2$, then $U_{S_1}\subset U_{S_2}$. It follows that there exists an increasing sequence of finite sets $S_1\subset S_2\subset S_3\subset \cdots\subset \Lambda^\mathbb{P}$ with
$$
\{a\in\mathcal{C}:\Gamma(A,a)\neq \mathrm{NH}\}=\bigcup_{n=1}^\infty U_{S_n}.
$$
Hence, $\mathbb{P}(\Gamma(A,a)\neq \mathrm{NH})=\lim\limits_{n\rightarrow\infty}\mathbb{P}(U_{S_n})$ and the lemma follows.
\end{proof}

We can now make sense of probabilistic classes in the SCI hierarchy.

\begin{definition}[Probabilistic classes in the SCI hierarchy]
\label{def:probabilistic_classes}
A computational problem $\{\Xi,\Omega,\mathcal{M},\Lambda\}$ does not belong to $\Delta_1^{\mathbb{P},1/2}$ if for any SPGA $\{\Gamma_n\}$,
$$
\inf_{A\in\Omega}\mathbb{P}\left(d(\Gamma_{n}(A),\Xi(A))\leq 2^{-n}\text{ for all }n\in\mathbb{N}\right)\leq 1/2.
$$
A computational problem $\{\Xi,\Omega,\mathcal{M},\Lambda\}$ does not belong to $\Delta_2^{\mathbb{P},1/2}$ if for any SPGA  $\{\Gamma_n\}$,
$$
\inf_{A\in\Omega}\mathbb{P}\left(\lim_{n\rightarrow\infty}\Gamma_{n}(A)=\Xi(A)\right)\leq 1/2.
$$
\end{definition}

Since a sequence of general algorithms is an SPGA, if $\{\Xi,\Omega,\mathcal{M},\Lambda\}\not\in\Delta_1^{\mathbb{P},1/2}$, then $\{\Xi,\Omega,\mathcal{M},\Lambda\}\not\in\Delta_1^{G}$. Similarly, if $\{\Xi,\Omega,\mathcal{M},\Lambda\}\not\in\Delta_2^{\mathbb{P},1/2}$, then $\{\Xi,\Omega,\mathcal{M},\Lambda\}\not\in\Delta_2^{G}$.

\subsection{The setup of this paper -- a perfect measurement device (which implies stronger results)}
In this paper, we consider computational problems where:
\begin{itemize}
	\item $\Omega$ is a class of dynamical systems, or $F$, on some state space $\mathcal{X}$ as in \eqref{eq:DynamicalSystem};
	\item $\Lambda$, the evaluation set, will be pointwise evaluations of the function $F$:
	$$
\Lambda=\Lambda_{\mathcal{X}}=\left\{F\mapsto F(\hat{x}_j):j=1,2,\ldots\right\},
$$
where $\{\hat{x}_j\}_{j=1}^\infty$ is a dense subset of $(\mathcal{X}, d_{\mathcal{X}})$.
	\item The metric space $(\mathcal{M},d)$ will be either $\Md$, in the case of decision problems, or $\MH$, when we compute spectral sets. 
\end{itemize}
In general, the choice of evaluation set matters \cite{chandler2024spectral}. The above setup agrees with the usual ``snapshot'' setting of data-driven Koopmanism, where one is given access to a finite collection of pairs
$$
\left\{\left(x^{(m)},y^{(m)}=F(x^{(m)})\right):m=1,\ldots,M\right\}.
$$
However, our lower bounds become \textit{stronger}. Specifically, this strengthening occurs when we prove lower bounds for a computational problem $\{\Xi,\Omega,\mathcal{M},\Lambda_{\mathcal{X}}\}^{\Delta_{1}}$, which we remind the reader, corresponds to allowing arbitrary precision of the evaluation set according to \cref{def:inexact_def_need}. Hence, our lower bound holds even if we allow algorithms arbitrarily many point samples of $F$ to arbitrary precision. When proving lower bounds, we will always deal with the case that $\omega$ is the natural Lebesgue measure on $\mathcal{X}\subset\mathbb{R}^d$, though many of our results can be easily extended to other measures.

\begin{remark}[Simplified notation]
To simplify notation, we drop the superscript $\Delta_{1}$ from the notation throughout the results below. We also drop the superscript $\alpha$ from the computational classes, e.g., writing $\Sigma_1$. Our lower bounds are proven for $\alpha=G$ (general algorithms), and our upper bounds can generally be made to work with $\alpha=A$ (arithmetic algorithms). Hence, we obtain the strongest possible results regarding computational models.
\end{remark}

\section{Theorems that tell us what is possible and what is not possible}

We now provide statements of our theorems. These theorems precisely state the results written in the main text, corresponding to classifications in the SCI hierarchy. Whenever we prove an upper bound in the SCI hierarchy (i.e., construct a convergent tower of algorithms), we shall reference the corresponding pseudocode given in \cref{section:pseudocodes}.

\subsection{A general computational result -- upper bounds or ``possibility results''}
\label{sec:general_ResDMD_method}

In this section, we let $(\mathcal{X},d_{\mathcal{X}})$ be a compact metric space and $\omega$ a finite Borel measure on $(\mathcal{X},d_{\mathcal{X}})$. Under these assumptions, there are two key properties of $L^2(\mathcal{X},\omega)$ that we shall make use of:
\begin{itemize}
	\item $L^2(\mathcal{X},\omega)$ is a separable Hilbert space \cite[Proposition 3.4.5]{cohn2013measure};
	\item The space of continuous functions on $\mathcal{X}$, denoted $C(\mathcal{X})$, is dense in $L^2(\mathcal{X},\omega)$ \cite[Proposition 7.9]{folland1999real}.
\end{itemize}
In particular, by the Gram--Schmidt process, there exists an orthonormal basis $\{g_1,g_2,\ldots\}\subset C(\mathcal{X})$ of $L^2(\mathcal{X},\omega)$.

\begin{definition}
Let $\alpha:\mathbb{R}_{\geq0}\rightarrow\mathbb{R}_{\geq0}$ be an increasing continuous function with $\alpha(0)=0$. We say that a continuous function $F:\mathcal{X}\rightarrow\mathcal{X}$ has a modulus of continuity $\alpha$ if
\begin{equation}
\label{mod_cty}
d_{\mathcal{X}}(F(x),F(y))\leq \alpha(d_{\mathcal{X}}(x,y))\quad \forall x,y\in \mathcal{X}.
\end{equation}
\end{definition}

Since $\mathcal{X}$ is compact, any continuous function $F:\mathcal{X}\rightarrow\mathcal{X}$ is uniformly continuous and hence has a choice of $\alpha$ so that \cref{mod_cty} holds. However, there is no $\alpha$ such that \cref{mod_cty} holds universally for all continuous functions. We set
\begin{align*}
\Omega_{\mathcal{X}}&=\left\{F:\mathcal{X}\rightarrow\mathcal{X}\text{ such that $F$ is continuous and non-singular, $\mathcal{K}_F$ is bounded}\right\},\\
\Omega_{\mathcal{X}}^{\alpha}&=\left\{F:\mathcal{X}\rightarrow\mathcal{X}\text{ such that $F$ has a modulus of continuity $\alpha$ and is non-singular, $\mathcal{K}_F$ is bounded}\right\},\\
\Omega_{\mathcal{X}}^{m}&=\left\{F:\mathcal{X}\rightarrow\mathcal{X}\text{ such that $F$ is measure-preserving}\right\},\\
\Omega_{\mathcal{X}}^{\alpha,m}&=\Omega_{\mathcal{X}}^{\alpha}\cap\Omega_{\mathcal{X}}^{m}.
\end{align*}
We consider the problem functions $\Xi_{\spec_{\mathrm{ap}}}(F)=\spec_{\mathrm{ap}}(\mathcal{K}_F)$ and $\Xi_{\spec_{\mathrm{ap},\epsilon}}(F)=\spec_{\mathrm{ap},\epsilon}(\mathcal{K}_F)$ for any $\epsilon>0$.

\begin{example}[EDMD does not work for $\Omega_{\mathcal{X}}^{\alpha,m}$, from \cite{mezic2020numerical}]
Let $\mathcal{X}=\mathbb{T}$ (the unit circle), equipped with the usual measure $\omega$, and consider the doubling map
$
F(x)=x^2.
$
To apply the algorithm EDMD, we use the Fourier basis $\psi_j(x)={x^j}/{\sqrt{2\pi}}$, $j\in \mathbb{Z}$. Note that $\mathcal{K}_F\psi_j=\psi_{2j}$ and $\spec(\mathcal{K}_F)=\{z\in\mathbb{C}:|z|\leq 1\}$ ($\mathcal{K}_F$ is an isometry whose range is a strict subspace of $L^2(\mathcal{X},\omega)$). We may split the space $L^2(\mathcal{X},\omega)$ into invariant subspaces as follows. Let $n\in\mathbb{N}$ be odd, then $\mathcal{K}_F$ acts as a unilateral shift on $\mathrm{span}\{\psi_{n2^k}:k=0,1,\ldots\}$ and $\mathrm{span}\{\psi_{-n2^k}:k=0,1,\ldots\}$. Hence, $\mathcal{K}_F$ acts as a direct sum of unilateral shifts. If we use a finite number of Fourier basis functions as our dictionary, the large data limit of EDMD is a direct sum of finite sections of unilateral shifts. These finite matrices have spectrum $\{0\}$, and hence, we completely miss regions of the spectrum.
\end{example}

The above example shows that methods such as EDMD do not converge, even for the class $\Omega_{\mathcal{X}}^{\alpha,m}$. This kind of argument can also be extended to systems in $\Omega_{\mathcal{X}}^{\alpha,m}$ whose Koopman operator is not just an isometry, but also unitary (e.g., Arnold's cat map). Nevertheless, part of the following theorem says we can use a different algorithm to ensure convergence.

\begin{theorem}
\label{thm:general_koopman_computation}
Given the above setup, for $\epsilon>0$, we have the following classifications:
\begin{gather*}
\{\Xi_{\spec_{\mathrm{ap},\epsilon}}, \Omega_{\mathcal{X}},\MH, \Lambda_{\mathcal{X}}\} \in \Sigma_2,\qquad
\{\Xi_{\spec_{\mathrm{ap}}}, \Omega_{\mathcal{X}},\MH, \Lambda_{\mathcal{X}}\} \in \Pi_3\\
\{\Xi_{\spec_{\mathrm{ap},\epsilon}}, \Omega_{\mathcal{X}}^\alpha,\MH, \Lambda_{\mathcal{X}}\} \in \Sigma_1,\qquad
\{\Xi_{\spec_{\mathrm{ap}}}, \Omega_{\mathcal{X}}^\alpha,\MH, \Lambda_{\mathcal{X}}\} \in \Pi_2\\
\{\Xi_{\spec_{\mathrm{ap},\epsilon}}, \Omega_{\mathcal{X}}^m,\MH, \Lambda_{\mathcal{X}}\} \in \Sigma_2,\qquad
\{\Xi_{\spec_{\mathrm{ap}}}, \Omega_{\mathcal{X}}^m,\MH, \Lambda_{\mathcal{X}}\} \in \Sigma_2\\
\{\Xi_{\spec_{\mathrm{ap},\epsilon}}, \Omega_{\mathcal{X}}^{\alpha,m},\MH, \Lambda_{\mathcal{X}}\} \in \Sigma_1,\qquad
\{\Xi_{\spec_{\mathrm{ap}}}, \Omega_{\mathcal{X}}^{\alpha,m},\MH, \Lambda_{\mathcal{X}}\} \in \Sigma_1.
\end{gather*}
\end{theorem}

\begin{proof}
\textbf{Step 1:} Classifications for $\Omega_{\mathcal{X}}^\alpha$.
Given $F\in\Omega_{\mathcal{X}}^\alpha$, we consider the \textit{infinite} matrices $A$ and $L$ acting on $l^2(\mathbb{N})$ with
$$
A_{i,j}=
\langle \mathcal{K}_Fg_j,g_i \rangle=\int_{\mathcal{X}} g_j(F(x))\overline{g_i(x)}\dd \omega(x),\quad
L_{i,j}=
\langle \mathcal{K}_F^*\mathcal{K}_Fg_j,g_i \rangle=\langle \mathcal{K}_Fg_j,\mathcal{K}_Fg_i \rangle
=\int_{\mathcal{X}} g_j(F(x))\overline{g_i(F(x))}\dd \omega(x),\quad i,j\in\mathbb{N}.
$$
We first claim that if $F\in\Omega_{\mathcal{X}}^\alpha$, then given any $i,j\in\mathbb{N}$ and $\delta>0$, there exists a general algorithm using $\Delta_1$-information that computes an approximation of $A_{i,j}$ and $L_{i,j}$ within an error bounded by $\delta$. We show this for $A_{i,j}$, and the case of $L_{i,j}$ is similar.

Since $\mathcal{X}$ is a compact metric space, given any $\eta>0$, there exists a finite subset $\{x_{1,\eta},\ldots,x_{N_{\eta},\eta}\}\subset\mathcal{X}$ and continuous functions $\rho_{p,\eta}:\mathcal{X}\rightarrow [0,1]$ such that
$$
\sum_{p=1}^{N_\eta} \rho_{p,\eta}(x)=1\quad\forall x\in\mathcal{X},\quad \mathrm{supp}(\rho_{p,\eta})\subset \{x\in\mathcal{X}:d_{\mathcal{X}}(x,x_{p,\eta})<\eta\}\text{ for }p=1,\ldots,N_{\eta}.
$$
By slightly changing the points $\{x_{j,\eta}\}$ if necessary, we may assume that $\{x_{1,\eta},\ldots,x_{N_{\eta},\eta}\}\subset\{\hat{x}_j\}_{j=1}^\infty$.
Let $y_{p,\eta}$ be an approximation of $F(x_{p,\eta})$ to accuracy $\eta$, which can be computed from the given $\Delta_1$-information.
We then approximate the integral $A_{i,j}$ by
$$
\Gamma_{\eta}(F)=\sum_{p=1}^{N_\eta}\int_{\mathcal{X}}\rho_{p,\eta}(x)g_j(y_{p,\eta})\overline{g_i(x)}\dd \omega(x).
$$
To bound the error in this approximation, note that if $x\in\mathrm{supp}(\rho_{p,\eta})$, then
$d_{\mathcal{X}}(x,x_{p,\eta})<\eta$. Recall that $\alpha$ is a modulus of continuity for $F$. Let $\alpha_j$ be a modulus of continuity for $g_j$, then
\begin{align*}
|g_j(F(x))-g_j(y_{p,\eta})|&\leq |g_j(F(x))-g_j(F(x_{p,\eta}))|+|g_j(F(x_{p,\eta}))-g_j(y_{p,\eta})|\\
&\leq \alpha_j(d_{\mathcal{X}}(F(x),F(x_{p,\eta})))+\alpha_j(d_{\mathcal{X}}(F(x_{p,\eta}),y_{p,\eta}))\leq \alpha_j(\alpha(\eta))+\alpha_j(\eta).
\end{align*}
It follows that
$$
\left|A_{i,j}-\Gamma_{\eta}(F)\right|\leq \int_{\mathcal{X}}\sum_{p=1}^{N_\eta}\rho_{p,\eta}(x)|g_j(F(x))-g_j(y_{p,\eta})||g_i(x)|\dd \omega(x)\leq \left[\alpha_j(\alpha(\eta))+\alpha_j(\eta)\right]\int_{\mathcal{X}}|g_i(x)|\dd \omega(x).
$$
We can make this bound smaller than a given $\delta>0$ by choosing $\eta$ sufficiently small.

For a given $n\in\mathbb{N}$, let $\mathcal{P}_n$ be the orthogonal projection onto the span of the first $n$ canonical basis vectors of $l^2(\mathbb{N})$. We define the function
$$
h_n(z,F)=\sigma_{\mathrm{inf}}((A-zI)\mathcal{P}_n^*),\quad\text{where}\quad \sigma_{\inf}(T)=\inf\{\|Tx\|:\|x\|=1\},
$$
and we view $(A-zI)\mathcal{P}_n^*$ as an operator from the range of $\mathcal{P}_n$ to $l^2(\mathbb{N})$.
Since $A^*A=L$, we can rewrite $h_n$ as
$$
h_n(z,F)=\sqrt{\sigma_{\mathrm{inf}}(\mathcal{P}_n(A-zI)^*(A-zI)\mathcal{P}_n^*)}=\sqrt{\sigma_{\mathrm{inf}}(\mathcal{P}_n[L-\overline{z}A-zA^*+|z|^2I]\mathcal{P}_n^*)}.
$$
In particular, the operator $\mathcal{P}_n[L-\overline{z}A-zA^*+|z|^2I]\mathcal{P}_n^*$ is built from a finite matrix truncation of $A$ and $L$. It follows that we may compute $h_n(z,F)$ to any desired accuracy using finitely many evaluations of $F$ to a given precision. By Dini's theorem, $h_n(z,F)$ converges locally uniformly to the function $\sigma_{\mathrm{inf}}(A-zI)$. We may now apply the general construction of \cite{colbrook3} to see that $\{\Xi_{\spec_{\mathrm{ap},\epsilon}}, \Omega_{\mathcal{X}}^\alpha,\MH, \Lambda_{\mathcal{X}}\} \in \Sigma_1$. Namely, we let $\tilde{h}_n$ be an approximation of $h_n$ computed to accuracy $1/n$ and set
$$
\gamma_n(F)=\left\{z\in\frac{1}{n}(\mathbb{Z}+i\mathbb{Z})\cap B_{n}(0):\tilde{h}_n(z,F)+\frac{1}{n}<\epsilon\right\}\subset \spec_{\mathrm{ap},\epsilon}(\mathcal{K}_F),
$$
which converges to $\spec_{\mathrm{ap},\epsilon}(\mathcal{K}_F)$ as $n\rightarrow\infty$. The method is summarized in \cref{pseudo_code_easy_pseudospec}. Note that $\spec_{\mathrm{ap}}(\mathcal{K}_F)\subset\spec_{\mathrm{ap},\epsilon}(\mathcal{K}_F)$ and that $\lim_{\epsilon\downarrow0}\spec_{\mathrm{ap},\epsilon}(\mathcal{K}_F)=\spec_{\mathrm{ap}}(\mathcal{K}_F)$. It follows from the classification for $\Xi_{\spec_{\mathrm{ap},\epsilon}}$ that $\{\Xi_{\spec_{\mathrm{ap}}}, \Omega_{\mathcal{X}}^\alpha,\MH, \Lambda_{\mathcal{X}}\} \in \Pi_2$. This method is summarized in \cref{pseudo_code_medium_spec}.

\textbf{Step 2:} Classifications for $\Omega_{\mathcal{X}}$. The proof is similar to the $\Omega_{\mathcal{X}}^\alpha$ case, but we do not assume access to $\alpha$, a modulus of continuity for $F$. It follows that we can compute any matrix element $A_{i,j}$ or $L_{i,j}$ in one limit without error control. In particular, we let
$
h_{n_2,n_1}(z,F)
$
be functions that we compute with $\lim_{n_1\rightarrow\infty}h_{n_2,n_1}(z,F)=h_{n_2}(z,F)$. However, the set
$$
\left\{z\in\frac{1}{n_2}(\mathbb{Z}+i\mathbb{Z})\cap B_{n_2}(0):h_{n_2,n_1}(z,F)<\epsilon\right\}
$$
need not converge as $n_1\rightarrow\infty$, since the convergence $\lim_{n_1\rightarrow\infty}h_{n_2,n_1}(z,F)=h_{n_2}(z,F)$ need not be monotonic. To fix this, we define $\Gamma_{n_2,n_1}(F)$ as follows. Let $z\in\frac{1}{n_2}(\mathbb{Z}+i\mathbb{Z})\cap B_{n_2}(0)$ and consider the separated intervals
$$
I_{n_2}^1(\epsilon)=[0,\epsilon-1/n_2],\quad I_{n_2}^2(\epsilon)=[\epsilon+1/(2n_2),\infty).
$$
Given $h_{n_2,j}(z,F)$ for $j=1,\ldots,n_1$, let $k$ be the largest such $j$ with $h_{n_2,j}(z,F)\in I_{n_2}^1(\epsilon)\cup I_{n_2}^2(\epsilon)$. If such a $k$ exists with $h_{n_2,k}(z,F)\in I_{n_2}^1(\epsilon)$, then $z\in \Gamma_{n_2,n_1}(F)$. Otherwise, $z\notin \Gamma_{n_2,n_1}(F)$. Since the sequence $h_{n_2,j}(z,F)$ cannot visit both intervals $I_{n_2}^1(\epsilon)$ and $I_{n_2}^2(\epsilon)$ infinitely often as $n_1\rightarrow\infty$, it follows that the limit $\lim_{n_1\rightarrow\infty}\Gamma_{n_2,n_1}(F)=\Gamma_{n_2}(F)$ exists. Moreover,
$$
\left\{z\in\frac{1}{n_2}(\mathbb{Z}+i\mathbb{Z})\cap B_{n_2}(0):h_{n_2}(z,F)<\epsilon-\frac{1}{2n_2}\right\}\subset \Gamma_{n_2}(F)\subset
\left\{z\in\frac{1}{n_2}(\mathbb{Z}+i\mathbb{Z})\cap B_{n_2}(0):h_{n_2}(z,F)<\epsilon\right\}\subset \spec_{\mathrm{ap},\epsilon}(\mathcal{K}_F).
$$
Hence, $\{\Gamma_{n_2,n_1}\}$ is a $\Sigma_2$-tower for $\{\Xi_{\spec_{\mathrm{ap},\epsilon}}, \Omega_{\mathcal{X}},\MH, \Lambda_{\mathcal{X}}\}$. The method is summarized in \cref{pseudo_code_medium_pseudospec}. Again by taking $\epsilon\downarrow 0$, we see that $\{\Xi_{\spec_{\mathrm{ap}}}, \Omega_{\mathcal{X}},\MH, \Lambda_{\mathcal{X}}\}\in\Pi_3$.
This method is summarized in \cref{pseudo_code_hard_spec}.

\textbf{Step 3:} Classifications for $\Omega_{\mathcal{X}}^m$. Since $\Omega_{\mathcal{X}}^m\subset\Omega_{\mathcal{X}}$, we need only prove the classification for the problem function $\Xi_{\spec_{\mathrm{ap}}}$. Note that if $F\in\Omega_{\mathcal{X}}^m$, then $\mathcal{K}_F$ is an isometry. It follows from the Wold--von Neumann decomposition \cite[Theorem I.1.1]{nagy2010harmonic} that $\mathcal{K}_F$ can be written as a direct sum of copies of the unilateral shift and a unitary operator. In particular,
$$
\sigma_{\mathrm{inf}}(\mathcal{K}_F-zI)=\mathrm{dist}(z,\spec_{\mathrm{ap}}(\mathcal{K}_F)),\quad \spec_{\mathrm{ap},\epsilon}(\mathcal{K}_F)=\spec_{\mathrm{ap}}(\mathcal{K}_F)+B_\epsilon(0).
$$
We can convert the $\Sigma_2$ tower for $\Xi_{\spec_{\mathrm{ap},\epsilon}}$ to a $\Sigma_2$ tower for $\Xi_{\spec_{\mathrm{ap}}}$ following the same argument for self-adjoint operators in \cite{ben2015can}.

\textbf{Step 4:} Classifications for $\Omega_{\mathcal{X}}^{\alpha,m}$. Again, if $F\in\Omega_{\mathcal{X}}^m$, then
$$
\sigma_{\mathrm{inf}}(\mathcal{K}_F-zI)=\mathrm{dist}(z,\spec_{\mathrm{ap}}(\mathcal{K}_F)),\quad \spec_{\mathrm{ap},\epsilon}(\mathcal{K}_F)=\spec_{\mathrm{ap}}(\mathcal{K}_F)+B_\epsilon(0).
$$
We can argue as in step 1 for $\Xi_{\spec_{\mathrm{ap},\epsilon}}$. We can convert the $\Sigma_1$ tower for $\Xi_{\spec_{\mathrm{ap},\epsilon}}$ to a $\Sigma_1$ tower for $\Xi_{\spec_{\mathrm{ap}}}$ as in \cite{colbrook3}. The method is summarized in \cref{pseudo_code_easy_spec}.
\end{proof}

\begin{remark}
It is remarkable that there is a constructive computational procedure for $\{\Xi_{\spec_{\mathrm{ap}}}, \Omega_{\mathcal{X}},\MH, \Lambda_{\mathcal{X}}\}$. The procedure involves three successive limits. The first limit is the large data limit, collecting more snapshots of the dynamical system. The second limit is the large subspace limit, observing the action of the Koopman operator on more observables. The final limit is the computation of spectra through pseudospectra, which is also how the classical computational spectral problem was solved \cite{Hansen_JAMS}. This phenomenon of several successive limits occurs in all algorithms for Koopman operators that provably converge. In particular, the above argument using the matrices $L$ is a generalization of the ResDMD algorithm \cite{colbrook2021rigorousKoop,colbrook2023residualJFM}.
We shall see below that several successive limits are necessary unless we can control the large data limit (e.g., using a modulus of continuity) and the pseudospectra limit (e.g., by assuming that the system is measure-preserving so that $\sigma_{\mathrm{inf}}(\mathcal{K}_F-zI)=\mathrm{dist}(z,\spec_{\mathrm{ap}}(\mathcal{K}_F))$, see \cref{def:gs} for a generalization).
\end{remark}

\vspace{3mm}

\begin{remark}
In the above proof, we used a basis constructed for $\mathcal{X}$ and assumed that certain system-independent integrals with respect to the measure $\omega$ could be computed. Once $\mathcal{X}$ and $\omega$ are fixed, this still defines a general algorithm. To obtain an arithmetic algorithm, the basis construction must depend on the state space and measure. Practical examples of this are discussed in \cref{SI_sec:others}. Methods for learning a well-conditioned dictionary include, for example, those in \cite{das2021reproducing,valva2023consistent,berry2015nonparametric}. Proving the SCI of basis construction lies beyond the scope of this paper, but represents an interesting direction for future work.
\end{remark}

\subsection{Measure-preserving maps on the unit disk}

In \cref{thm:general_koopman_computation}, we saw that the approximate point spectrum of Koopman operators associated with measure-preserving dynamical systems can be computed in one limit if we have a bound on the modulus of continuity of $F$, and two limits otherwise. We now show that this classification is sharp by proving a lower bound. We consider invertible, measure-preserving dynamical systems on the closed unit disk $\mathcal{X}=\cl{\mathbb{D}} \subset\mathbb{R}^2$, equipped with the Euclidean metric and standard Lebesgue measure. We set
$$
\Omega_{\mathbb{D}}=\left\{F:\cl{\mathbb{D}}\rightarrow \cl{\mathbb{D}}\text{ such that } F\text{ is continuous, measure-preserving and invertible}\right\}
$$
and also consider maps with a priori known Lipschitz constant,
$$
\Omega_{\mathbb{D}}^{L}=\left\{F\in\Omega_{\mathbb{D}}:\mathrm{Lip}(F)\leq L\right\},
$$
where $\mathrm{Lip}(F)$ is the (optimal) Lipschitz constant if $F$ is Lipschitz and $+\infty$ otherwise.

\begin{theorem}\label{thm:disc}
Given the above setup, we have the following classifications for the spectrum:
\begin{gather*}
\Delta^{\mathbb{P},1/2}_2 \not\owns\{\Xi_{\spec}, \Omega_{\mathbb{D}},\MH, \Lambda_{\cl{\mathbb{D}}}\} \in \Sigma_2,\qquad
\Delta^{\mathbb{P},1/2}_1 \not\owns \{\Xi_{\spec}, \Omega_{\mathbb{D}}^L, \MH, \Lambda_{\cl{\mathbb{D}}}\}  \in \Sigma_1.
\end{gather*}
The same classifications hold for the pseudospectrum $\spec_\epsilon$ (and the approximate point spectrum and approximate point pseudospectrum).
In other words, to compute the spectral sets in one limit, we must be able to bound the variability of $F$.
\end{theorem}

Let $F_0\in\Omega_{\mathbb{D}}$ be the map $(r,\theta)\mapsto (r,\theta+\pi)$, where we use polar coordinates. Using the eigenfunctions of the Dirichlet Laplacian, we see that $\spec(\mathcal{K}_{F_0})=\{\pm 1\}$. To prove \cref{thm:disc}, we will need the following technical lemma regarding $F_0$.

\begin{lemma}
\label{technical_lemma}
Let $\mathcal{A}=\{x\in\mathbb{R}^2:0<R\leq |x|\leq r<1\}$ be an annulus and $X=\{x_1,\ldots,x_N\}$ and $Y=\{y_1,\ldots,y_N\}$ two sets of $N$ distinct points in $\mathcal{A}$ with $X\cap Y=\emptyset$. Then there exists a measure-preserving homeomorphism $H$ that acts as the identity on $\cl{\mathbb{D}}\backslash \mathcal{A}$ such that $H^{-1}\circ F_0\circ H(x_{j})=y_{j}$ for $j=1,\ldots,N$.
\end{lemma}

\begin{proof}
It suffices to show there exists a measure-preserving homeomorphism $H$ that acts as the identity on $\cl{\mathbb{D}}\backslash \mathcal{A}$ with
$$
H(x_j) = (r_j, 0) ,\quad
H(y_j) = (r_j, \pi), \quad j=1,\ldots,N,
$$
where we employ polar coordinates and
$
r > r_1 > \cdots > r_N > R.  
$
Consider a family $\mathcal{G}$ of nested, smooth, closed simple curves filling $\cl{\mathbb{D}}$ such that $\mathcal{G}$ consists of concentric circles on $\cl{\mathbb{D}}\backslash \mathcal{A}$ and otherwise all curves enclose $R\cdot\overline{\mathbb{D}}$. We can construct such a family $\mathcal{G}$ so that $x_j$ and $y_j$ lie on the same curve in $\mathcal{G}$ for all $j$. Let $\mathfrak{c}$ be a Jordan arc from $\partial\mathbb{D}$ to the origin, such that $\mathfrak{c}$ meets each curve in $\mathcal{G}$ at a single point. Assume that $\mathfrak{c}$ passes through $x_j$ for all $j$, starting with $x_1$ then passing through $x_2,\ldots, x_N$ consecutively on the way to the origin  (after reordering the pairs $(x_j,y_j)$ if necessary). Without loss of generality, $\mathcal{G}$ may be represented implicitly by a smooth positive function on $\cl{\mathbb{D}}$ via the formula 
$
    f(x) = a$, $0 < a < 1$
with $f(x_1) > \cdots > f(x_N).$ By \cite[Theorem 1]{brown1935certain}, there exists a measure-preserving homeomorphism $H$, acting as the identity on $\cl{\mathbb{D}}\backslash \mathcal{A}$, which maps each family of curves $\mathcal{G}$ to a family of nested concentric circles in $\cl{\mathbb{D}}$ about the origin. Furthermore, the theorem ensures that $H$ maps the chord $\mathfrak{c}$  onto the radius of $\cl{\mathbb{D}}$ corresponding to zero angular coordinates. It follows that $H(x_j) = (r_j, 0)$ and $|H(y_j)| = r_j$ for all $j$. 

It remains to prove that we can choose $f$ so that the angular coordinate of $H(y_j)$ is $\pi$. We employ the formula for $H$ provided in \cite[Equations 3 and 4]{brown1935certain}, 
\begin{align}
    |H(x)|  = f(x), \qquad \theta & = \frac{1}{f(x)}\int_{T(x)}^x \frac{\dd s}{f_n} = \frac{1}{f(x)}\int_{T(x)}^x \frac{\dd s}{\sqrt{[\partial_1 f(s)]^2 + [\partial_2 f(s)]^2}},
\end{align}
where $T(x)$ is the point of intersection with $\mathfrak{c}$ of the curve of $\mathcal{G}$ through
$x$, $f_n$ is the directional derivative in the outer normal direction, and the integral is evaluated along the curve of $\mathcal{G}$ with $s$ measured in the sense for which the interior is on the left. Let $\theta_j$ denote the angular coordinate of $H(y_j)$. By adjusting $f$ only in a small neighborhood of the curve passing through $x_j$ and $y_j$ (avoiding the curves passing through the other points) and making $f_n$ small or large along the path of integration from $T(y_j) = x_j$ to $y_j$, we can ensure that $\theta_j=\pi$.
\end{proof}

\begin{proof}[Proof of \cref{thm:disc}]
We split the proof into three parts: upper bounds, and proving lower bounds for $\Omega_{\mathbb{D}}$ and $\Omega_{\mathbb{D}}^L$.

\textbf{Step 1: Upper bounds.} To prove the upper bounds, we must alter the proof of \cref{thm:general_koopman_computation} to arithmetic algorithms. We let $\{g_n\}_{n=1}^\infty$ be an orthonormal basis of $L^2(\mathbb{D})$ made up of the eigenfunctions of the Dirichlet Laplacian. These functions can be explicitly expressed in terms of Fourier basis functions in the angle coordinate and Bessel functions of the first kind in the radial coordinate. Let $F\in\Omega_{\mathbb{D}}^L$ and consider the two types of integrals
$$
A_{i,j}=
\langle \mathcal{K}_Fg_j,g_i \rangle=\int_{\mathbb{D}} g_j(F(x))\overline{g_i(x)}\dd x,\quad
L_{i,j}=
\langle \mathcal{K}_F^*\mathcal{K}_Fg_j,g_i \rangle=\langle \mathcal{K}_Fg_j,\mathcal{K}_Fg_i \rangle
=\int_{\mathbb{D}} g_j(F(x))\overline{g_i(F(x))}\dd x,\quad i,j\in\mathbb{N}.
$$
We can bound the modulus of continuity and modulus of the products $ g_j(F(x))\overline{g_i(x)}$ and $g_j(F(x))\overline{g_i(F(x))}$. It follows that the integrals can be computed to any desired accuracy using quadrature, for example, Riemann sums. (In practice, other choices may be better.) Similarly, if $F\in \Omega_{\mathbb{D}}$, we can construct arithmetic algorithms that converge to these inner products without error control. The rest of the proofs of upper bounds follow those of the upper bounds in \cref{thm:general_koopman_computation}.

\textbf{Step 2:} $\{\Xi_{\spec}, \Omega_{\mathbb{D}},\MH, \Lambda_{\cl{\mathbb{D}}}\}\notin \Delta^{\mathbb{P},1/2}_2$.
Suppose for a contradiction that $\{\Gamma_n\}$ is an SPGA for $\{\Xi_{\spec}, \Omega_{\mathbb{D}},\MH, \Lambda_{\cl{\mathbb{D}}}\}$ such that
\begin{equation}
\label{_prob_cotradiction_mp1}
\inf_{F\in\Omega_{\mathbb{D}}}\mathbb{P}\left(\lim_{n\rightarrow\infty}\Gamma_{n}(\tilde{F})=\spec(\mathcal{K}_F)\right)=\epsilon>1/2.
\end{equation}
Here, $\tilde{F}$ represents inexact input for $F$. That is, an admissible $\tilde{F}$ is of the form $\{f_{j,n}(F):j,n\in\mathbb{N}\}$, where $\|f_{j,n}(F)-F(\hat{x}_j)\|\leq 2^{-n}$. We will obtain a contradiction by constructing an $F\in\Omega_{\mathbb{D}}$ and corresponding $\tilde{F}$ so that \cref{_prob_cotradiction_mp1} cannot hold. Throughout the proof, we work in polar coordinates $x=(r,\theta)$ in $\cl{\mathbb{D}}$. We also fix a smooth surjective bump function $\phi:[0,1]\rightarrow[0,1]$ with $\mathrm{supp}(\phi)\subset[1/4,3/4]$. All of the elements of $\Omega_{\mathbb{D}}$ we consider act as the identity on $\{0\}\cup\partial\mathbb{D}$, and hence we may remove any $\hat{x}_j\in \{0\}\cup\partial\mathbb{D}$ from the evaluation set $\Lambda_{\cl{\mathbb{D}}}$ without loss of generality. We construct $F$ from a sequence of functions $\{F_k\}_{k=1}^\infty\subset\Omega_{\mathbb{D}}$ that are defined inductively as follows.

\textbf{Base case:} For $k=1$, we begin by defining the perturbed rotation
$$
Z_1(r,\theta)=(r,\theta+\pi+\phi(r)).
$$
A simple argument using Jacobians shows that $Z_1$ is measure-preserving. It is also smooth and invertible. Let $r_0$ be any point in $[0,1]$ with $\phi(r_0)/\pi\not\in\mathbb{Q}$ and $f_n:[0,1]\rightarrow[0,1]$ be a non-vanishing continuous function supported in $I_n=[\max\{r_0-1/n,0\},\min\{r_0+1/n,1\}]$.  For $j\in\mathbb{Z}$, set $f_{n,j}(r,\theta)=c_{n}f_n(r)e^{ij\theta}$, where $c_{n}$ is a normalization constant so that $\|f_{n,j}\|_{L^2}=1$. Then
$$
\left\|\mathcal{K}_{Z_1}f_{n,j}{-}e^{ij(\pi+\phi(r_0))}f_{n,j}\right\|_{L^2}^2=|c_{n}|^2\int_{0}^1\int_{0}^{2\pi} f_n(r)^2\left|e^{ij\phi(r)}{-}e^{ij\phi(r_0)}\right|^2r\dd \theta \dd r\leq \sup_{r\in I_n}\left|e^{ij\phi(r)}{-}e^{ij\phi(r_0)}\right|^2.
$$
This bound converges to zero as $n\rightarrow\infty$ and hence $e^{ij\pi(1+\phi(r_0)/\pi)}\in \spec(\mathcal{K}_{Z_1})$. Since $j\in\mathbb{Z}$ was arbitrary and $\phi(r_0)/\pi\notin\mathbb{Q}$, $\spec(\mathcal{K}_{Z_1})=\mathbb{T}$.

Let $\tilde{Z}_1=\{f_{j,n}(Z_1):j,n\in\mathbb{N}\}$ be a set of $\Delta_1$-information for $Z_1$ such that all of the points $f_{j,n}(Z_1)$ are distinct, $\tilde{Z}_1\cap\{\hat{x}_j:j\in\mathbb{N}\}=\emptyset$, and $\|f_{j,n}(Z_1)-Z_1(\hat{x}_j)\|\leq 2^{-(n+1)}$. The assumed convergence in \cref{_prob_cotradiction_mp1} implies that there exists $n_1\in\mathbb{N}$ such that
$$
\mathbb{P}\left(\mathrm{dist}(i,\Gamma_{n_1}(\tilde{Z}_1))\leq 1\right)\geq\frac{1}{2}\left(\epsilon+\frac{1}{2}\right)=\frac{\epsilon}{2}+\frac{1}{4}>1/2.
$$
Let $E_1$ be the event $\mathrm{dist}(i,\Gamma_{n_1}(\tilde{Z}_1))\leq 1$. Let $E_2$ be the event that $\Gamma_{n_1}(\tilde{Z}_1)\neq \mathrm{NH}$. Note that $E_1\subset E_2$. Given $\eta>0$, we may apply the covering lemma (\cref{lem:prob_covering_lemma}) to deduce the existence of a finite subset $S\subset\Lambda_{\cl{\mathbb{D}}}^{\mathbb{P}}$ such that
$$
\mathbb{P}\left(\{a\in\mathcal{C}:\Lambda_{\Gamma_{n_1}}(\tilde{Z}_1,a)\subset S\}\cap E_2\right)\geq \mathbb{P}(E_2)-\eta.
$$
It follows that
$$
\mathbb{P}(\{\Lambda_{\Gamma_{n_1}}(\tilde{Z}_1,a)\subset S\}\cap E_1)= \mathbb{P}(E_1)-\mathbb{P}(E_1\backslash \{\Lambda_{\Gamma_{n_1}}(\tilde{Z}_1,a)\subset S\})\geq \mathbb{P}(E_1)-\mathbb{P}(E_2\backslash \{\Lambda_{\Gamma_{n_1}}(\tilde{Z}_1,a)\subset S\})\geq \mathbb{P}(E_1)-\eta.
$$
We choose $\eta$ sufficiently small so that $\mathbb{P}(E_1)-\eta\geq {\epsilon}/{4}+{3}/{8}$.
Hence, there exists a finite index set $\mathcal{I}_1\subset\mathbb{N}^2$ so that
$$
\mathbb{P}\left(\mathrm{dist}(i,\Gamma_{n_1}(\tilde{Z}_1))\leq 1\text{ and }C_1(\tilde{Z}_1)\right)\geq\frac{\epsilon}{4}+\frac{3}{8}>1/2,
$$
where $C_1(\tilde{F})$ is the event that $\Gamma_{n_1}$ only samples from $\{f_{j,n}(F):(j,n)\in \mathcal{I}_1\}\subset \tilde{F}$ and outputs $\Gamma_{n_1}(\tilde{F})\neq \mathrm{NH}$. Let $J_1$ be the set of $j$ for which there exists $n$ with $(j,n)\in\mathcal{I}_1$ and for each such $j$, let $p_1(j)$ be the maximal such $n$. Define the sets
$$
X_1=\{\hat{x}_j:j\in J_1\}\quad\text{and}\quad Y_1=\{f_{j,p_1(j)}(Z_1):j\in J_1\}.
$$
Our assumptions on $\tilde{Z}_1$ imply that $X_1$ and $Y_1$ are sets of distinct points with $X_1\cap Y_1=\emptyset$.

Let $\mathcal{A}_1=\{R_1\leq r\leq r_1\}$ be a closed annulus in $\mathbb{D}\backslash\{0\}$, whose interior contains $X_1\cup\{f_{j,n}(Z_1):(j,n)\in \mathcal{I}_1\}$. Using \cref{technical_lemma}, there is a measure-preserving homeomorphism $H_1$ that acts as the identity on $\cl{\mathbb{D}}\backslash \mathcal{A}_1$ such that $[H_1^{-1}\circ F_0\circ H_1](\hat{x}_j)=f_{j,p_1(j)}(Z_1)$ for $j\in J_1$. We define $F_1=H_1^{-1}\circ F_0\circ H_1$ so that $F_1(\hat{x}_j)=f_{j,p_1(j)}(Z_1)$. If $(j,n)\in \mathcal{I}_1$, then
\begin{align*}
\|f_{j,n}(Z_1)-F_1(\hat{x}_j)\|&\leq \|f_{j,n}(Z_1)-Z_1(\hat{x}_j)\|+\|Z_1(\hat{x}_j)-F_1(\hat{x}_j)\|\\
&\leq 2^{-(n+1)}+2^{-(p_1(j)+1)}\leq 2^{-n}. 
\end{align*}
In particular, we may extend $\{f_{j,n}(Z_1):(j,n)\in \mathcal{I}_1\}$ to admissible $\Delta_1$-information for $F_1$, $\tilde{F}_1$, by setting $f_{j,n}(F_1)=F_1(\hat{x}_j)$ for $(j,n)\notin \mathcal{I}_1$. By consistency of probabilistic general algorithms,
\begin{align*}
\mathbb{P}\left(\mathrm{dist}(i,\Gamma_{n_1}(\tilde{F}_1))\leq 1\text{ and }C_1(\tilde{F}_1)\right)=\mathbb{P}\left(\mathrm{dist}(i,\Gamma_{n_1}(\tilde{Z}_1))\leq 1\text{ and }C_1(\tilde{Z}_1)\right)\geq\frac{\epsilon}{4}+\frac{3}{8}>1/2.
\end{align*}
This completes the base case.

\textbf{Inductive step:} For the inductive step, suppose that $F_k\in \Omega_{\mathbb{D}}$, $\tilde{F}_k=\{f_{j,n}(F_k):j,n\in\mathbb{N}\}$, $n_k\in\mathbb{N}$ and $R_k\in(0,1)$ have been defined. To define the next function $F_{k+1}$ and set of $\Delta_1$-information $\tilde{F}_{k+1}$, we first choose a positive $r_{k+1}<\min\{(k+1)^{-1},R_k\}$ and define the following function in $\Omega_{\mathbb{D}}$:
$$
Z_{k+1}(r,\theta)=\begin{cases}
(r,\theta+\pi+\phi(r/r_{k+1})),\quad &\text{if $r\leq r_{k+1}$},\\
F_k(r,\theta),\quad &\text{otherwise}.
\end{cases}
$$
The argument for determining $\spec(\mathcal{K}_{Z_{1}})$ extends to show that $\spec(\mathcal{K}_{Z_{k+1}})=\mathbb{T}$.

Let $\tilde{Z}_{k+1}=\{f_{j,n}(Z_{k+1}):j,n\in\mathbb{N}\}$ be a set of $\Delta_1$-information for $Z_{k+1}$ with the following properties. If $|\hat{x}_j|\geq r_{k+1}$, then $f_{j,n}(Z_{k+1})=f_{j,n}(F_k)$. If $|\hat{x}_j|< r_{k+1}$, then $|f_{j,n}(Z_{k+1})|<r_{k+1}$, $\|f_{j,n}(Z_{k+1})-Z_{k+1}(\hat{x}_j)\|\leq 2^{-(n+1)}$, and the points $\{f_{j,n}(Z_{k+1}):|\hat{x}_j|< r_{k+1}\}$ are distinct with $\{f_{j,n}(Z_{k+1}):|\hat{x}_j|< r_{k+1}\}\cap\{\hat{x}_j:j\in\mathbb{N}\}=\emptyset$. Arguing as above, there exists $n_{k+1}\in\mathbb{N}$ with $n_{k+1}>n_k$, a finite set $\mathcal{I}_{k+1}\subset\mathbb{N}^2$ such that the following two conditions hold. First, letting $J_{k+1}$ be the set of $j$ for which there exists $n$ with $(j,n)\in\mathcal{I}_{k+1}$, it holds that $|\hat{x}_j|< r_{k+1}$ whenever $j\in J_{k+1}$. Second,
$$
\mathbb{P}\left(\mathrm{dist}(i,\Gamma_{n_{k+1}}(\tilde{Z}_{k+1}))\leq 1\text{ and }C_{k+1}(\tilde{Z}_{k+1})\right)\geq\frac{\epsilon}{4}+\frac{3}{8}>1/2,
$$
where $C_{k+1}(\tilde{F})$ is the event that $\Gamma_{n_{k+1}}$ only samples from $\{f_{j,n}(F):(j,n)\in \mathcal{I}_{k+1}\}\cup\{f_{j,n}(F):|\hat{x}_j|\geq r_{k+1}\}\subset\tilde{F}$ and outputs $\Gamma_{n_{k+1}}(\tilde{F})\neq \mathrm{NH}$. For each $j\in J_{k+1}$, let $p_{k+1}(j)$ be the maximal $n$ such that $(j,n)\in\mathcal{I}_{k+1}$. We define the two sets
$$
X_{k+1}=\{\hat{x}_j:j\in J_{k+1}\}\quad\text{and}\quad Y_{k+1}=\{f_{j,p_{k+1}(j)}(Z_{k+1}):j\in J_{k+1}\}.
$$
Similar to the base case, our assumptions on $\tilde{Z}_{k+1}$ imply that $X_{k+1}$ and $Y_{k+1}$ are sets of distinct points with $X_{k+1}\cap Y_{k+1}=\emptyset$.

Let $\mathcal{A}_{k+1}=\{R_{k+1}\leq r\leq r_{k+1}\}$ for some $0<R_{k+1}< r_{k+1}$ be a closed annulus in $\mathbb{D}\backslash\{0\}$, whose interior contains $X_{k+1}\cup\{f_{j,n}(Z_{k+1}):(j,n)\in \mathcal{I}_{k+1}\}$. Using \cref{technical_lemma}, there is a measure-preserving homeomorphism $H_{k+1}$ that acts as the identity on $\cl{\mathbb{D}}\backslash \mathcal{A}_{k+1}$ such that $[H_{k+1}^{-1}\circ F_0\circ H_{k+1}](\hat{x}_j)=f_{j,p_{k+1}(j)}(Z_{k+1})$ for $j\in J_{k+1}$. We then define
$$
F_{k+1}(r,\theta)=\begin{cases}
[H_{k+1}^{-1}\circ F_0\circ H_{k+1}](r,\theta),\quad &\text{if $r\leq r_{k+1}$},\\
F_k(r,\theta),\quad &\text{otherwise},
\end{cases}
$$
so that $F_{k+1}(\hat{x}_j)=f_{j,p_{k+1}(j)}(Z_{k+1})$ for $j\in J_{k+1}$. If $(j,n)\in \mathcal{I}_{k+1}$, then
$$
\|f_{j,n}(Z_{k+1})-F_{k+1}(\hat{x}_j)\|\leq \|f_{j,n}(Z_{k+1})-Z_{k+1}(\hat{x}_j)\|+\|Z_{k+1}(\hat{x}_j)-F_{k+1}(\hat{x}_j)\|\leq 2^{-(n+1)}+2^{-(p_{k+1}(j)+1)}\leq 2^{-n}. 
$$
In particular, we may extend $\{f_{j,n}(Z_{k+1}):(j,n)\in \mathcal{I}_{k+1}\}\cup\{f_{j,n}(F_k):|\hat{x}_j|\geq r_{k+1}\}$ to admissible $\Delta_1$-information for $F_{k+1}$, $\tilde{F}_{k+1}$, by setting $f_{j,n}(F_{k+1})=F_{k+1}(\hat{x}_j)$ if $|\hat{x}_j|< r_{k+1}$ and $(j,n)\notin \mathcal{I}_{k+1}$. By consistency of probabilistic general algorithms,
$$
\mathbb{P}\left(\mathrm{dist}(i,\Gamma_{n_{k+1}}(\tilde{F}_{k+1}))\leq 1\text{ and }C_{k+1}(\tilde{F}_{k+1})\right)=\mathbb{P}\left(\mathrm{dist}(i,\Gamma_{n_{k+1}}(\tilde{Z}_{k+1}))\leq 1\text{ and }C_{k+1}(\tilde{Z}_{k+1})\right)\geq\frac{\epsilon}{4}+\frac{3}{8}>1/2.
$$
This completes the inductive step.

\textbf{Limit argument:} We now let 
$$
F=\lim_{k\rightarrow\infty}F_k,\qquad H=\lim_{k\rightarrow\infty} H_k\circ H_{k-1}\circ\cdots\circ H_1,
$$
and define the $\Delta_1$-information $\tilde{F}$ by
$f_{j,n}(F)=f_{j,n}(F_k)$ if $|\hat{x}_j|\geq r_{k+1}$ for $k=1,2,\ldots$. Since $f_{j,n}({F}_{k+1})=f_{j,n}({F}_{k})$ if $|\hat{x}_j|\geq r_{k+1}$ and $F_{k+1}(x)=F_k(x)$ for $|x|\geq r_{k+1}$, $\tilde{F}$ is well-defined and admissible. By consistency of probabilistic general algorithms, for any $k\in\mathbb{N}$,
$$
\mathbb{P}\left(\mathrm{dist}(i,\Gamma_{n_{k}}(\tilde{F}))\leq 1\text{ and }C_{k}(\tilde{F})\right)=
\mathbb{P}\left(\mathrm{dist}(i,\Gamma_{n_{k}}(\tilde{Z}_{k}))\leq 1\text{ and }C_{k}(\tilde{Z}_{k})\right)>1/2.
$$
Let $T_k$ be the random variable $\mathrm{dist}(i,\Gamma_{n_{k}}(\tilde{F}))$, then
$$
\mathbb{P}\left(T_k\leq 1\right)\geq \mathbb{P}\left(\mathrm{dist}(i,\Gamma_{n_{k}}(\tilde{F}))\leq 1\text{ and }C_{k}(\tilde{F})\right)>1/2\quad \forall k\in\mathbb{N}.
$$
Let $A_m$ be the event
$
\cap_{k= m}^\infty\{T_k>1 \text{ and }\Gamma_{n_{k}}(\tilde{F})\neq \mathrm{NH}\},
$
and note that
\begin{equation}
\label{eq:butAmlessthanhalf}
\mathbb{P}(A_k)\leq 1-\mathbb{P}(T_k\leq 1)<1/2.
\end{equation}
The function $H$ is a measure-preserving homeomorphism and $F=H^{-1}\circ F_0 \circ H$. Hence, $\mathcal{K}_{H}$ is unitary and $\mathcal{K}_{F}=\mathcal{K}_{H}^*\mathcal{K}_{F_0}\mathcal{K}_{H}$ so that $\spec(\mathcal{K}_{F})=\spec(\mathcal{K}_{F_0})=\{\pm 1\}$. In particular, $\mathrm{dist}(i,\spec(\mathcal{K}_{F}))>1$, so by our initial assumption in \cref{_prob_cotradiction_mp1},
$$
\mathbb{P}\left(\cup_{m=1}^\infty A_m\right)\geq \mathbb{P}\left(\lim_{n\rightarrow\infty}\Gamma_{n}(\tilde{F})=\spec(\mathcal{K}_F)\right)\geq \epsilon>1/2.
$$
Since $A_1\subset A_2\subset A_3 \subset \cdots$, there exists $A_M$ with $\mathbb{P}(A_M)>1/2$, which contradicts the bound in \cref{eq:butAmlessthanhalf}.

\textbf{Step 3:} $\{\Xi_{\spec}, \Omega_{\mathbb{D}}^L,\MH, \Lambda_{\cl{\mathbb{D}}}\}\notin \Delta^{\mathbb{P},1/2}_1$. 
We prove the slightly stronger result for exact input. Suppose for a contradiction that $\{\Gamma_n\}$ is an SPGA for $\{\Xi_{\spec}, \Omega_{\mathbb{D}}^L,\MH, \Lambda_{\cl{\mathbb{D}}}\}$ such that
\begin{equation}
\label{prob_cotradiction_mp1_easy}
\inf_{F\in\Omega_{\mathbb{D}}^L}\mathbb{P}\left(\dH(\Gamma_{n}(F),\spec(\mathcal{K}_{F}))\leq 2^{-n}\text{ for all }n\in\mathbb{N}\right)=\epsilon>1/2.
\end{equation}
Let $F_1$ be the identity map. We may argue as above to show that there exists a finite set $X\subset\mathcal{X}=\cl{\mathbb{D}}$ such that
$$
\mathbb{P}\left(\dH(\Gamma_1(F_1),\{1\})\leq 2^{-1}\text{ and }C(F_1)\right)> 1/2,
$$
where $C(F)$ is the event that $\Gamma_{1}$ samples $F$ within the fixed finite set of points $X$ to output $\Gamma_{1}(F)$. As before, let $\phi:[0,1]\rightarrow[0,1]$ be a smooth surjective bump function with $\mathrm{supp}(\phi)\subset[1/4,3/4]$. There exists $r_0\in[0,1]$ and $\delta_0>0$ such that with 
$$
F_2(r,\theta)=(r,\theta+\delta_0\phi(r/r_0)),
$$
we have $F_2\in\Omega_{\mathbb{D}}^L$, $F_1=F_2$ on $X$, and hence
$$
\mathbb{P}\left(\dH(\Gamma_1(F_2),\{1\})\leq 2^{-1}\text{ and }C(F_2)\right)=\mathbb{P}\left(\dH(\Gamma_1(F_1),\{1\})\leq 2^{-1}\text{ and }C(F_1)\right)>1/2.
$$
Given $\delta>0$, we may also assume that $\mathbb{P}(C(F_1))=\mathbb{P}(C(F_2))\geq 1-\delta$. 

We run two independent instances of $\Gamma_1$ for the two inputs $F_1$ and $F_2$. Let $T_1=\Gamma_1(F_1)$ and $T_2=\Gamma_1(F_2)$ be the corresponding outputs, which are random variables. Let $B_j$ be the event $\dH(\Gamma_{1}(F_j),\spec(\mathcal{K}_{F_j}))\leq 2^{-1}$ for $j=1,2$, and let $E=C(F_1)\cap C(F_2)$. If $E$ occurs, then the laws of $T_1$ and $T_2$ are the same. Moreover, $\spec(\mathcal{K}_{F_1})=\{1\}$, whereas $\spec(\mathcal{K}_{F_2})=\mathbb{T}$. Hence,
$$
\mathbb{P}(B_1\cap E)\!+\!\mathbb{P}(B_2\cap E)\!=\!\mathbb{P}(\{\dH(T_1,\{1\})\leq 2^{-1}\}\cap E)\!+\!\mathbb{P}(\{\dH(T_1,\mathbb{T})\leq 2^{-1}\}\cap E).
$$
Since
$
\dH(\{1\},\mathbb{T})>1,
$
it follows from the triangle inequality that
$$
\mathbb{P}(\{\dH(T_1,\{1\})\leq 2^{-1}\}\cap E)+\mathbb{P}(\{\dH(T_1,\mathbb{T})\leq 2^{-1}\}\cap E)\leq 1.
$$
Note that $\mathbb{P}(E)\geq 1-2\delta$. Upon combining with the assumed convergence in \cref{prob_cotradiction_mp1_easy}, we see that
$$
1<2\epsilon\leq\mathbb{P}(B_1)+\mathbb{P}(B_2)\leq  \mathbb{P}(B_1\cap E)+\mathbb{P}(B_2\cap E)+2\mathbb{P}(E^c)\leq 1+4\delta.
$$
This is a contradiction for sufficiently small $\delta$.
\end{proof}

\begin{remark}
We can alter the argument in step 2 of the above proof to smooth functions $F$ on the punctured disk $\cl{\mathbb{D}}\backslash\{0\}$.
\end{remark}

\subsection{Maps on the unit interval}

We now complement \cref{thm:disc} and show that, even for smooth, invertible systems, where we can bound the variability of $F$, the spectrum cannot be computed in one limit if we drop the measure-preserving assumption from \cref{thm:general_koopman_computation}. We consider $\mathcal{X}=[0,1]$ equipped with the Euclidean metric and Lebesgue measure.

Let $c=\{c_n\}_{n=0}^\infty$ be a strictly increasing sequence in $[1/2,1)$ with $c_0=1/2$ and $
\lim_{n\rightarrow\infty}c_n=1$. We set $c_{-n}=1-c_n$ for $n\in\mathbb{N}$ and define the intervals
$$
I_n=[c_n,c_{n+1}),\quad n\in\mathbb{Z}.
$$
We consider a continuous bijection $F:[0,1]\rightarrow[0,1]$ with $F(0)=0$ and $F(1)=1$ such that
$$
F(c_n)=c_{n+1},\quad F(I_n)=I_{n+1}.
$$
In other words, $F$ acts as a bijection between $I_n$ and $I_{n+1}$.
To study the spectrum of $\mathcal{K}_F$, we define the ratios
$$
a_n=\frac{|I_{n+1}|}{|I_n|},\quad n\in\mathbb{Z}.
$$
We assume that $F$ and $F^{-1}$ are smooth on each subinterval $I_n$ and for $\delta>0$, we define
$$
\Omega_{[0,1]}^{\delta}=\left\{F:[0,1]\rightarrow[0,1]\text{ such that the above holds for a sequence $c=\{c_n\}$ with }\max\{\|F'\|_{L^\infty},\|(F^{-1})'\|_{L^\infty}\}\leq1+\delta\right\}.
$$
(The 1 in the $1+\delta$ here comes from the fact that $F$ is a bijection on the whole interval $[0,1]$.)
Note that if the ratios $a_n$ are bounded and bounded below away from zero, then the piecewise affine function $F_c$ (uniquely) defined by acting as an affine function on each $I_n$ is a particular case for which
$$
\max\{\|F_c'\|_{L^\infty},\|(F_c^{-1})'\|_{L^\infty}\}\leq \max\left\{\sup_{n\in\mathbb{N}}a_n,\sup_{n\in\mathbb{N}}a_n^{-1}\right\}.
$$
The map $F_c$ is an example of an affine interval exchange map \cite{levitt1982feuilletages}, which are well-studied in dynamical systems theory.

\begin{theorem}\label{thm:interval}
Given the above setup and $\delta>0$, we have the following classifications for the approximate point spectrum
\begin{gather*}
\Delta^{\mathbb{P},1/2}_2 \not\owns\{\Xi_{\spec_{\mathrm{ap}}}, \Omega_{[0,1]}^{\delta},\MH, \Lambda_{[0,1]}\} \in \Pi_2.
\end{gather*}
The same classifications hold when restricting to smooth functions in $\Omega_{[0,1]}^{\delta}$, or to piecewise affine functions in $\Omega_{[0,1]}^{\delta}$.
\end{theorem}

\begin{remark}
Since $\Omega_{[0,1]}^{\delta}\subset \Omega_{\mathcal{X}}^\alpha$ for $\alpha(x)=(1+\delta)x$, \cref{thm:general_koopman_computation} and \cref{thm:interval} demonstrate that computing $\spec_{\mathrm{ap}}$ is strictly harder than computing $\spec_{\mathrm{ap},\epsilon}$ for this class of Koopman operators.
\end{remark}

\begin{figure}
\centering
\raisebox{-0.5\height}{\includegraphics[width=0.8\textwidth,trim={0mm 0mm 5mm 0mm},clip]{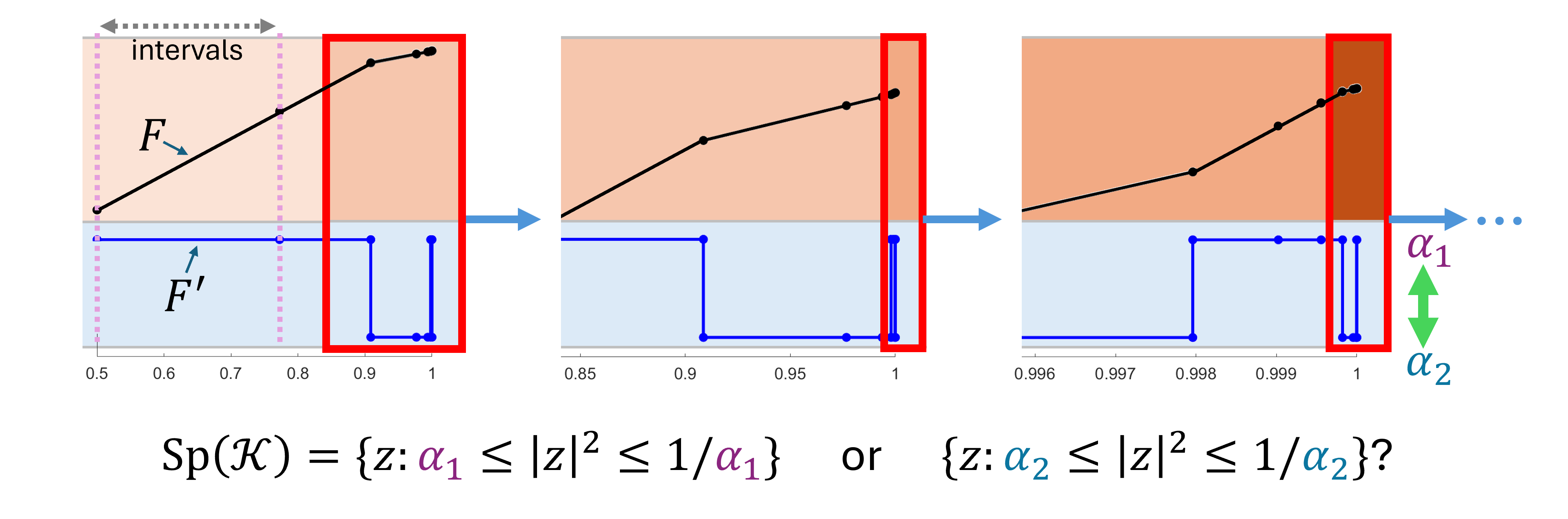}}
\caption{The idea of the proof of the lower bound in \cref{thm:interval}. The plots show successive zoomed-in views of $F$ as we move from left to right. We consider (smoothed) interval exchange maps, whose derivatives switch an infinite number of times as we approach the interval endpoints, causing the approximated spectrum to alternate between different annuli.}
\label{fig:proof2}
\end{figure}

We need the following technical lemma to prove \cref{thm:interval} and the idea is summarized in \cref{fig:proof2}. The first part of this lemma states that any $F\in\Omega_{[0,1]}^{\delta}$ induces a bounded Koopman operator $\mathcal{K}_F$ on $L^2([0,1])$. The second part of the lemma gives conditions for the spectrum to be contained in (or equal to) certain annuli.

\begin{lemma}
\label{lem:annulus_phase}
If $F\in\Omega_{[0,1]}^{\delta}$, then $\|\mathcal{K}_F\|\leq \sqrt{1+\delta}$. Furthermore, suppose that $\lim_{n\rightarrow\infty}a_n=\alpha\in(0,1]$ and that
$$
(1-\delta_F)a_n\leq\frac{\mathrm{d}F}{\mathrm{d}x}(x)\leq (1+\delta_F)a_n\quad \quad \text{and}\quad
(1-\delta_F)/a_{n-1}\leq\frac{\mathrm{d}F^{-1}}{\mathrm{d}x}(x)\leq (1+\delta_F)/a_{n-1}\quad \forall x\in I_n
$$
for some $\delta_F\geq 0$.
Then
$$
\spec(\mathcal{K}_F)\subset\left\{z\in\mathbb{C}:\sqrt{\alpha/(1+\delta_F)}\leq |z|\leq \sqrt{(1+\delta_F)/\alpha}\right\}.
$$
If $\alpha<1$, then
$$
\left\{z\in\mathbb{C}:\sqrt{\alpha}\leq |z|\leq 1/\sqrt{\alpha}\right\}\subset \spec_{\mathrm{ap}}(\mathcal{K}_F).
$$
and hence
$
\spec(\mathcal{K}_{F_c})=\spec_{\mathrm{ap}}(\mathcal{K}_{F_c})=\{z\in\mathbb{C}:\sqrt{\alpha}\leq |z|\leq 1/\sqrt{\alpha}\}.
$
\end{lemma}

\begin{proof}
For ease of notation, let
$
b_n=\sup_{t\in I_{n}}|(F^{-1})'(t)|.
$
If $g\in L^2([0,1])$ with $\|g\|=1$, then
$$
\|\mathcal{K}g\|^2=\sum_{n\in\mathbb{Z}}\int_{I_n} |g(F(x))|^2 \dd x\leq\sum_{n\in\mathbb{Z}}\int_{I_{n+1}} |g(x)|^2 b_{n+1} \dd x\leq
(1+\delta)\sum_{n\in\mathbb{Z}}\int_{I_{n+1}} |g(x)|^2  \dd x=(1+\delta) \|g\|^2.
$$
This bound proves the first part of the lemma, i.e., that $\|\mathcal{K}_F\|\leq \sqrt{1+\delta}$.

Suppose in addition that $\lim_{n\rightarrow\infty}a_n=\alpha\in(0,1]$ and that
$$
(1-\delta_F)a_n\leq\frac{\mathrm{d}F}{\mathrm{d}x}(x)\leq (1+\delta_F)a_n\quad \quad \text{and}\quad
(1-\delta_F)/a_{n-1}\leq\frac{\mathrm{d}F^{-1}}{\mathrm{d}x}(x)\leq (1+\delta_F)/a_{n-1}\quad \forall x\in I_n
$$
for some $\delta_F\geq 0$. We may apply the above argument $l\in\mathbb{N}$ times to see that
$$
\|\mathcal{K}^lg\|^{1/l}\leq\left[\sum_{n\in\mathbb{Z}}\int_{I_{n+l}} |g(x)|^2 \dd x\prod_{j=1}^lb_{n+j}\right]^{1/(2l)}\leq \sqrt{\sup_{n\in\mathbb{Z}}\prod_{j=1}^lb_{n+j}^{1/l}}\leq\sqrt{\sup_{n\in\mathbb{Z}}\frac{1}{l}\sum_{j=1}^lb_{n+j}},
$$
where we have used the generalized AM-GM inequality for the last inequality. For any $L\in\mathbb{N}$,
$$
\limsup_{l\rightarrow\infty}\sup_{n\in\mathbb{Z}}\frac{1}{l}\sum_{j=1}^lb_{n+j}\leq \limsup_{l\rightarrow\infty}\frac{1}{l}\sum_{j=1}^l\sup_{|n|>L}b_n=\sup_{|n|>L}b_n.
$$
Taking $L\rightarrow\infty$, we have
$$
\limsup_{l\rightarrow\infty}\sup_{n\in\mathbb{Z}}\frac{1}{l}\sum_{j=1}^lb_{n+j}\leq\limsup_{|l|\rightarrow\infty}b_l\leq (1+\delta_F)\max\left\{\lim_{n\rightarrow-\infty}1/a_n,\lim_{n\rightarrow+\infty}1/a_n\right\}\leq\frac{(1+\delta_F)}{\alpha}.
$$
It follows that $\lim_{l\rightarrow\infty}\|\mathcal{K}^l\|^{1/l}\leq \sqrt{(1+\delta_F)/\alpha}$. Gelfand's formula for the spectral radius implies that if $z\in \spec(\mathcal{K}_F)$, then $|z|\leq \sqrt{(1+\delta_F)/\alpha}$. We can argue in the same manner, using the fact that $\lim_{n\rightarrow-\infty}a_n=1/\alpha\in[1,\infty)$, to see that $\lim_{l\rightarrow\infty}\|\mathcal{K}^{-l}\|^{1/l}\leq \sqrt{(1+\delta_F)/\alpha}$ and hence that
$$
\inf_{z\in\spec(\mathcal{K}_F)}|z|= \frac{1}{\sup_{z\in\spec(\mathcal{K}_{F^{-1}})}|z|}\geq \frac{1}{\sqrt{(1+\delta_F)/\alpha}}=\sqrt{\alpha/(1+\delta_F)}.
$$
It follows that
$$
\spec(\mathcal{K}_F)\subset\left\{z\in\mathbb{C}:\sqrt{\alpha/(1+\delta_F)}\leq |z|\leq \sqrt{(1+\delta_F)/\alpha}\right\}.
$$
To prove the inclusion for the approximate point spectrum, suppose that $\alpha<1$, let $z\in\mathbb{C}$ with $\sqrt{\alpha}<|z|<1/\sqrt{\alpha}$ and set
$$
g=\sum_{n\in\mathbb{Z}} z^n \chi_{I_n}.
$$
Given $\eta>0$ with $\sqrt{\alpha}<\sqrt{\eta}<|z|<1/\sqrt{\eta}<1/\sqrt{\alpha}$, let $\rho=\max\{|z|\sqrt{\eta},\sqrt{\eta}/|z|\}<1$. For sufficiently large $n$ we have $a_n\leq \eta.$ It follows that there exists a constant $C>0$ such that if $n\geq 0$, then
$$
|z|^{2n}|I_{n}|=|I_0||z|^{2n}\prod_{j=0}^{n-1} a_j\leq C (|z|\sqrt{\eta})^{2n}\leq C\rho^{2n},
$$
and if $n>1$ then
$$
|z|^{-2n}|I_{-n}|=|z|^{-2n}|I_{n-1}|=|I_0||z|^{-2n}\prod_{j=0}^{n-2} a_j\leq C (\sqrt{\eta}/|z|)^{2n}\leq C\rho^{2n}.
$$
It follows that
$$
\|g\|^2=\sum_{n\in\mathbb{Z}}|z|^{2n} |I_n|\leq 2C\sum_{n=0}^\infty \rho^{2n}<\infty.
$$
Since $F$ acts as a bijection between $I_n$ and $I_{n+1}$, $F^{-1}(I_{n+1})=I_n$ and hence $\mathcal{K}\chi_{I_n}=\chi_{I_{n-1}}$. It follows that $\mathcal{K}g=zg$ and hence $z$ is an eigenvalue. Since $z$ with $\sqrt{\alpha}<|z|<1/\sqrt{\alpha}$ was arbitrary, the closed annulus $\{z\in\mathbb{C}:\sqrt{\alpha}\leq |z|\leq 1/\sqrt{\alpha}\}$ lies in $\spec_{\mathrm{ap}}(\mathcal{K}_F)$. The final statement follows from the fact that we can take $\delta_{F_c}= 0$.
\end{proof}

\begin{proof}[Proof of \cref{thm:interval}]
The upper bound immediately follows from \cref{thm:general_koopman_computation}, so we only need to prove the lower bounds. We prove this for piecewise affine functions with exact input (which implies the result for $\Delta_1$-information), and the proof is almost identical when restricting to smooth functions in $\Omega_{[0,1]}^{\delta}$.

Suppose, for a contradiction, that $\{\Gamma_n\}$ is an SPGA for $\{\Xi_{\spec_{\mathrm{ap}}}, \Omega_{[0,1]}^{\delta},\MH, \Lambda_{[0,1]}\}$ with
\begin{equation}
\label{prob_cotradiction_smooth1}
\inf_{F\in\Omega_{[0,1]}^{\delta}}\mathbb{P}\left(\lim_{n\rightarrow\infty}\Gamma_{n}(F)=\spec_{\mathrm{ap}}(\mathcal{K}_F)\right)=\epsilon>1/2.
\end{equation}
We consider $F=F_c$ defined by acting as an affine function on each subinterval $I_n$, and choose the sequence $c=\{c_n\}_{n=0}^\infty$ to contradict \cref{prob_cotradiction_smooth1}. The probabilistic part of the proof is similar to step 2 of the proof of \cref{thm:disc}. 
Let $r_1\in(0,1)$ be such that
$
r_1^{-2}<1+\delta.
$
Set
$$
r_2=(1-r_1)^2+r_1>r_1,
$$
so that $(1-r_2)=r_1(1-r_1)$.
We may choose $\tau>0$ such that
$$
(1+2\tau)\sqrt{r_1}< (1-2\tau)\sqrt{r_2}.
$$
The numbers $r_1$, $r_2$, and $\tau$ are fixed throughout the proof. We use a superscript $(j)$ to denote choices of the sequences $\{c_n\}$ and corresponding objects such as the intervals $I_n^{(j)}$.

We begin with the sequence $c^{(1)}=\{c_n^{(1)}\}_{n=0}^\infty$ chosen so that
$$
|I_n^{(1)}|=\beta^{(1)}r_1^n,\quad n=0,1,2,\ldots.
$$
Here, the constant $\beta^{(1)}$ is chosen such that
$$
\sum_{n=-\infty}^\infty|I_n^{(1)}|=2\sum_{n=0}^\infty|I_n^{(1)}|=\frac{2\beta^{(1)}}{1-r_1}=1.
$$
Since $\max\{r_1,r_1^{-1}\}=r_1^{-1}<1+\delta$, $F_{c^{(1)}}\in \Omega_{[0,1]}^{\delta}$. The intervals $I_n^{(1)}$ have constant ratio $a_n^{(1)}=r_1$ and hence \cref{lem:annulus_phase} implies that
$$
\spec(\mathcal{K}_{F_{c^{(1)}}})=\spec_{\mathrm{ap}}(\mathcal{K}_{F_{c^{(1)}}})=\left\{z\in\mathbb{C}:\sqrt{r_1}\leq |z|\leq 1/\sqrt{r_1}\right\}.
$$
Applying the covering lemma (\cref{lem:prob_covering_lemma}), it follows that there exists $n_1,N_1\in\mathbb{N}$ such that
$$
\mathbb{P}\left(\inf\{|z|:z\in \Gamma_{n_1}(F_{c^{(1)}})\}\leq (1+\tau)\sqrt{r_1}\text{, }\Lambda_{\Gamma_{n_1}}(F_{c^{(1)}})\subset \cup_{j=-N_1}^{N_1-1}I^{(1)}_{j} \right)>1/2,
$$
where the notation $\Lambda_{\Gamma}(F)$ means that $\Gamma$ samples $F$ only within the set $\Lambda_{\Gamma}(F)\subset\mathcal{X}$ before producing its output. Moreover, we use the convention that $\inf\{|z|:z\in \Gamma_{n_1}(F_{c^{(1)}})\}\leq (1+\tau)\sqrt{r_1}$ necessarily implies that $\Gamma_{n_1}(F_{c^{(1)}})\neq \mathrm{NH}$.

Next, we define the sequence $c^{(2)}$ implicitly by
$$
|I_n^{(2)}|=\begin{cases}
|I_n^{(1)}|,\quad&\text{if }0\leq n\leq N_1,\\
\beta^{(2)}r_2^{n-(N_1+1)},\quad&\text{if }n> N_1.
\end{cases}
$$
Here, the factor $\beta^{(2)}$ is chosen so that
$$
\frac{\beta^{(2)}}{1-r_2}=\sum_{n={N_1}+1}^\infty|I_n^{(2)}|=\sum_{n={N_1}+1}^\infty|I_n^{(1)}|=|I_{N_1}^{(1)}|\sum_{n=1}^\infty r_1^n=|I_{N_1}^{(1)}|\frac{r_1}{1-r_1}.
$$
In particular,
$$
\frac{|I_{N_1+1}^{(2)}|}{|I_{N_1}^{(2)}|}=\frac{\beta^{(2)}}{|I_{N_1}^{(1)}|}=\frac{r_1(1-r_2)}{1-r_1}=r_1^2\leq 1< 1+\delta.
$$
Similarly, the reciprocal of this ratio is $r_1^{-2}<1+\delta$.
Hence, $F_{c^{(2)}}\in \Omega_{[0,1]}^{\delta}$ and
\cref{lem:annulus_phase} implies that
$$
\spec(\mathcal{K}_{F_{c^{(2)}}})=\spec_{\mathrm{ap}}(\mathcal{K}_{F_{c^{(2)}}})=\left\{z\in\mathbb{C}:\sqrt{r_2}\leq |z|\leq 1/\sqrt{r_2}\right\}.
$$
Applying the covering lemma (\cref{lem:prob_covering_lemma}), it follows that there exists $n_2,N_2\in\mathbb{N}$ such that $n_2>n_1$, $N_2>N_1$, and
$$
\mathbb{P}\left(\inf\{|z|:z\in \Gamma_{n_2}(F_{c^{(2)}})\}\geq (1-\tau)\sqrt{r_2}\text{, }\Lambda_{\Gamma_{n_2}}(F_{c^{(2)}})\subset \cup_{j=-N_2}^{N_2-1}I^{(2)}_{j}\right)>1/2.
$$
Again, we use the convention that $\inf\{|z|:z\in \Gamma_{n_2}(F_{c^{(2)}})\}\geq (1-\tau)\sqrt{r_2}$ necessarily implies that $\Gamma_{n_2}(F_{c^{(2)}})\neq \mathrm{NH}$.

We continue this process inductively. Let $\hat r_j=r_2$ if $j$ is even and $\hat r_j=r_1$ if $j$ is odd. Suppose that the intervals $I_n^{(k-1)}$ have been defined for $n\in\mathbb{Z}$ and $n_{k-1},N_{k-1}\in\mathbb{N}$ are such that
$$
\mathbb{P}\left(\inf\{|z|:z\in \Gamma_{n_{k-1}}(F_{c^{({k-1})}})\}\begin{rcases}
    \begin{dcases}
\leq (1+\tau)\sqrt{r_1},&\text{if $k$ is even}\\
\geq (1-\tau)\sqrt{r_2},&\text{if $k$ is odd}
\end{dcases}
  \end{rcases}\text{ and }\Lambda_{\Gamma_{n_{k-1}}}(F_{c^{(k-1)}})\subset \cup_{j=-N_{k-1}}^{N_{k-1}-1}I^{(k-1)}_{j}\right)>1/2.
$$
We define the sequence $c^{(k)}$ implicitly by
$$
|I_n^{(k)}|=\begin{cases}
|I_n^{(k-1)}|,\quad&\text{if }0\leq n\leq N_{k-1},\\
\beta^{(k)}\hat{r}_k^{n-(N_{k-1}+1)},\quad&\text{if }n> N_{k-1},
\end{cases}
$$
Here, the factor $\beta^{(k)}$ is chosen so that
\begin{align*}
\frac{\beta^{(k)}}{1-\hat{r}_k}=\sum_{n={N_{k-1}}+1}^\infty|I_n^{(k)}|=\sum_{n={N_{k-1}}+1}^\infty|I_n^{(k-1)}|=|I_{N_{k-1}}^{(k-1)}|\sum_{n=1}^\infty \hat{r}_{k-1}^n=|I_{N_{k-1}}^{(k-1)}|\frac{\hat{r}_{k-1}}{1-\hat{r}_{k-1}}.
\end{align*}
In particular, a case by case analysis shows that
$$
\frac{|I_{N_{k-1}+1}^{(k)}|}{|I_{N_{k-1}}^{(k)}|}=\frac{\beta^{(k)}}{|I_{N_{k-1}}^{(k-1)}|}=\frac{\hat{r}_{k-1}(1-\hat{r}_k)}{1-\hat{r}_{k-1}}\leq 1+\delta,\quad 
\frac{1-\hat{r}_{k-1}}{\hat{r}_{k-1}(1-\hat{r}_k)}\leq 1+\delta.
$$
Hence, $F_{c^{(k)}}\in \Omega_{[0,1]}^{\delta}$.
Using \cref{lem:annulus_phase}, we select 
$n_{k},N_k\in\mathbb{N}$ so that
\begin{equation}\label{eq:key_prob_int}
\mathbb{P}\left(\inf\{|z|:z\in \Gamma_{n_{k}}(F_{c^{({k})}})\}\begin{rcases}
    \begin{dcases}
\leq (1+\tau)\sqrt{r_1},\quad\text{if $k+1$ is even}\\
\geq (1-\tau)\sqrt{r_2},\quad\text{if $k+1$ is odd}
\end{dcases}
  \end{rcases}\text{ and }\Lambda_{\Gamma_{n_{k}}}(F_{c^{(k)}})\subset \cup_{j=-N_{k}}^{N_{k}-1}I^{(k)}_{j}\right)>1/2.
\end{equation}
and this completes the inductive step.

We can now define the sequence $c=\{c_n\}$ by taking the limit $c_n=\lim_{j\rightarrow\infty}c_n^{(j)}$. For any fixed $n$, $c_n^{(j)}$ is constant for large $j$ and hence $F_c\in \Omega_{[0,1]}^{\delta}$. For each $k\in\mathbb{N}$, let $B_k$ be the event
$$
\inf\{|z|:z\in \Gamma_{n_{k}}(F_{c})\}\begin{cases}
\leq (1+\tau)\sqrt{r_1},\quad&\text{if $k+1$ is even},\\
\geq (1-\tau)\sqrt{r_2},\quad&\text{if $k+1$ is odd}.
\end{cases}
$$
Due to \cref{eq:key_prob_int}, and the consistency of probabilistic general algorithms (see the discussion in \cref{exa:consistency_prob}), $\mathbb{P}(B_k)>1/2$ for all $k$. From the bound
$(1+2\tau)\sqrt{r_1}< (1-2\tau)\sqrt{r_2}$, we see that either $\inf\{|z|:z\in \spec_{\mathrm{ap}}(F_{c})\}>(1+\tau)\sqrt{r_1}$ or $\inf\{|z|:z\in \spec_{\mathrm{ap}}(F_{c})\}<(1-\tau)\sqrt{r_2}$ (or both). Assume without loss of generality that the former holds and let $A_m=\cap_{k=m}^\infty B_{2k+1}^c$. Note that $\mathbb{P}(A_m)\leq 1- \mathbb{P}(B_{2m+1})<1/2$. However, from the assumed lower bound of the probability of convergence in \cref{prob_cotradiction_smooth1},
$$
\mathbb{P}\left(\cup_{m=1}^\infty A_m\right)\geq \mathbb{P}\left(\lim_{n\rightarrow\infty}\Gamma_{n}(F_c)=\spec_{\mathrm{ap}}(\mathcal{K}_{F_c})\right)\geq \epsilon>1/2.
$$
Since $A_1\subset A_2\subset A_3 \subset \cdots$, there must exist some $A_M$ with $\mathbb{P}(A_M)>1/2$, the required contradiction.
\end{proof}

\subsection{Computing the spectral type is impossible in one limit}

We now consider the problem of computing the pure point spectrum of the Koopman operator and determining if there are any eigenvalues except $1$. This last problem is important since one usually wishes to know if there is coherency in the sense of \cref{eq:perfectly_coherent} for $\lambda\neq 1$. Our results also show that one cannot determine whether the spectral measure away from $\lambda=1$ is continuous. We consider the torus $\mathcal{X}=[0,2\pi]_{\mathrm{per}}\times[0,2\pi]_{\mathrm{per}}=[0,2\pi]_{\mathrm{per}}^2$, equipped with the standard, normalized, Lebesgue measure. We set
$$
\Omega_{p}=\left\{F:[0,2\pi]_{\mathrm{per}}^2
\rightarrow [0,2\pi]_{\mathrm{per}}^2
\text{ s.t. } F\text{ is smooth, measure-preserving, invertible},\mathrm{Lip}(F)\leq 2,\mathrm{Lip}(F^{-1})\leq 2\right\}.
$$
In other words, $\Omega_{p}$ includes both assumptions needed to compute spectra in one limit in \cref{thm:general_koopman_computation}. We consider two problems. The first is computing the pure point spectrum (closure of the set of eigenvalues), i.e., the problem function:
$$
\Xi_p:\Omega_{p}\ni F\mapsto \mathrm{Sp}_{\mathrm{pp}}(\mathcal{K}_F)=\mathrm{Cl}\left(\{\lambda\in\mathbb{C}:\lambda\text{ is an eigenvalue of }\mathcal{K}_F\}\right)\in \MH.
$$
The second is the following decision problem:
$$
\Xi_p^{\mathrm{dec}}:\Omega_{p}\ni F\mapsto \begin{rcases}
    \begin{dcases}
0,\quad&\text{if there are no eigenvalues except $1$}\\
1,\quad&\text{otherwise}
\end{dcases}
\end{rcases} \in \Md.
$$
The following theorem precisely classifies these two computational problems.

\begin{theorem}\label{thm:torus}
Given the above setup, we have the following classifications for non-trivial eigenvalues
$$
\Delta^{\mathbb{P},1/2}_2 \not\owns\{\Xi_p, \Omega_p,\MH, \Lambda_{[0,2\pi]_{\mathrm{per}}^2}\} \in \Sigma_2,\qquad
\Delta^{\mathbb{P},1/2}_2 \not\owns \{\Xi_p^\mathrm{dec}, \Omega_p,\Md, \Lambda_{[0,2\pi]_{\mathrm{per}}^2}\}  \in \Sigma_2.
$$
\end{theorem}

The following technical lemma is used in the proof of \cref{thm:torus} and the idea is summarized in \cref{fig:proof3}.

\begin{figure}
\centering
\raisebox{-0.5\height}{\includegraphics[width=0.8\textwidth,trim={0mm 0mm 0mm 0mm},clip]{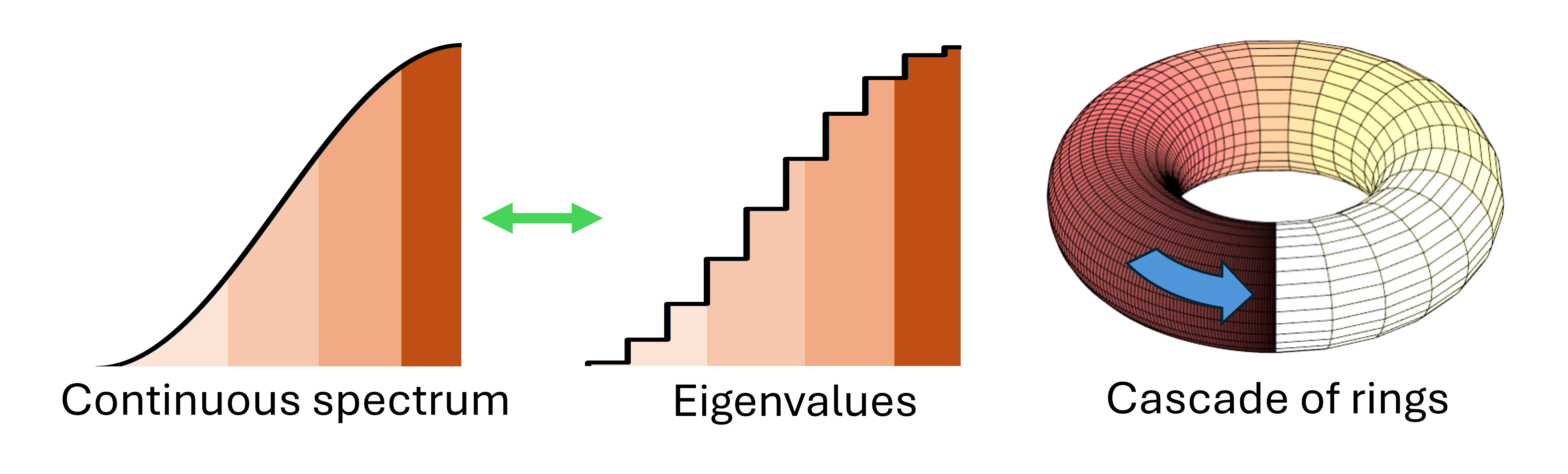}}
\caption{The idea of the proof of \cref{thm:torus}'s lower bound. We examine skew products on tori, where the spectral type depends on the existence of neighborhoods with constant cross-sections. The spectral change comes from a skew product constructed through iterations of locally piecewise constant and smooth approximations.}
\label{fig:proof3}
\end{figure}

\begin{lemma}\label{technical_lemma3}
Let $F_f\in\Omega_p$ be the skew product map
$
F_f(x,y)=(x,f(x)+y),
$
where $f:[0,2\pi]_{\mathrm{per}}\rightarrow [0,2\pi]_{\mathrm{per}}$ is a smooth function that is non-decreasing on $[0,2\pi)$ with $f(0)=0$. Then:
\begin{itemize}[leftmargin=0.7cm]
	\item[(a)] If $f$ is a bijection with $f'(x)>0$ for all $x\in[0,2\pi)$, $\mathrm{Sp}_{\mathrm{pp}}(\mathcal{K}_{F_f})=\{1\}$.
	\item[(b)] If there exists an open interval $(a,b)\subset[0,2\pi)$ upon which $f(x)=c$ is constant, then $\{e^{ijc}:j\in\mathbb{Z}\}\subset\mathrm{Sp}_{\mathrm{pp}}(\mathcal{K}_{F_f})$. 
\end{itemize}
\end{lemma}

\begin{proof}
For (a), suppose that $f$ is a bijection with $f'(x)>0$ for all $x\in[0,2\pi)$. Consider the smooth invertible map $\phi(x,y)=(f^{-1}(x),y)$. Then $(\phi^{-1}\circ F_f \circ \phi)(x,y)=(x,x+y)$ and 
$
\mathcal{K}_{\phi^{-1}\circ F_f \circ \phi}=\mathcal{K}_{\phi}\mathcal{K}_{F_f}[\mathcal{K}_{\phi}]^{-1}.
$
Hence, $\mathcal{K}_{F_f}$ is similar to the Koopman operator of the skew product map $(x,y)\mapsto(x,x+y)$. Similar operators have the same eigenvalues, so the statement in (a) now follows.
For (b), suppose that there exists an open interval $(a,b)\subset[0,2\pi)$ upon which $f(x)=c$ is constant. Let $g_j(x,y)=\chi_{(a,b)}(x)e^{ijy}$, where $\chi_S$ denotes the indicator function of a set $S$. Then $\mathcal{K}_{F_f}g_j=e^{ijc}g_j$ and the result follows.
\end{proof}

We also need the following discrete-time RAGE theorem \cite[Proof of Theorem A.2]{fillman2017purely}. The continuous-time RAGE theorem \cite[Theorem 2.6]{aizenman2015random}, named after Ruelle \cite{ruelle1969remark}, Amrein and Georgescu \cite{amrein1973characterization}, and Enss \cite{enss1978asymptotic}, is a classical dynamical characterization of the continuous spectrum of self-adjoint operators, and is a valuable tool in the study of Schr\"odinger operators.

\begin{theorem}
\label{thm:RAGE_discrete_time}
Let $A$ be a unitary operator acting on a separable Hilbert space $\mathcal{H}$, and $\{\mathcal{P}_n\}_{n\in\mathbb{N}}$ a sequence of increasing finite-rank orthogonal projections such that $\mathcal{P}_n^*\mathcal{P}_n$ converges strongly to the identity. (We view $\mathcal{P}_n$ as an operator from $\mathcal{H}$ to its range, hence the need for $\mathcal{P}_n^*$.) Let $\mathcal{P}_{\mathrm{c}}$ and $\mathcal{P}_{\mathrm{pp}}$ denote the orthogonal projections onto $\mathcal{H}_{\mathrm{c}}$ (continuous part) and $\mathcal{H}_{\mathrm{pp}}$ (pure point part), respectively. Then for any $g\in\mathcal{H}$,
$$
\|\mathcal{P}_{\mathrm{c}}g\|^2=\lim_{n\rightarrow\infty}\lim_{L\rightarrow\infty}
\frac{1}{2L+1}\sum_{{\ell}=-L}^L\left\|(I-\mathcal{P}_n^*\mathcal{P}_n)A^{\ell}g\right\|^2,\qquad
\|\mathcal{P}_{\mathrm{pp}}g\|^2=\lim_{n\rightarrow\infty}\lim_{L\rightarrow\infty}
\frac{1}{2L+1}\sum_{\ell=-L}^L\left\|\mathcal{P}_nA^{\ell}g\right\|^2.
$$
\end{theorem}

\begin{proof}[Proof of \cref{thm:torus}]
We first prove the upper bounds and then prove the lower bounds.

\textbf{Step 1: Upper bounds.} Let $\{e_j\}_{j=1}^\infty$ be an orthonormal basis of $L^2(\mathcal{X},\omega)$ constructed using a tensor product of Fourier bases. Using the arguments in Step 1 of the proof of \cref{thm:general_koopman_computation} and those of \cite{colbrook2019computing}, for any $z\not\in\mathbb{T}$ and any $g\in L^2(\mathcal{X},\omega)$ that is a finite linear combination of the $\{e_j\}_{j=1}^\infty$, we can compute $(\mathcal{K}_F-zI)^{-1}g$ with error control using the given $\Delta_1$-information. The proof follows the arguments in \cite[Theorem 2.1]{colbrook2019computing}. It follows from the unitary version of \cite[Theorem 5.1.1]{colbrook2020PhD} that $\{\Xi_p, \Omega_p,\MH, \Lambda_{[0,2\pi]_{\mathrm{per}}^2}\} \in \Sigma_2$. To deal with $\Xi_p^{\mathrm{dec}}$, let $\mathcal{P}_n$ be the orthogonal projection onto $\mathrm{span}\{e_1,\ldots,e_n\}$. Let $g$ be a finite linear combination of the $\{e_j\}_{j=1}^\infty$. For $n\in\mathbb{N}$ and $\ell\in\mathbb{Z}$, using the unitary version of \cite[Theorem 4.3.3 \& Remark 4.3.4]{colbrook2020PhD} (which uses \cref{thm:RAGE_discrete_time}), we may compute $W_{n_2,n_1}(g)$ (see \cref{U_spec_meas}) such that
$$
\lim_{n_1\rightarrow\infty}W_{n_2,n_1}(g)=W_{n_2}(g),\quad \lim_{n_2\rightarrow\infty}W_{n_2}(g)=\mu_g^{(\mathrm{pp})}(\mathbb{T}\backslash\{1\}),
$$
where $\mu_g^{(\mathrm{pp})}$ is the pure point part of the spectral measure of $\mathcal{K}_F$ with respect to $g$ and the convergence as $n_2\rightarrow\infty$ is from below. Let $\{g_k\}_{k=1}^\infty$ be a set of such $g$ that form a dense subset of $L^2(\mathcal{X},\omega)$. We then set
$$
a_{n_2,n_1}(F)=\max_{1\leq k\leq n_2}W_{n_2,n_1}(g_k).
$$
Define the two separated intervals $I_1=[0,1/4]$ and $I_2=[1/2,\infty)$. As $n_1\rightarrow\infty$, $a_{n_2,n_1}(F)$ converges to $a_{n_2}(F)=\max_{1\leq k\leq n_2}W_{n_2}(g_k)$ and hence cannot visit both $I_1$ and $I_2$ infinitely often. For a given $n_1$, we set $\Gamma_{n_2,n_1}(F)=0$ if the largest $l=1,\ldots,n_1$ with $a_{n_2,l}\in I_1\cup I_2$ has $a_{n_2,l}\in I_1$. If no such $l$ exists, or $a_{n_2,l}\in I_2$, we set $\Gamma_{n_2,n_1}(F)=1$. If $\Xi_p^{\mathrm{dec}}(F)=0$, then $\mu_{g_k}^{(\mathrm{pp})}(\mathbb{T}\backslash\{1\})=0$ for all $k$ and hence $a_{n_2}(F)=0$ for all $n_2$. It follows that $\lim_{n_1\rightarrow\infty} \Gamma_{n_2,n_1}(F)=0$. If $\Xi_p^{\mathrm{dec}}(F)=1$, then there exists $g_k$ with $\mu_{g_k}^{(\mathrm{pp})}(\mathbb{T}\backslash\{1\})>1/2$ and hence $a_{n_2}(F)>1/2$ for sufficiently large $n_2$. It follows that $\lim_{n_1\rightarrow\infty} \Gamma_{n_2,n_1}(F)=1$ for sufficiently large $n_2$. Moreover, since $a_{n_2,n_1}(F)$ are increasing in $n_2$, $\Gamma_{n_2}(F)$ is increasing in $n_2$. It follows that $\{\Xi_p^{\mathrm{dec}}, \Omega_p,\Md, \Lambda_{[0,2\pi]_{\mathrm{per}}^2}\} \in \Sigma_2$. The process is summarized in \cref{U_spec_meas2}.

\textbf{Step 2: Lower bounds.}
Throughout the proof, we write $(x,y)\in[0,2\pi]^2_{\mathrm{per}}$ to denote the state. We prove the lower bound for the subclass of maps $F_q\in\Omega_p$ of the form $F_q(x,y)=(x,q(x)+y)$, where $q$ has the properties outlined in \cref{technical_lemma3} (smooth, non-decreasing on $[0,2\pi)$ with $q(0)=0$). This implies the lower bound for the full class $\Omega_p$. Without loss of generality and with a slight abuse of notation, we may restrict our argument to an evaluation set that samples $q$ at a dense set of points $\{\hat{x}_j\}_{j=1}^\infty\subset(0, 2\pi)$. Suppose, for a contradiction, that $\{\Gamma_n\}$ is an SPGA for $\{\Xi_p, \Omega_p,\MH, \Lambda_{[0,2\pi]_{\mathrm{per}}^2}\}$ with the property that
\begin{equation}
\label{prob_contradiction_type}
\inf_{F_q\in\Omega_p}\mathbb{P}\left(\lim_{n\rightarrow\infty}\Gamma_{n}(\tilde{q})=\mathrm{Sp}_{\mathrm{pp}}(\mathcal{K}_{F_q})\right)=\epsilon>1/2.
\end{equation}
Here, $\tilde{q}$ represents inexact input for $q$. That is, an admissible $\tilde{q}$ is of the form $\{f_{j,n}(q):j,n\in\mathbb{N}\}$, where $|f_{j,n}(q)-q(\hat{x}_j)|\leq 2^{-n}$. We will construct a function $q$ with inexact input $\tilde{q}$ so that the bound in \cref{prob_contradiction_type} cannot hold. The function $q$ will be a pointwise limit of functions $q^{(k)}$. Throughout, $q^{(k)}$ will be assumed to satisfy the properties outlined in \cref{technical_lemma3} and $F_{q^{(k)}}\in\Omega_p$, even if not explicitly stated.

\textbf{Base case:} We begin with a bijection $q^{(1)}$ such that $(q^{(1)})'(x)>0$ for all $x\in[0,2\pi)$. Case (a) of \cref{technical_lemma3} implies that $\Xi_p(F_{q^{(1)}})=\{1\}$. Let $\tilde{q}^{(1)}$ be an arbitrary set of $\Delta_1$-information for $q^{(1)}$. Applying the covering lemma (\cref{lem:prob_covering_lemma}), there exists $n_1\in\mathbb{N}$, a finite set $\mathcal{I}_1\subset\mathbb{N}^2$ and $r_1\in(0,2\pi)$ such that the following two conditions hold. First, letting $J_1$ be the set of $j$ for which there exists $n$ with $(j,n)\in\mathcal{I}_{1}$, it holds that $\hat{x}_j< r_1$ whenever $j\in J_{1}$. Second,
$$
\mathbb{P}\left(\mathrm{dist}(-1,\Gamma_{n_1}(\tilde{q}^{(1)}))\geq 1 \text{ and }C_1(\tilde{q}^{(1)})\right)>1/2,
$$
where $C_1(\tilde{q})$ is the event that $\Gamma_{n_1}$ samples from $\{f_{j,n}(q):(j,n)\in\mathcal{I}_1\}$ and outputs $\Gamma_{n_1}(\tilde{q})\neq\mathrm{NH}$. This completes the base case.

\textbf{Inductive step:} Suppose $q^{(k)}$, $\tilde{q}^{(k)}=\{f_{j,n}(q^{(k)}):j,n\in\mathbb{N}\}$, $n_k\in\mathbb{N}$ and $r_k\in(0,2\pi)$ are defined. We consider two cases, depending on the parity of $k$. 

If $k$ is odd, we choose $q^{(k+1)}$ so that $q^{(k+1)}(x)=q^{(k)}(x)$ for $x\in(0,r_k]$ and $q^{(k+1)}$ is constant on an open interval containing $r_k/2+\pi$, where it takes a value $c_{k+1}$ with $c_{k+1}/\pi\notin\mathbb{Q}$. We define $\tilde{q}^{(k+1)}$, the $\Delta_1$-information for $q^{(k+1)}$, as follows. If $\hat{x}_j\leq r_{k}$, then $f_{j,n}(q^{(k+1)})=f_{j,n}(q^{(k)})$, otherwise $f_{j,n}(q^{(k+1)})=q^{(k+1)}(\hat{x}_j)$. Case (b) of \cref{technical_lemma3} implies that $\Xi_p(F_{q^{(k+1)}})=\mathbb{T}$. Applying the covering lemma (\cref{lem:prob_covering_lemma}), there exists $n_{k+1}\in\mathbb{N}$ with $n_{k+1}\geq n_k$, a finite set $\mathcal{I}_{k+1}\subset\mathbb{N}^2$ with $\mathcal{I}_{k}\subset \mathcal{I}_{k+1}$, and $r_{k+1}\in(r_k/2+\pi,2\pi)$ such that the following two conditions hold. First, letting $J_{k+1}$ be the set of $j$ for which there exists $n$ with $(j,n)\in\mathcal{I}_{k+1}$, it holds that $\hat{x}_j< r_{k+1}$ whenever $j\in J_{k+1}$. Second,
$$
\mathbb{P}\left(\mathrm{dist}(-1,\Gamma_{n_{k+1}}(\tilde{q}^{(k+1)}))\leq 1/2\text{ and }C_{k+1}(\tilde{q}^{(k+1)})\right)>1/2,
$$
where $C_{k+1}(\tilde{q})$ is the event that $\Gamma_{n_{k+1}}$ samples from $\{f_{j,n}(q):(j,n)\in\mathcal{I}_{k+1}\}$ and outputs $\Gamma_{n_{k+1}}(\tilde{q})\neq\mathrm{NH}$.

If $k$ is even, we can choose $q^{(k+1)}$ with $(q^{(k+1)})'(x)>0$ for all $x\in[0,2\pi)$ and $\Delta_1$-information $\tilde{q}^{(k+1)}=\{f_{j,n}(q^{(k+1)}):j,n\in\mathbb{N}\}$ so that
$$
f_{j,n}(q^{(k+1)})=f_{j,n}(q^{(k)})\quad \forall (j,n)\in \mathcal{I}_k,\quad q^{(k+1)}(x)=q^{(k)}(x)=q^{(k-1)}(x) \quad\forall x\in(0,r_{k-1}].
$$
The first of these conditions can be achieved since $f_{j,n}(q^{(k)})=q^{(k)}(\hat{x}_j)$ if $\hat{x}_j> r_{k-1}$ and we may slightly perturb $q^{(k)}$ to achieve a strictly increasing function, whilst ensuring $f_{j,n}(q^{(k+1)})$ is still admissible for $q^{(k+1)}$. The second condition is consistent with how we selected the functions in the inductive step for odd $k$. Case (a) of \cref{technical_lemma3} implies that $\Xi_p(F_{q^{(k+1)}})=\{1\}$. Applying the covering lemma (\cref{lem:prob_covering_lemma}), there exists $n_{k+1}\in\mathbb{N}$ with $n_{k+1}\geq n_k$, a finite set $\mathcal{I}_{k+1}\subset\mathbb{N}^2$ with $\mathcal{I}_{k}\subset \mathcal{I}_{k+1}$, and $r_{k+1}\in(r_k/2+\pi,2\pi)$ such that the following two conditions hold. First, letting $J_{k+1}$ be the set of $j$ for which there exists $n$ with $(j,n)\in\mathcal{I}_{k+1}$, it holds that $\hat{x}_j< r_{k+1}$ whenever $j\in J_{k+1}$. Second,
$$
\mathbb{P}\left(\mathrm{dist}(-1,\Gamma_{n_{k+1}}(\tilde{q}^{(k+1)}))\geq 1\text{ and }C_{k+1}(\tilde{q}^{(k+1)})\right)>1/2,
$$
where $C_{k+1}(\tilde{q})$ is the event that $\Gamma_{n_{k+1}}$ samples from $\{f_{j,n}(q):(j,n)\in\mathcal{I}_{k+1}\}$ and outputs $\Gamma_{n_{k+1}}(\tilde{q})\neq\mathrm{NH}$. This completes the inductive step.

\textbf{Limit argument:} Since $q^{(2m+1)}(x)=q^{(2m)}(x)=q^{(2m-1)}(x)$ for $x\in(0,r_{2m-1}]$ for all $m\in\mathbb{N}$ and $\lim_{m\rightarrow\infty}r_{2m-1}=2\pi$, the following pointwise limit exists:
$$
q(x)=\lim_{m\rightarrow\infty}q^{(2m+1)}(x).
$$
The function $q$ satisfies the conditions in \cref{technical_lemma3}. Furthermore, we can perform the construction so that $F_q\in\Omega_p$. We define $\tilde{q}$ by setting $f_{j,n}(q)=f_{j,n}(q^{(k)})$ if $(j,n)\in\mathcal{I}_k$ for $k=1,2,\ldots$ and $f_{j,n}(q)=q(\hat{x}_j)$ if $(j,n)\not\in\cup_{k=1}^\infty\mathcal{I}_k$. For each $k\in\mathbb{N}$, let $B_k$ be the event
$$
\mathrm{dist}(-1,\Gamma_{n_k}(\tilde{q}))\begin{cases}
\geq 1,\quad&\text{if $k$ is odd},\\
\leq 1/2,\quad&\text{if $k$ is even}.
\end{cases}
$$
Due to our construction and the consistency of probabilistic general algorithms, we have $\mathbb{P}(B_k)>1/2$ for all $k$. We now argue as in the proof of \cref{thm:interval} to obtain the required contradiction.
\end{proof}

\subsection{Spectra of discrete-space systems}

We now consider the state space $\mathcal{X}=\mathbb{N}$, equipped with the usual counting measure $\omega=\sum_{j=1}^\infty \delta_j$, and consider the dynamical systems governed by a function $F:\mathbb{N}\rightarrow \mathbb{N}$. We define the Koopman operator on sequences $x:\mathbb{N}\rightarrow\mathbb{C}$ via
$$
[\mathcal{K}_Fx](j)=x(F(j)).
$$
We consider $\mathcal{K}_F$ as an operator on $L^2(\mathbb{N},\omega)\cong l^2(\mathbb{N})$, and assume that it is bounded. Our input class and evaluation sets are 
$$
\Omega_{\mathbb{N}}=\left\{F\text{ s.t. }F:\mathbb{N}\rightarrow\mathbb{N}\text{ and }\mathcal{K}_F\text{ is bounded}\right\},\quad
\Lambda_{\mathbb{N}}=\{F\mapsto F(j):j\in\mathbb{N}\}.
$$
We are interested in computing $\spec_{\mathrm{ap}}(\mathcal{K}_F)$.

\begin{theorem}\label{thm:discrete_space}
Given the above setup, we have the following classifications:
$$
\Delta_3\not\owns\{\Xi_{\spec_{\mathrm{ap}}}, \Omega_{\mathbb{N}},\MH, \Lambda_{\mathbb{N}}\}\in \Pi_3.
$$
In other words, we cannot compute the approximate point spectrum in two limits via any algorithm. However, we can obtain an arithmetic $\Pi_3$-tower.
\end{theorem}

To prove \cref{thm:discrete_space}, our strategy will be to embed a certain combinatorial problem into the spectral problem of interest. This problem will have a known lower bound complexity, allowing us to prove the lower bound in \cref{thm:discrete_space}. Specifically, let $\Omega_{\mathrm{Mat}}$ be the collection of all infinite matrices $a=\{a_{m_1,m_2}\}_{m_1,m_2\in\mathbb{N}}$ with entries $a_{m_1,m_2}\in\{0,1\}$ and $\Lambda_{\mathrm{Mat}}$ be the set of component-wise evaluation functions. We consider the formula
\begin{align*}
Q(a)&=\begin{cases}
1,\quad\text{if for all but a finite number of } i,\,\forall j\, \exists n>j \text{ s.t. }a_{n,i}=1,\\
0,\quad\text{otherwise}.
\end{cases}
\end{align*}
In other words, $Q$ decides whether the matrix has only finitely many columns with only finitely many $1$'s. It was proven in \cite{colbrook4} that $\{Q,\Omega_{\mathrm{Mat}},[0,1],\Lambda_{\mathrm{Mat}}\}\not\in\Delta_3$. Note that the metric space is not $\Md$, but $[0,1]$ with the usual metric. We shall also need the following technical lemma and the idea is summarized in \cref{fig:proof4}.

\begin{figure}
\centering
\raisebox{-0.5\height}{\includegraphics[width=0.6\textwidth,trim={0mm 0mm 0mm 0mm},clip]{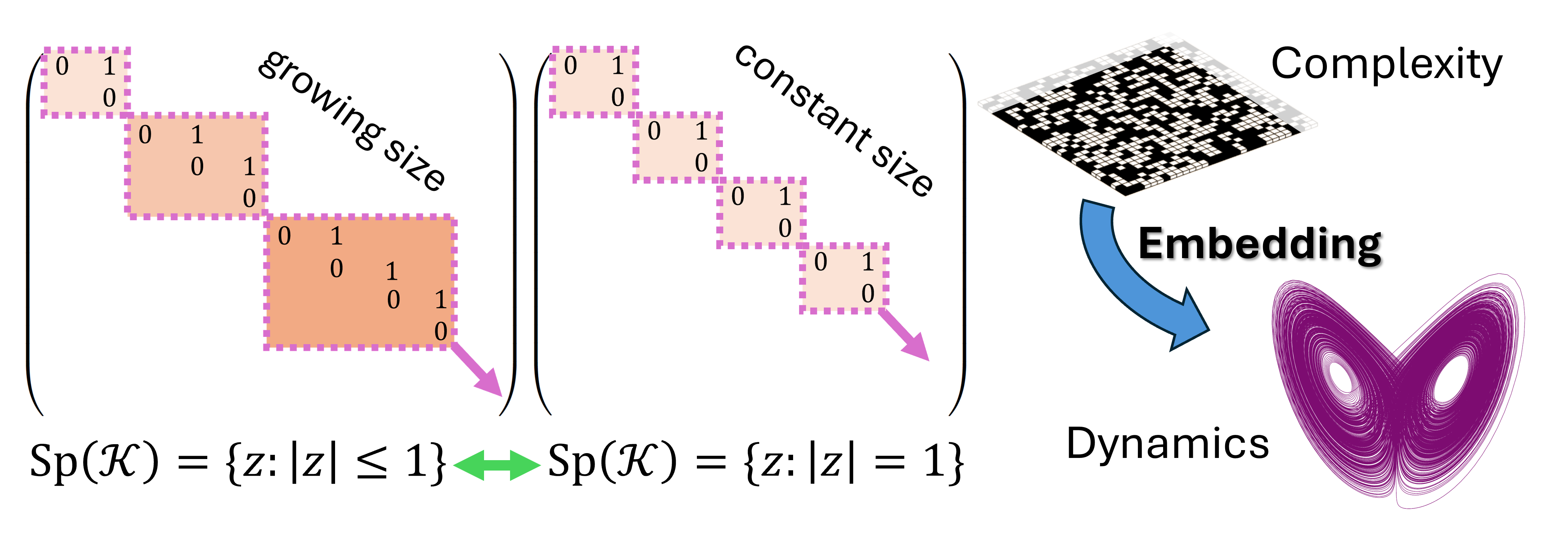}}
\caption{Outline of the proof of the lower bound in \cref{thm:discrete_space}. We embed a combinatorial problem involving matrices of $0$s and $1$s such that the spectrum fundamentally changes depending on the solution. This construction is analogous to symbolic dynamics, where dynamical systems are studied by encoding orbits as sequences of symbols from a finite alphabet. Many chaotic systems are semi-conjugate or conjugate to a shift on such sequences.}
\label{fig:proof4}
\end{figure}

\begin{lemma}\label{lem:shifts_phase}
For $n\in\mathbb{N}$, define the perturbed shift matrix
$$
B_n=
\begin{pmatrix}
1& & & \\
1  &0& & \\
  & \ddots&\ddots&\\
 & & 1&0\\
\end{pmatrix}\in\mathbb{C}^{n\times n}
$$
and let $B_{\infty}$ denote the corresponding limit operator on $l^2(\mathbb{N})$. Then $\spec_{\mathrm{ap}}(B_\infty)=\mathbb{T}$. For a sequence $\{l_r\}_{r=1}^\infty\subset\mathbb{N}$, define $B(\{l_r\})=\bigoplus_{r=1}^{\infty} B_{l_r}$. If $\{l_r\}$ is bounded, $\spec_{\mathrm{ap}}(B(\{l_r\}))\subset\{0,1\}$, otherwise $\spec_{\mathrm{ap}}(B(\{l_r\}))=\{z\in\mathbb{C}:|z|\leq 1\}$.
\end{lemma}

\begin{proof}
For the first part, recall that $\spec_{\mathrm{ap}}(B_\infty)$ is the set of points $z$ for which there exists a unit-norm sequence $\{u_n\}\subset l^2(\mathbb{N})$ with $\lim_{n\rightarrow\infty}\|(B_\infty-zI)u_n\|=0$. Let $S_\infty$ denote the operator on $l^2(\mathbb{N})$ defined by $e_{j}\mapsto e_{j+1}$. By considering the action of $B_\infty$ on $\mathrm{span}\{e_1\}^\perp$, we see that $\spec_{\mathrm{ap}}(S_\infty)\subset \spec_{\mathrm{ap}}(B_\infty)$. Since $S_\infty$ is a non-invertible isometry, $\spec_{\mathrm{ap}}(S_\infty)= \mathbb{T}$. Now suppose that $z\not\in\mathbb{T}$. Then for any unit-norm $u\in l^2(\mathbb{N})$
$$
\|(S_\infty-zI)u\|\leq \|(B_\infty-zI)u\| +\|u_1e_1\|\leq \|(B_\infty-zI)u\|+\frac{1}{|1-z|}\|(B_\infty-zI)u\|.
$$
Since $z\not\in\spec_{\mathrm{ap}}(S_\infty)$, $\|(S_\infty-zI)u\|$ is bounded below independently of the choice of $u$ and hence so is $\|(B_\infty-zI)u\|$. It follows that $z\not\in\spec_{\mathrm{ap}}(B_\infty)$.

Suppose that the sequence $\{l_r\}$ is bounded and that $z\not\in\{0,1\}$. Then $(B_{l_r}-zI)^{-1}$ exists for all $r$ and is uniformly bounded. Hence, $z\not\in \spec_{\mathrm{ap}}(B(\{l_r\}))$. 

For the final case, suppose that the sequence $\{l_r\}$ is unbounded. Let $S_n$ be the matrix $B_n$, but with the top-left entry set to zero, and define $S(\{l_r\})=\bigoplus_{r=1}^{\infty} S_{l_r}$. By arguing as in the first part of the proof, if $z\neq 1$ and $u\in l^2(\mathbb{N})$, then
$$
\|(S(\{l_r\})-zI)u\|\leq \|(B(\{l_r\})-zI)u\|+\frac{1}{|1-z|}\|(B(\{l_r\})-zI)u\|.
$$
Hence, $\spec_{\mathrm{ap}}(B(\{l_r\}))\subset \spec_{\mathrm{ap}}(S(\{l_r\}))\cup\{1\}\subset\{z\in\mathbb{C}:|z|\leq 1\}$. Let $z\in\mathbb{C}$ with $|z|<1$ and define $u_z=e_n+ze_{n-1}+\cdots +z^{n-1}e_1\in\mathbb{C}^n$. Then, $(B_n -zI)u_z= (1-z)z^{n-1}e_1$ and hence $\lim_{n\rightarrow\infty}\|(B_n -zI)u_z\|/\|u_z\|=0$. By considering $u_z$ for $n$ corresponding to an increasing subsequence of $\{l_r\}$, we see that $z\in \spec_{\mathrm{ap}}(B(\{l_r\}))$. Hence, $\spec_{\mathrm{ap}}(B(\{l_r\}))=\{z\in\mathbb{C}:|z|\leq 1\}$.
\end{proof}

\begin{proof}[Proof of \cref{thm:discrete_space}]
\textbf{Step 1: Upper bound.} As an orthonormal basis of $L^2(\mathbb{N},\omega)$, we consider $\{\phi_n\}_{n=1}^\infty$, where $\phi_n(j)=\delta_{n,j}$. Given $F\in\Omega_{\mathbb{N}}$, let $A_F$ be the matrix representing $\mathcal{K}_F$ with respect to this basis. We can evaluate any entry of $A_F$ using finitely many evaluations from $\Lambda_{\mathbb{N}}$. The upper bound now follows from the $\Pi_3$ towers of algorithms in \cite{ben2015can}.

\textbf{Step 2: Lower bound.} First, given a sequence $\{c_i\}_{i\in\mathbb{N}}$ with each $c_i\in\{0,1\}$, we define the bounded operator $C(\{c_i\})$ acting on $l^2(\mathbb{N})$ via its matrix entries:
$$
[C(\{c_i\})]_{k,l}=\begin{cases}
1,\quad& l<k,  c_l=c_k=1, \text{ and } c_n=0 \text{ for all }l<n<k,\\
1,\quad&l=k, c_k=1, \text{ and } c_n=0 \text{ for all }n<k,\\
1,\quad&l=k\text{ and }c_k=0,\\
0,\quad& \text{ otherwise.}
\end{cases}
$$
Then $C(\{c_i\})$ acts as $B_{\sum_i c_i}$ on the closure of the span of $\{e_i:c_i=1\}$ and the identity on the orthogonal complement of this subspace. Given $a\in\Omega_{\mathrm{Mat}}$, we consider the family of sequences
$$
c^{(j)}_i=c^{(j)}_i(a)=
\begin{cases}
1,\quad& i\leq j,\\
a_{i+1-j,j},\quad& \text{ otherwise}.
\end{cases}
$$
We then consider the following operator acting on $\oplus_{j=1}^\infty l^2(\mathbb{N})$:
$
A(a)= \oplus_{j=1}^\infty C(\{c^{(j)}_i\}).
$
\cref{lem:shifts_phase} shows that if $Q(a)=1$, then $\spec_{\mathrm{ap}}(A(a))\subset\{0\}\cup\mathbb{T}$. Otherwise, $\spec_{\mathrm{ap}}(A(a))=\{z\in\mathbb{C}:|z|\leq 1\}$.

Suppose for a contradiction that $\{\Gamma_{n_2,n_1}\}$ is a height-two tower of general algorithms for $\{\Xi_{\spec_{\mathrm{ap}}}, \Omega_{\mathbb{N}},\MH, \Lambda_{\mathbb{N}}\}$. Given $a\in\Omega_{\mathrm{Mat}}$, we consider the function $F=F_a$ defined as follows. We can identify $\oplus_{j=1}^\infty l^2(\mathbb{N})$ with $l^2(\mathbb{N})$ via a computable bijection between the relevant index sets, and, hence, consider $A(a)$ as a bounded operator on $l^2(\mathbb{N})$. When viewed as an infinite matrix, this matrix has exactly one $1$ in each row, and every other entry is $0$. Moreover, the operator is bounded. By considering this matrix acting with respect to the basis $\{\phi_n\}_{n=1}^\infty$, this matrix induces $F_a$ and a corresponding Koopman operator $\mathcal{K}_{F_a}$. Since we know that each row of $A(a)$ has exactly one non-zero entry, any evaluation ${F_a}(j)$ of $F_a$ can be evaluated in finitely many calls to $\Lambda_{\mathrm{Mat}}$ with input $a$. It follows that
$
\tilde{\Gamma}_{n_2,n_1}(a)=\min\left\{2\times\dist(1/2,\Gamma_{n_2,n_1}(F_a)),1\right\}
$
defines general algorithms that map into $[0,1]$ and use the evaluation set $\Lambda_{\mathrm{Mat}}$. Moreover, the above discussion shows that $\lim_{n_2\rightarrow\infty}\lim_{n_1\rightarrow\infty}\tilde{\Gamma}_{n_2,n_1}(a)=Q(a)$, the required contradiction.
\end{proof}

\section{Pseudocode for algorithms}
\label{section:pseudocodes}

Here, we provide pseudocode for the algorithms that form the upper bounds in our classifications. The proofs of convergence can be found in the proofs of \cref{thm:general_koopman_computation}. Each algorithm uses a dictionary $\{g_j\}_{j=1}^\infty$. For simplicity, we have given pseudocode for when these basis functions are orthonormal (to be consistent with the proofs), but they can be implemented for non-orthonormal systems.

\subsection{Algorithms for the approximate point spectrum and pseudospectrum}

We begin by computing the approximate point spectrum and pseudospectrum of Koopman operators. Recall from \cref{sec:general_ResDMD_method} that a modulus of continuity $\alpha$ for $F$ allows us to approximate the integrals $\langle \mathcal{K}_Fg_j,g_i \rangle$ and $\langle \mathcal{K}_Fg_j,\mathcal{K}_Fg_i \rangle$ with error bounds using snapshot data. By considering the pencil
$$
z\mapsto (\mathcal{K}_F-zI)^*(\mathcal{K}_F-zI),
$$
we can use these approximations of inner products to compute
\begin{equation}
\label{bdbd}
h_n(z,F)=\sqrt{\sigma_{\mathrm{inf}}(\mathcal{P}_n(\mathcal{K}_F-zI)^*(\mathcal{K}_F-zI)\mathcal{P}_n^*)},
\end{equation}
where $\mathcal{P}_n$ denotes the orthogonal projection onto $\mathrm{span}\{g_1,\ldots,g_n\}$. The functions $h_n(z,F)$ converge locally uniformly (i.e., uniformly on any compact subset of $\mathbb{C}$) to $\sigma_{\inf}(\mathcal{K}_F-zI)$. One always has
$$
\sigma_{\inf}(\mathcal{K}_F-zI)\leq \mathrm{dist}(z,\spec_{\mathrm{ap}}(\mathcal{K}_F)).
$$
To compute $\spec_{\mathrm{ap}}(\mathcal{K}_F)$, we need to compute an upper bound on $\mathrm{dist}(z,\spec_{\mathrm{ap}}(\mathcal{K}_F))$, so we require an inequality that is the `other way round'.
If the system is measure-preserving, then $\sigma_{\inf}(\mathcal{K}_F-zI)= \mathrm{dist}(z,\spec_{\mathrm{ap}}(\mathcal{K}_F))$. We can generalize this as follows.

\begin{definition}
\label{def:gs}
Let $\mathfrak{G}:\mathbb{R}_{\geq 0}\rightarrow\mathbb{R}_{\geq 0}$ be a strictly increasing function with $\mathfrak{G}(x)\geq x$ and $\lim_{x\rightarrow\infty}\mathfrak{G}(x)=\infty$. 
We say that $\mathcal{K}_F$ has its resolvent bounded by $\mathfrak{G}$ if
\begin{equation}\label{eq:comp_ball_gs}
\dist(z,\spec_{\mathrm{ap}}(\mathcal{K}_F))\leq\mathfrak{G}(\sigma_{\inf}(\mathcal{K}_F-zI))\quad \forall z\in \mathbb{C}.
\end{equation}
The set of all such systems with modulus of continuity $\alpha$ is denoted by $\Omega_{\mathcal{X}}^{\alpha,\mathfrak{G}}$.
\end{definition}

A simple compactness argument shows that, for any bounded $\mathcal{K}_F$, a suitable $\mathfrak{G}$ always exists so that $\mathcal{K}_F$ has resolvent bounded by $\mathfrak{G}$. However, given $\mathcal{K}_F$, we may not know a suitable $\mathfrak{G}$. As with the modulus of continuity $\alpha$, this uncertainty underlies some of our impossibility results. Typically, we want an algorithm that can handle a whole class of systems, such as $\Omega_{\mathcal{X}}^{\alpha,\mathfrak{G}}$, rather than a single system. Beyond measure-preserving systems, there are dissipative systems where $\mathfrak{G}$ is known \cite{gil2003operator} and for the systems in this paper, one can take $\mathfrak{G}(x)=cx$ for a constant $c$ that depends on the non-normality of the operator (approximated through pseudospectra).

\begin{algorithm}[t]
\textbf{Input:} Function $F$ (in the form of snapshots $\{F(\hat{x}_j):j=1,2,\ldots\}$), dictionary of functions $\{g_j\}_{j=1}^{\infty}$, modulus of continuity $\alpha$, resolvent bounding function $\mathfrak{G}$, $n\in\mathbb{N}$.\\
\vspace{-4mm}
\begin{algorithmic}[1]
\State Compute the $n\times n$ matrices $$
A_{i,j}=
\langle \mathcal{K}_Fg_j,g_i \rangle=\int_{\mathcal{X}} g_j(F(x))\overline{g_i(x)}\dd \omega(x),\quad
L_{i,j}=\langle \mathcal{K}_Fg_j,\mathcal{K}_Fg_i \rangle
=\int_{\mathcal{X}} g_j(F(x))\overline{g_i(F(x))}\dd \omega(x),\quad i,j=1,\ldots n,
$$
using quadrature (evaluation of $F$ at snapshots) and $\alpha$ to bound quadrature errors (see proof of \cref{thm:general_koopman_computation}).
\State Define the finite grid
$$
\mathrm{Grid}(n)=\left\{z\in\frac{1}{n}\mathbb{Z}+\frac{i }{n}\mathbb{Z}:|z|\leq n\right\}.
$$
\State For each $z\in \mathrm{Grid}(n)$, compute
$$
h_{n}(z,F)=\sqrt{\sigma_{\mathrm{inf}}(L-\overline{z}A-zA^*+|z|^2I)},\quad
d_{n,z}=\mathfrak{G}\left(h_{n}(z,F)\right).
$$
Compute approximate eigenfunctions (corresponding right-singular vectors) if desired.
\State For each $z\in \mathrm{Grid}(n)$, set $\Upsilon_{n,z}=\{w\in \mathrm{Grid}(n):|z-w|\leq d_{n,z}\}$. If $h_{n}(z,F)\leq 1/2$, compute the set of local minimizers
$$
M_z=\left\{w\in \Upsilon_{n,z} : h_{n}(w,F)={\min\limits_{v\in \Upsilon_{n,z}}}h_{n}(v,F) \right\}.
$$
Otherwise, set $M_z=\emptyset$.
\State Set $\Gamma_n(F)=\cup_{z\in \mathrm{Grid}(n)}M_z$ and $E_n(w)=d_{n,w}$ for $w\in\Gamma_n(F)$.
\end{algorithmic} \textbf{Output:} 
$\Gamma_n(F)$, an approximation of $\spec_{\mathrm{ap}}(\mathcal{K}_F)$, and the error function $E_n$.
\caption{A $\Sigma_1$-tower of algorithms for computing the spectrum of $\mathcal{K}_F$ for $F\in\Omega_{\mathcal{X}}^{\alpha,\mathfrak{G}}$. We have written the code for an orthonormal dictionary of observables, but the algorithm is easily adapted to the non-orthonormal case.}
\label{pseudo_code_easy_spec}
\end{algorithm}

\cref{pseudo_code_easy_spec} shows the $\Sigma_1^A$-tower of algorithms for the approximate point spectrum of systems in $\Omega_{\mathcal{X}}^{\alpha,\mathfrak{G}}$. The idea is to search for local minimizers of $h_n(z,F)$, where the bound
$$
d_{n,z}=\mathfrak{G}\left(h_{n}(z,F)\right)\geq \mathfrak{G}\left(\sigma_{\inf}(\mathcal{K}_F-zI)\right)
\geq \dist(z,\spec_{\mathrm{ap}}(\mathcal{K}_F))
$$
provides a local search radius and the sets $M_z$ are our best estimate of the approximate point spectrum locally near the point $z\in \mathrm{Grid}(n)$. In addition to providing error bounds, the algorithm is local, trivially parallelizable (owing to the separate computation for different $z$ points), and stable.

Without the function $\mathfrak{G}$, we can still compute the approximate point pseudospectrum with error control. This is due to the fact that
$$
\spec_{\mathrm{ap},\epsilon}(\mathcal{K}_F)=\left\{z\in\mathbb{C}:\sigma_{\inf}(\mathcal{K}_F-zI)\leq \epsilon\right\}
$$
is defined directly in terms of $\sigma_{\inf}(\mathcal{K}_F-zI)$. The process is summarized in \cref{pseudo_code_easy_pseudospec} and is a generalization of the ResDMD algorithm in \cite{colbrook2021rigorousKoop}. By taking $\epsilon=1/n_2$, we obtain a $\Pi_2$-tower for the approximate point spectrum, summarized in \cref{pseudo_code_medium_spec}.

\begin{algorithm}[t]
\textbf{Input:} Function $F$ (in the form of snapshots $\{F(\hat{x}_j):j=1,2,\ldots\}$), dictionary of functions $\{g_j\}_{j=1}^{\infty}$, modulus of continuity $\alpha$, $n\in\mathbb{N}$, $\epsilon>0$.\\
\vspace{-4mm}
\begin{algorithmic}[1]
\State Compute the $n\times n$ matrices $$
A_{i,j}=
\langle \mathcal{K}_Fg_j,g_i \rangle=\int_{\mathcal{X}} g_j(F(x))\overline{g_i(x)}\dd \omega(x),\quad
L_{i,j}=\langle \mathcal{K}_Fg_j,\mathcal{K}_Fg_i \rangle
=\int_{\mathcal{X}} g_j(F(x))\overline{g_i(F(x))}\dd \omega(x),\quad i,j=1,\ldots n,
$$
using quadrature (evaluation of $F$ at snapshots) and $\alpha$ to bound quadrature errors.
\State Define the finite grid
$$
\mathrm{Grid}(n)=\left\{z\in\frac{1}{n}\mathbb{Z}+\frac{i }{n}\mathbb{Z}:|z|\leq n\right\}.
$$
\State For each $z\in \mathrm{Grid}(n)$, compute
$
h_{n}(z,F)=\sqrt{\sigma_{\mathrm{inf}}(L-\overline{z}A-zA^*+|z|^2I)}.
$
\State Set $\Gamma_n^\epsilon(F)=\{z\in \mathrm{Grid}(n):h_{n}(z,F)<\epsilon\}$.
\end{algorithmic} \textbf{Output:} 
$\Gamma_n^\epsilon(F)$, an approximation of $\spec_{\mathrm{ap},\epsilon}(\mathcal{K}_F)$.
\caption{A $\Sigma_1$-tower of algorithms for computing the pseudospectrum of $\mathcal{K}_F$ for $F\in\Omega_{\mathcal{X}}^{\alpha}$. We have written the code for an orthonormal dictionary of observables, but the algorithm is easily adapted to the non-orthonormal case.}
\label{pseudo_code_easy_pseudospec}
\end{algorithm}

\begin{algorithm}[t]
\textbf{Input:} Function $F$ (in the form of snapshots $\{F(\hat{x}_j):j=1,2,\ldots\}$), dictionary of functions $\{g_j\}_{j=1}^{\infty}$, modulus of continuity $\alpha$, $n_1,n_2\in\mathbb{N}$.\\
\vspace{-4mm}
\begin{algorithmic}[1]
\State Set $\Gamma_{n_2,n_1}(F)$ to be the output of \cref{pseudo_code_easy_pseudospec} with $\epsilon=1/n_2$ and $n=n_1$.
\end{algorithmic} \textbf{Output:} 
$\Gamma_{n_2,n_1}(F)$, an approximation of $\spec_{\mathrm{ap}}(\mathcal{K}_F)$.
\caption{A $\Pi_2$-tower of algorithms for computing the spectrum of $\mathcal{K}_F$ for $F\in\Omega_{\mathcal{X}}^{\alpha}$. We have written the code for an orthonormal dictionary of observables, but the algorithm is easily adapted to the non-orthonormal case.}
\label{pseudo_code_medium_spec}
\end{algorithm}

To compute the approximate point pseudospectrum of systems in $\Omega_{\mathcal{X}}$, we can no longer approximate the inner products $\langle \mathcal{K}_Fg_j,g_i \rangle$ and $\langle \mathcal{K}_Fg_j,\mathcal{K}_Fg_i \rangle$ with explicit error bounds (or rates) but can compute them at the cost of a limit. This then requires an alteration, as laid out in the proof of \cref{thm:general_koopman_computation}, and summarized in \cref{pseudo_code_medium_pseudospec}, to ensure convergence. Again, we may take $\epsilon=1/n_3$ to obtain a $\Pi_3$-tower of algorithms for the approximate point spectrum, summarized in \cref{pseudo_code_hard_spec}.

\begin{algorithm}[t]
\textbf{Input:} Function $F$ (in the form of snapshots $\{F(\hat{x}_j):j=1,2,\ldots\}$), dictionary of functions $\{g_j\}_{j=1}^{\infty}$, $n_1,n_2\in\mathbb{N}$, $\epsilon>0$.\\
\vspace{-4mm}
\begin{algorithmic}[1]
\State Let $A'$ and $L'$ be $n_1$-point quadrature approximations of the $n_2\times n_2$ matrices $$
A_{i,j}=
\langle \mathcal{K}_Fg_j,g_i \rangle=\int_{\mathcal{X}} g_j(F(x))\overline{g_i(x)}\dd \omega(x),\quad
L_{i,j}=\langle \mathcal{K}_Fg_j,\mathcal{K}_Fg_i \rangle
=\int_{\mathcal{X}} g_j(F(x))\overline{g_i(F(x))}\dd \omega(x),\quad i,j=1,\ldots n_2,
$$
using evaluation of $F$ at snapshots.
\State Define the finite grid
$$
\mathrm{Grid}(n_2)=\left\{z\in\frac{1}{n_2}\mathbb{Z}+\frac{i }{n_2}\mathbb{Z}:|z|\leq n_2\right\}.
$$
\State For each $z\in \mathrm{Grid}(n_2)$, compute
$
h_{n_2,n_1}(z,F)=\sqrt{\sigma_{\mathrm{inf}}(L'-\overline{z}A'-zA'^*+|z|^2I)}.
$
\State Consider the separated intervals
$$
I_{n_2}^1(\epsilon)=[0,\epsilon-1/n_2],\quad I_{n_2}^2(\epsilon)=[\epsilon+1/(2n_2),\infty).
$$
Given $h_{n_2,j}(z,F)$ for $j=1,\ldots,n_1$, let $k$ be the largest such $j$ with $h_{n_2,j}(z,F)\in I_{n_2}^1(\epsilon)\cup I_{n_2}^2(\epsilon)$. If such a $k$ exists with $h_{n_2,k}(z,F)\in I_{n_2}^1(\epsilon)$, then $z\in \Gamma_{n_2,n_1}^\epsilon(F)$. Otherwise, $z\notin \Gamma_{n_2,n_1}^\epsilon(F)$.
\end{algorithmic} \textbf{Output:} 
$\Gamma_{n_2,n_1}^\epsilon(F)$, an approximation of $\spec_{\mathrm{ap},\epsilon}(\mathcal{K}_F)$.
\caption{A $\Sigma_2$-tower of algorithms for computing the pseudospectrum of $\mathcal{K}_F$ for $F\in\Omega_{\mathcal{X}}$. We have written the code for an orthonormal dictionary of observables, but the algorithm is easily adapted to the non-orthonormal case.}
\label{pseudo_code_medium_pseudospec}
\end{algorithm}

\begin{algorithm}[t]
\textbf{Input:} Function $F$ (in the form of snapshots $\{F(\hat{x}_j):j=1,2,\ldots\}$), dictionary of functions $\{g_j\}_{j=1}^{\infty}$, $n_1,n_2,n_3\in\mathbb{N}$.\\
\vspace{-4mm}
\begin{algorithmic}[1]
\State Set $\Gamma_{n_3,n_2,n_1}(F)$ to be the output of \cref{pseudo_code_medium_pseudospec} with $\epsilon=1/n_3$.
\end{algorithmic} \textbf{Output:} 
$\Gamma_{n_3,n_2,n_1}(F)$, an approximation of $\spec_{\mathrm{ap}}(\mathcal{K}_F)$.
\caption{A $\Pi_3$-tower of algorithms for computing the spectrum of $\mathcal{K}_F$ for $F\in\Omega_{\mathcal{X}}$. We have written the code for an orthonormal dictionary of observables, but the algorithm is easily adapted to the non-orthonormal case.}
\label{pseudo_code_hard_spec}
\end{algorithm}

\subsection{Algorithms for spectral type and pure point spectrum}

To compute the pure point part of the spectral measure, we apply the method of \cref{thm:torus}, which can be generalized to unitary Koopman operators on arbitrary $\mathcal{X}$. The method is based on the RAGE theorem (\cref{thm:RAGE_discrete_time}). Recall that $\{\mathcal{P}_n\}_{n\in\mathbb{N}}$ is a sequence of increasing finite-rank orthogonal projections such that $\mathcal{P}_n^*\mathcal{P}_n$ converges strongly to the identity. For example, we can take $\mathcal{P}_n$ to be the orthogonal projection onto $\mathrm{span}\{g_1,\ldots,g_n\}$ (as in the previous subsection). \cref{U_spec_meas} summarises the procedure, where we consider the pure point part of the spectral measure of an observable $g$ measured on an open set $U$. In practice, the indicator function $\chi_{U}(\mathcal{K}_F)$ is computed in steps 2 and 3 using the measure-preserving EDMD algorithm \cite{colbrook2023mpedmd}, which provides a unitary Galerkin approximation of $\mathcal{K}_F$. The approximation of the functional calculus is outlined in \cite[Section 5]{colbrook2023mpedmd}. To decide whether the pure point spectrum of $\mathcal{K}_F$ (away from $1$) is empty, we alter \cref{U_spec_meas} as outlined in the proof of \cref{thm:torus}. The idea is to set $U=\mathbb{T}\backslash\{1\}$ and consider a set $\{g_k\}_{k=1}^\infty$ that form a dense subset of $L^2(\mathcal{X},\omega)$. The process is summarized in \cref{U_spec_meas2}.

\begin{algorithm}[t]
\textbf{Input:} Function $F$ (in the form of snapshots $\{F(\hat{x}_j):j=1,2,\ldots\}$), dictionary of functions $\{g_j\}_{j=1}^{\infty}$, observable $g\in L^2(\mathcal{X},\omega)$, open set $U\subset \mathbb{T}$, $n_1,n_2\in\mathbb{N}$.
\begin{algorithmic}[1]
\State Compute the $n_2\times n_2$ matrix
$$
A_{i,j}=
\langle \mathcal{K}_Fg_j,g_i \rangle=\int_{\mathcal{X}} g_j(F(x))\overline{g_i(x)}\dd \omega(x),\quad i,j=1,\ldots,n_2.
$$
\State Compute an SVD of $A^*=U_1\Sigma U_2^*$.
\State Compute the eigendecomposition $U_2U_1^*=V\Lambda V^*$ and let $\Lambda_U$ and $V_{U}$ be the subset of eigenpairs with eigenvalues in $U$.
\State Compute $m_k=\|\Lambda_U^kV_U^*g\|^2/(2n_1+1)$ for $k=-n_1,\ldots,n_1$.
\State Set $\Gamma_{n_2,n_1}(\mathcal{K}_F,g,U)=\min\left\{\sum_{k=-n_1}^{n_1}m_k,\|g\|^2\right\}$.
\end{algorithmic} \textbf{Output:} $\Gamma_{n_2,n_1}(\mathcal{K}_F,g,U)$, an approximation of $\mu_g^{\mathrm{(pp)}}(U)$.
\caption{A $\Sigma_2$-tower of algorithms for computing the pure point part of spectral measures of a general unitary Koopman operator $\mathcal{K}_F$.}\label{U_spec_meas}
\end{algorithm}

\begin{algorithm}[t]
\textbf{Input:} Function $F$ (in the form of snapshots $\{F(\hat{x}_j):j=1,2,\ldots\}$), dense set $\{g_k\}_{k=1}^\infty$ of observables, $n_1,n_2\in\mathbb{N}$.
\begin{algorithmic}[1]
\State For $k=1,\ldots,n_2$, let $W_{n_2,n_1}(g_k)$ be the output of \cref{U_spec_meas} with $g=g_k$ and $U=\mathbb{T}\backslash\{1\}$.
\State Set
$$
a_{n_2,n_1}(F)=\max_{1\leq k\leq n_2}W_{n_2,n_1}(g_k).
$$
\State Define the two separated intervals $I_1=[0,1/4]$ and $I_2=[1/2,\infty)$. 
\State Set $\Gamma_{n_2,n_1}(F)=0$ if the largest $l=1,\ldots,n_1$ with $a_{n_2,l}\in I_1\cup I_2$ has $a_{n_2,l}\in I_1$. If no such $l$ exists, or $a_{n_2,l}\in I_2$, set $\Gamma_{n_2,n_1}(F)=1$. 
\end{algorithmic} \textbf{Output:} $\Gamma_{n_2,n_1}(F)$, an approximation of $\Xi_p^{\mathrm{dec}}(F)$.
\caption{A $\Sigma_2$-tower of algorithms for the decision problem  $\Xi_p^{\mathrm{dec}}(F)$, which decides whether the pure point spectrum of a unitary Koopman operator $\mathcal{K}_F$ is larger than $\{1\}$.}\label{U_spec_meas2}
\end{algorithm}

\section{Further experimental diagnostics}

\subsection{Further analysis for Duffing oscillator}

\begin{figure}[t]
\centering
\includegraphics[width=0.4\textwidth,clip,trim={0mm 0mm 0mm 0mm}]{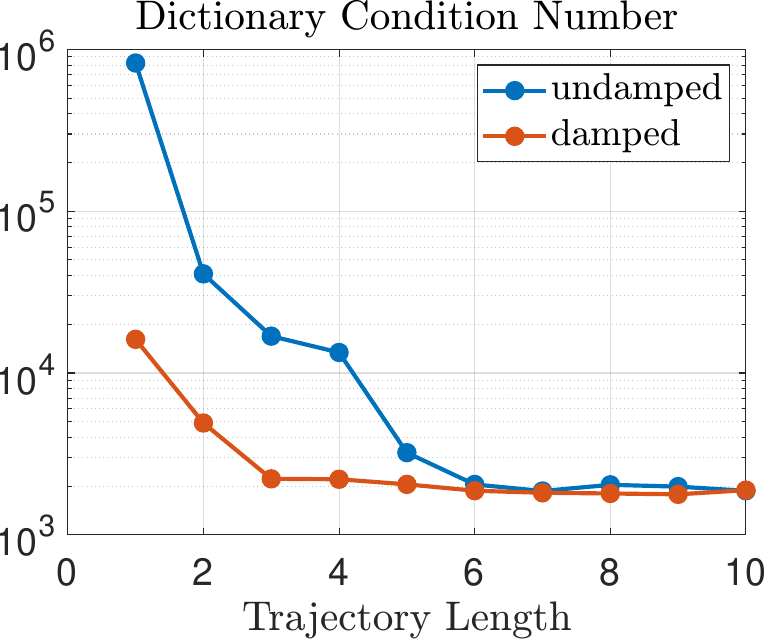}
\caption{Conditioning of the dictionary for $N=500$ and various trajectory lengths for the first set of data used to find the centers of the radial basis functions. The choice of trajectory length 5 produces a well-conditioned dictionary for the examples in this paper.}\label{duffing_condition}
\end{figure}

\cref{duffing_condition} shows the condition number of the dictionary for the Duffing oscillator example, computed using $k$-means clustering for different initial trajectory lengths (shown on the horizontal axis). The condition number is defined as follows: we evaluate the dictionary on a second, independent set of data points for Monte Carlo estimation, yielding an $M\times N$ matrix. The condition number of this matrix approximates how well-conditioned the dictionary is. A trajectory length of 5 produces a well-conditioned basis and is used for all examples in the paper.

\begin{figure}[t]
\centering
\includegraphics[width=0.4\linewidth,trim={0mm 0mm 0mm 0mm},clip]{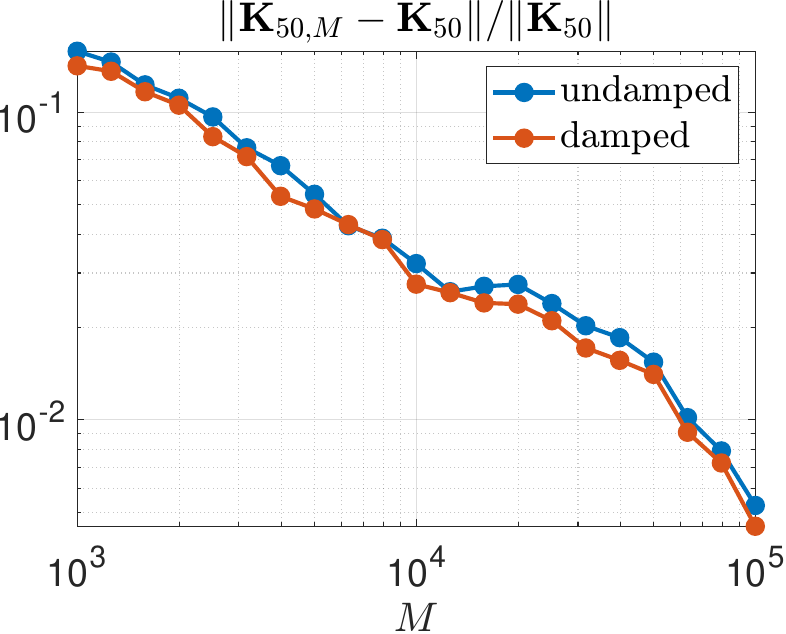}
\includegraphics[width=0.4\linewidth,trim={0mm 0mm 0mm 0mm},clip]{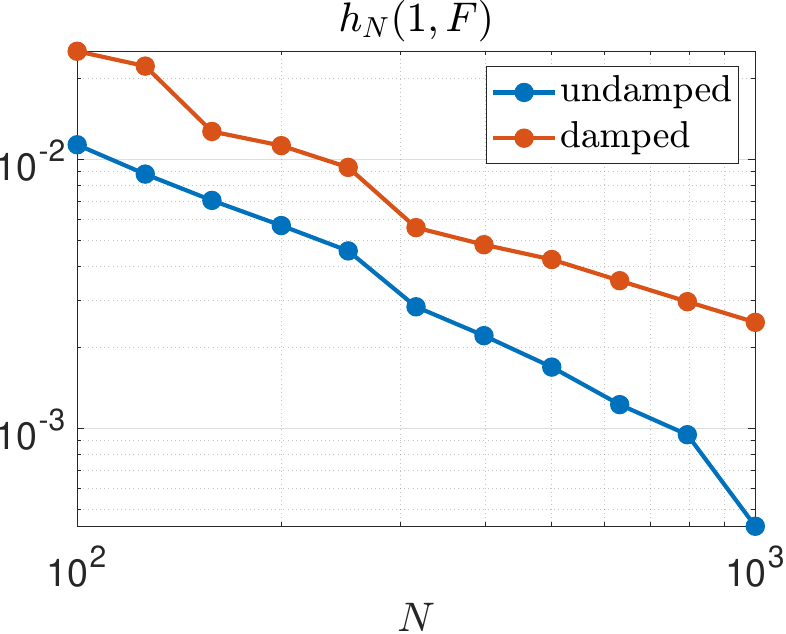}
\caption{Left: Convergence of the EDMD matrix to its large-data limit, with relative error measured in the Frobenius norm. The observed convergence rate is approximately $\mathcal{O}(M^{-1/2})$. Right: Convergence of the spectral distance function $h_{N}(z,F)$ at $z=1$, a representative spectral point. The results confirm the theoretical convergence guarantees of our method.}\label{duffing_convergence}
\end{figure}

\cref{duffing_convergence} illustrates convergence for the Duffing oscillator example. The left panel shows the convergence of the EDMD matrix ${\bf{K}}_{N,M}$ (computed as outlined in the main text with a dictionary of $N$ observables and $M$ snapshot data points) to its large-data limit (after taking $M\rightarrow\infty$ so that corresponding correlations have converged) ${\bf{K}}_{N}$ for $N=50$.
The error is measured as the relative Frobenius norm, the square root of the sum of the squared absolute values of the matrix entries. The convergence rate is approximately the Monte Carlo rate $\mathcal{O}(M^{-1/2})$ \cite{caflisch1998monte}. The right panel shows the convergence of the function $h_{N}(z,F)$ (see \cref{bdbd}) for $z=1$, which is chosen to demonstrate convergence since it lies in the spectrum.

\subsection{Further analysis for cavity flow}

Here, we plot the extracted eigenfunctions for the cavity flow, illustrating how their structure is related to the geometry of the attractor at different Reynolds numbers. To display the attractor, we use the total kinetic energy as one coordinate and the delayed kinetic energy with a time delay of 1 second. For $Re = 13,000$, the attractor forms a limit cycle, so we use one time delay (i.e., the attractor in a two-dimensional state space). For larger Reynolds numbers, we use two time delay coordinates (i.e., the attractor in a three-dimensional state space).

To extract the eigenfunctions, we begin with a 20-dimensional subspace of observables, constructed from time delays of the kinetic energy as described in the Methods section of the main text for this example. We then apply \cref{U_spec_meas} (the algorithmic realization of the RAGE theorem) to project these onto the eigenspaces of the Koopman operator, effectively removing any continuous spectral components. Next, we run \cref{pseudo_code_easy_spec} to compute the associated spectrum on this subspace, which yields the eigenvalues and eigenfunctions. We use the full $M=20000$ snapshots in all cases.

\begin{figure}[t]
\centering
\includegraphics[width=0.4\textwidth,clip,trim={0mm 0mm 0mm 0mm}]{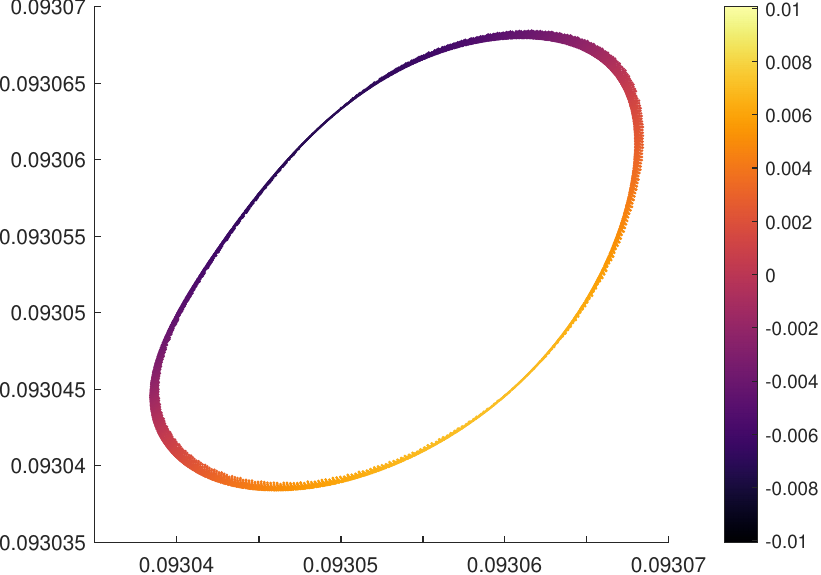}\hspace{5mm}
\includegraphics[width=0.4\textwidth,clip,trim={0mm 0mm 0mm 0mm}]{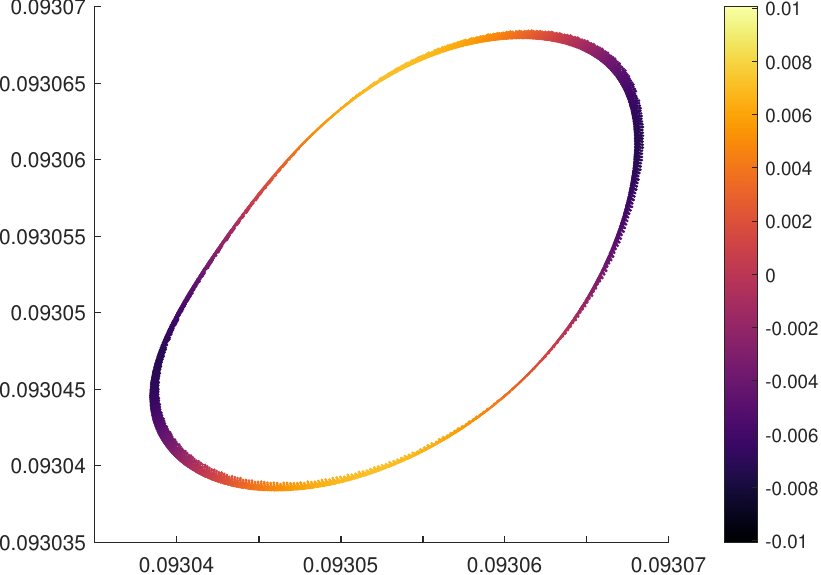}
\caption{Limit cycle of the cavity flow at $Re = 13,000$. The horizontal axis represents the kinetic energy, and the vertical axis represents the time-delayed kinetic energy. The plot shows the real part of the eigenfunction corresponding to the fundamental eigenvalue $\lambda$ (left) and its square $\lambda^2$ (right).}\label{cavity_extra1}
\end{figure}

\begin{figure}[t]
\centering
\includegraphics[width=0.4\textwidth,clip,trim={0mm 0mm 0mm 0mm}]{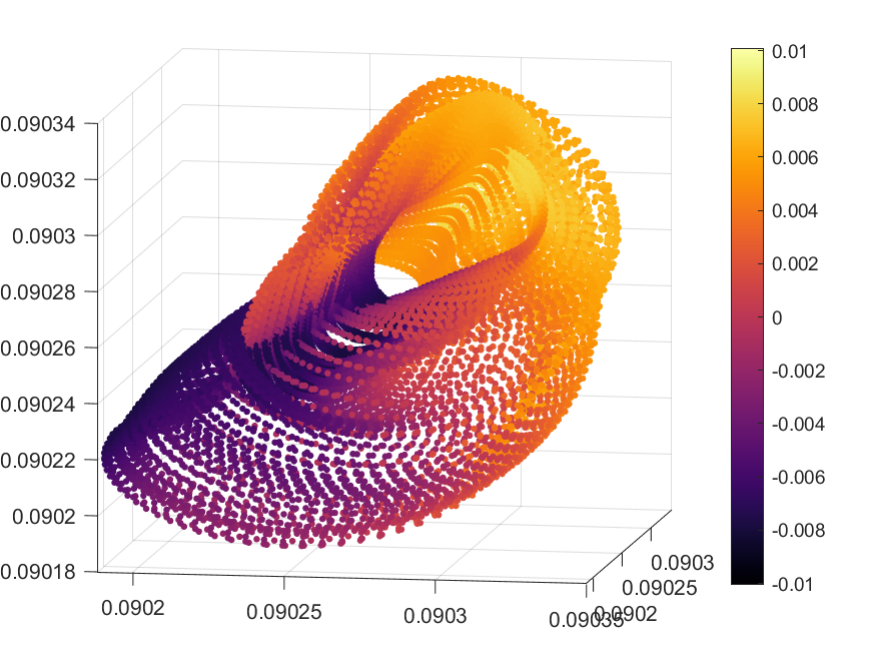}\hspace{5mm}
\includegraphics[width=0.4\textwidth,clip,trim={0mm 0mm 0mm 0mm}]{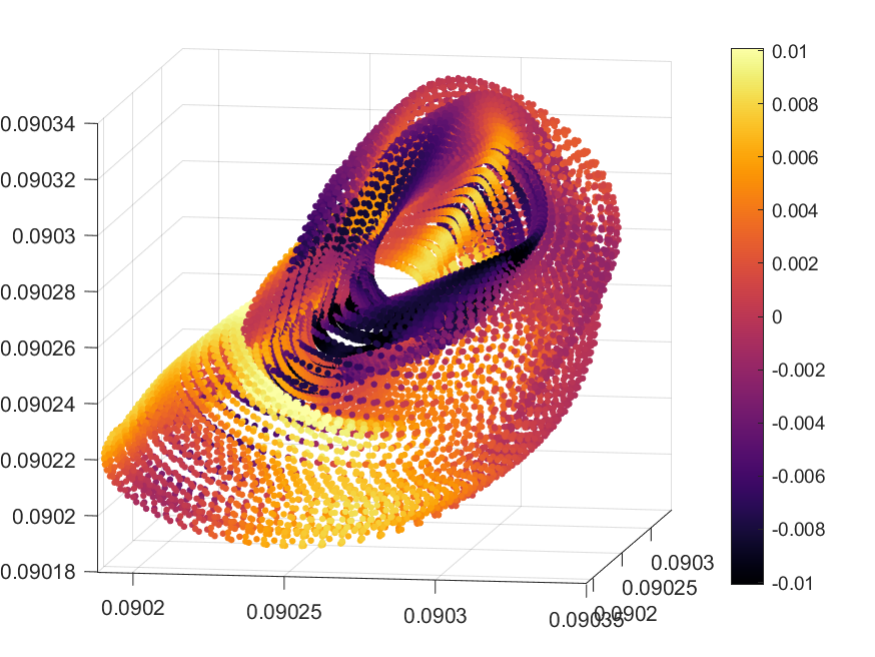}
\caption{Torus attractor of the cavity flow at $Re = 16,000$. The axes represent the kinetic energy and two time-delayed kinetic energies. The plot shows the real part of the eigenfunctions corresponding to the fundamental eigenvalues $\lambda$ and $\mu$.}\label{cavity_extra2}
\end{figure}

\begin{figure}[t]
\centering
\includegraphics[width=0.4\textwidth,clip,trim={0mm 0mm 0mm 0mm}]{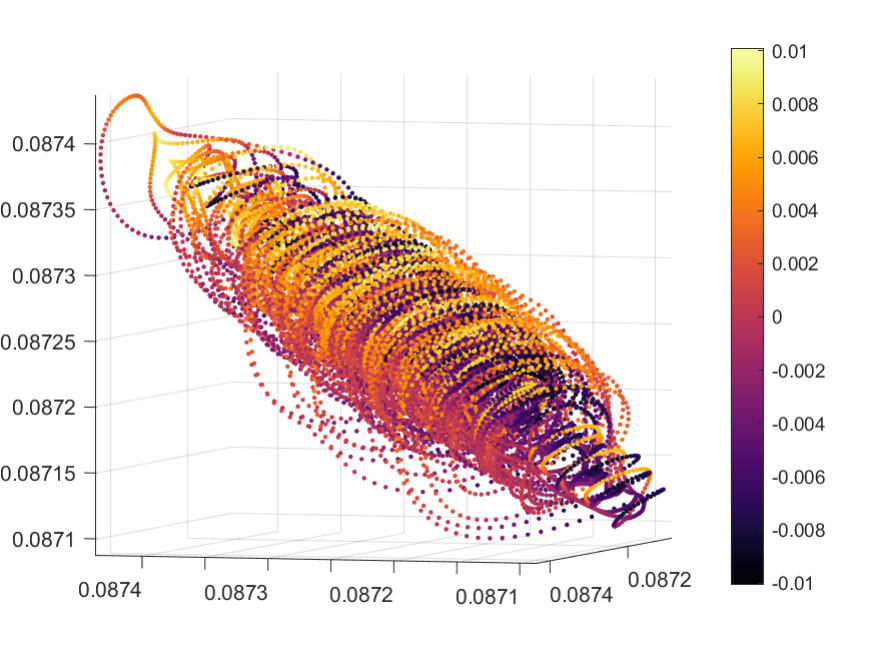}\hspace{5mm}
\includegraphics[width=0.4\textwidth,clip,trim={0mm 0mm 0mm 0mm}]{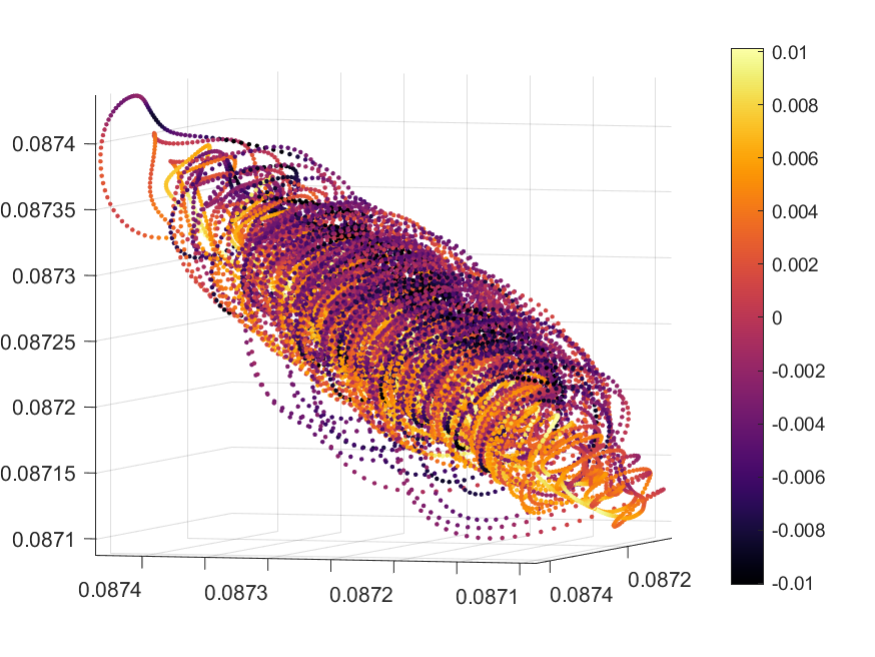}
\caption{Skew-periodic attractor of the cavity flow at $Re = 19,000$. The axes represent the kinetic energy and two time-delayed kinetic energies. The plot shows the real part of the eigenfunctions corresponding to the two basic frequencies of the discrete part of the spectrum.}\label{cavity_extra3}
\end{figure}

The results are shown in \cref{cavity_extra1,cavity_extra2,cavity_extra3} for various Reynolds numbers. For $Re = 13,000$, the attractor is a limit cycle. The plot displays the eigenfunction corresponding to the fundamental eigenvalue $\lambda$ (left) and its square $\lambda^2$ (right). The color represents the real part of the eigenfunction at each point. The doubling of frequency from left to right is consistent with a Koopman spectrum consisting of powers $\{\lambda^n\}$. For $Re = 16,000$, the attractor is a torus, corresponding to two fundamental eigenvalues, $\lambda$ and $\mu$ (quasiperiodic dynamics). The Koopman eigenfunctions reveal the directions on the torus where the evolution is linear and periodic. For $Re = 19,000$, the attractor is skew-periodic, with a strong discrete spectrum (in the energy sense) consisting of two basic frequencies and a relatively weak continuous component. The corresponding eigenfunctions are plotted as functions of the kinetic energy and two time delays, as before. These eigenfunctions provide geometric slices on which the motion is purely quasiperiodic.

\subsection{Further analysis for Arctic sea ice}

\begin{figure}[t]
\centering
\includegraphics[width=0.4\textwidth,clip,trim={0mm 0mm 0mm 0mm}]{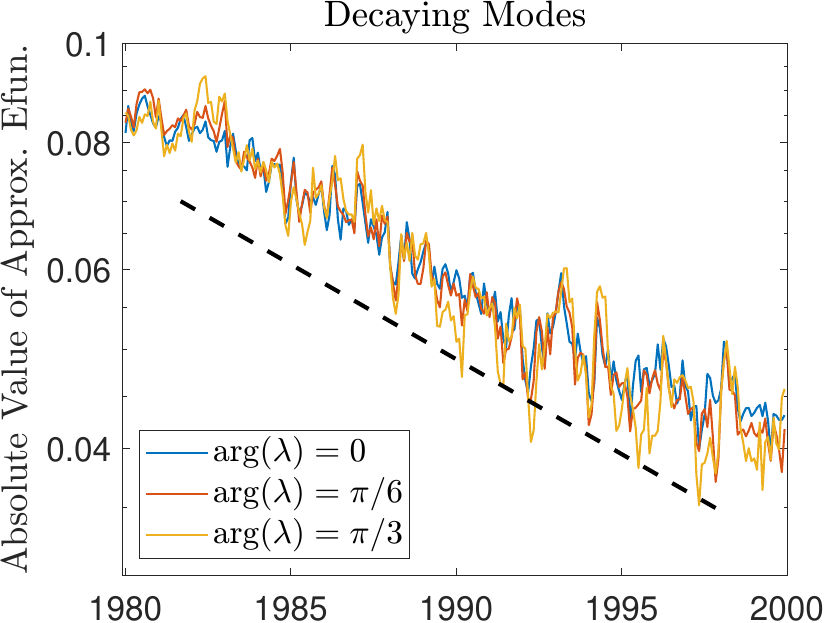}\hfill
\includegraphics[width=0.4\textwidth,clip,trim={0mm 0mm 0mm 0mm}]{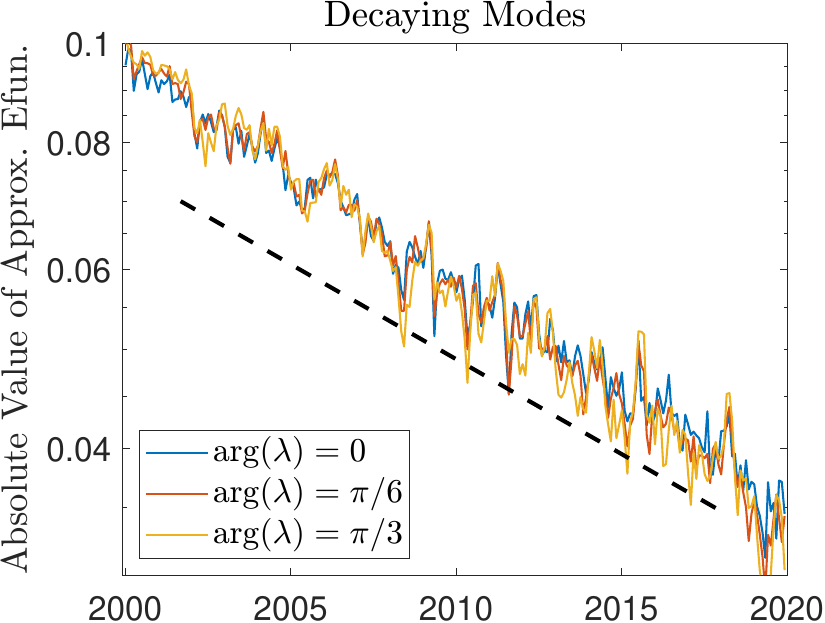}\\
\includegraphics[width=0.4\textwidth,clip,trim={0mm 0mm 0mm 0mm}]{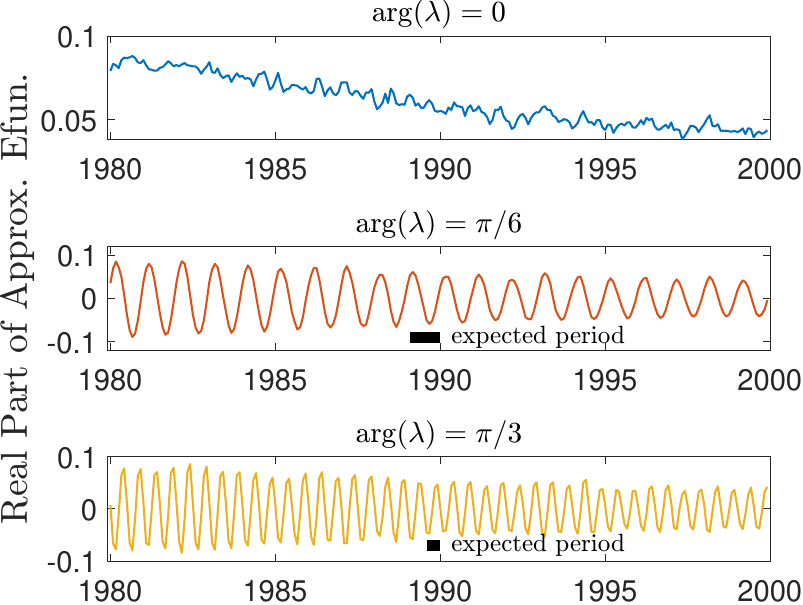}\hfill
\includegraphics[width=0.4\textwidth,clip,trim={0mm 0mm 0mm 0mm}]{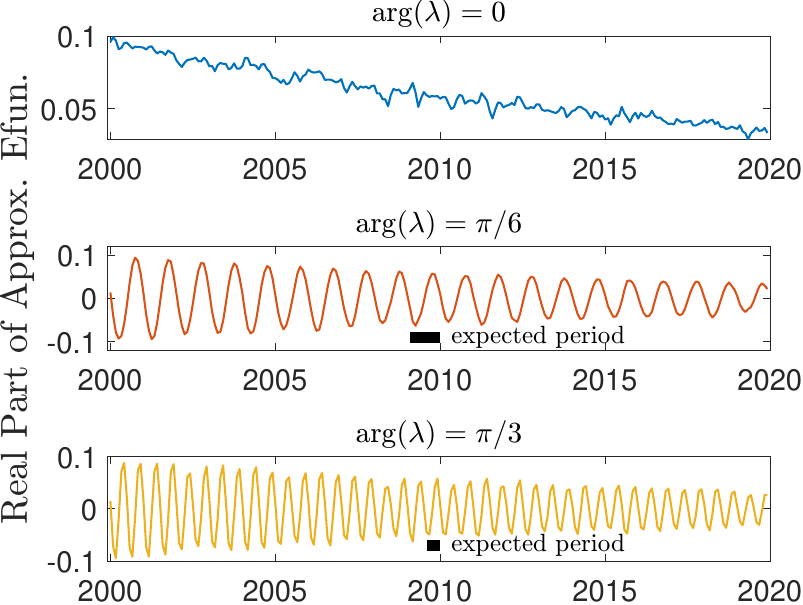}
\caption{Time evolution of the hidden decaying Koopman eigenfunction $\phi_{\mathrm{decay}}$ computed separately over the first and second halves of the dataset. We only show those whose eigenvalue have non-negative complex arguments due to complex-conjugate symmetry. The top row shows the absolute value $|\phi_{\mathrm{decay}}|$, and the bottom row shows the real part $\mathrm{Re}(\phi_{\mathrm{decay}})$. The left and right columns correspond to the first and second halves of the dataset, respectively. The dashed lines indicate the exponential decay rate predicted by the associated eigenvalue $\lambda$. Both halves exhibit near-monotonic decay with broadly consistent slopes, indicating robustness of the estimated long-term decay mode across the analysis interval.}\label{arctic_extra1}
\end{figure}

\cref{arctic_extra1} shows the time evolution of the hidden decaying modes from the main text computed over the first and second halves of the dataset. The magnitude decays approximately exponentially at approximately the rate expected from the eigenvalue $\lambda$ (dashed line), suggesting a connection to long-term sea-ice decline associated with climate change. This connection is likely nuanced: dissipative Koopman eigenfunctions exhibit near-monotonic reduction as time increases, whereas observed sea-ice loss rates over recent decades are known to be non-monotonic. Since the full sea-ice state is expressed as a sum of modes, the near-monotonic behavior of $\phi_{\mathrm{decay}}$ does not contradict the non-monotonic trend patterns seen in the data. In particular, the second half of the dataset produces an eigenfunction with a slightly faster rate of decay than the first, although the slopes are broadly consistent with those obtained from the full-interval analysis. The oscillations of the eigenfunctions follow the period expected from the complex argument of the associated eigenvalues.

\begin{figure}[t]
\centering
\includegraphics[width=0.32\linewidth]{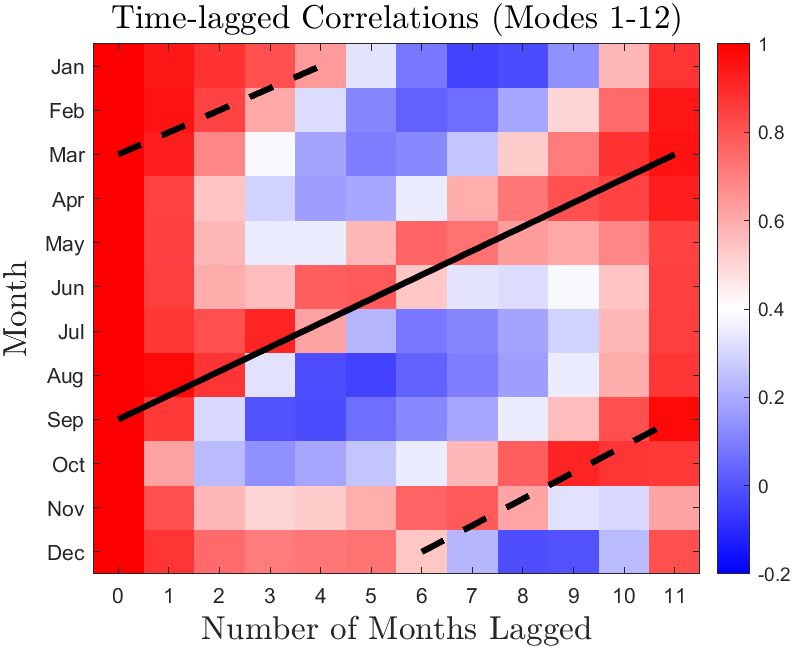}\hfill
\includegraphics[width=0.32\linewidth]{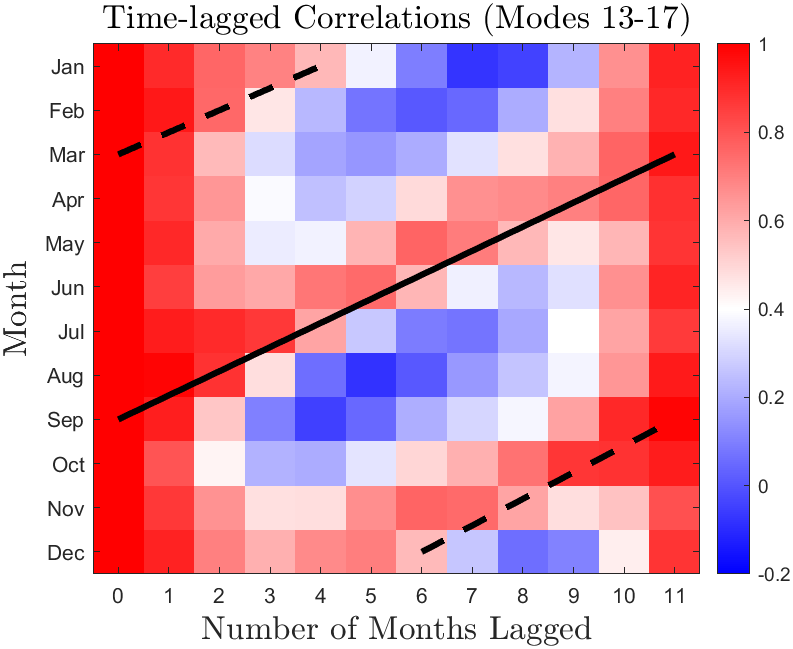}\hfill
\includegraphics[width=0.32\linewidth]{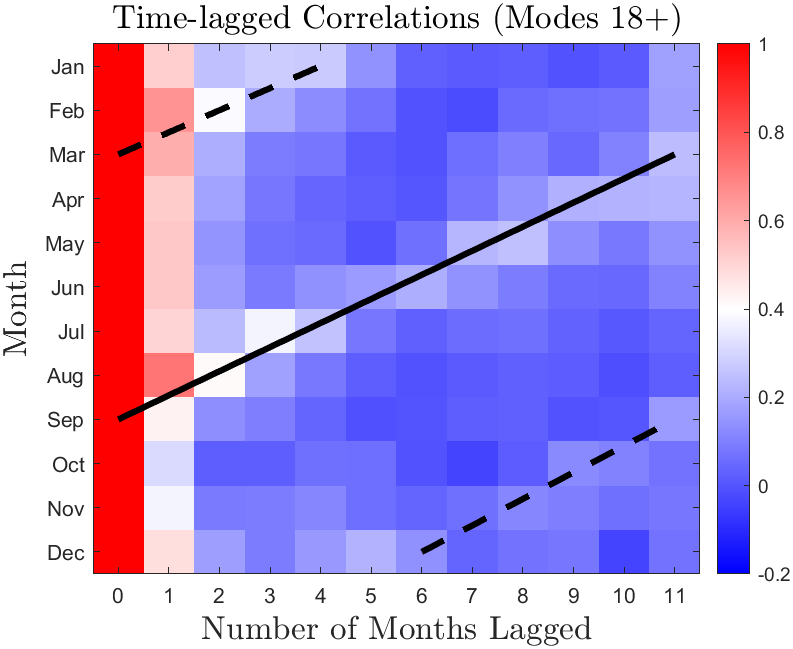}
\caption{Time-lagged pattern correlations of sea ice anomalies. Solid lines indicate months with increased correlation due to melt-to-growth reemergence, while dashed lines indicate increased correlation due to growth-to-melt reemergence. The modes are ordered by their approximation errors: the leftmost corresponds to the dominant annual variation, the middle group contains the five hidden decaying modes, and the rightmost includes the remaining modes.}\label{arctic_extra2}
\end{figure}

\cref{arctic_extra2} shows the autocorrelation functions of sea ice anomalies reconstructed from individual Koopman modes. Correlations are computed by treating the sea ice concentration at each grid point as a component of a vector. Notably, the hidden modes capture seasonally modulated reemergence of correlations, whereby sea ice anomalies that develop during the growth season reappear in the following melt season, despite a loss of correlation during the intervening winter months \cite{blanchard2011persistence}. This behavior highlights the Koopman decomposition's ability to capture meaningful long-term memory in sea ice dynamics.

\cref{arctic_extra3} shows the Koopman spectrum computed using data from 2012–2021 (and the same methods used to produce Figure 2 in the main text), the period used in the binary accuracy forecasting experiment. The resulting spectra and associated errors are broadly similar to those obtained from the full dataset. However, the precise locations of the hidden modes vary, and the error metric is generally higher when using a smaller training dataset, as expected. 

\begin{figure}[t]
\centering
\includegraphics[width=0.4\textwidth,clip,trim={0mm 0mm 0mm 0mm}]{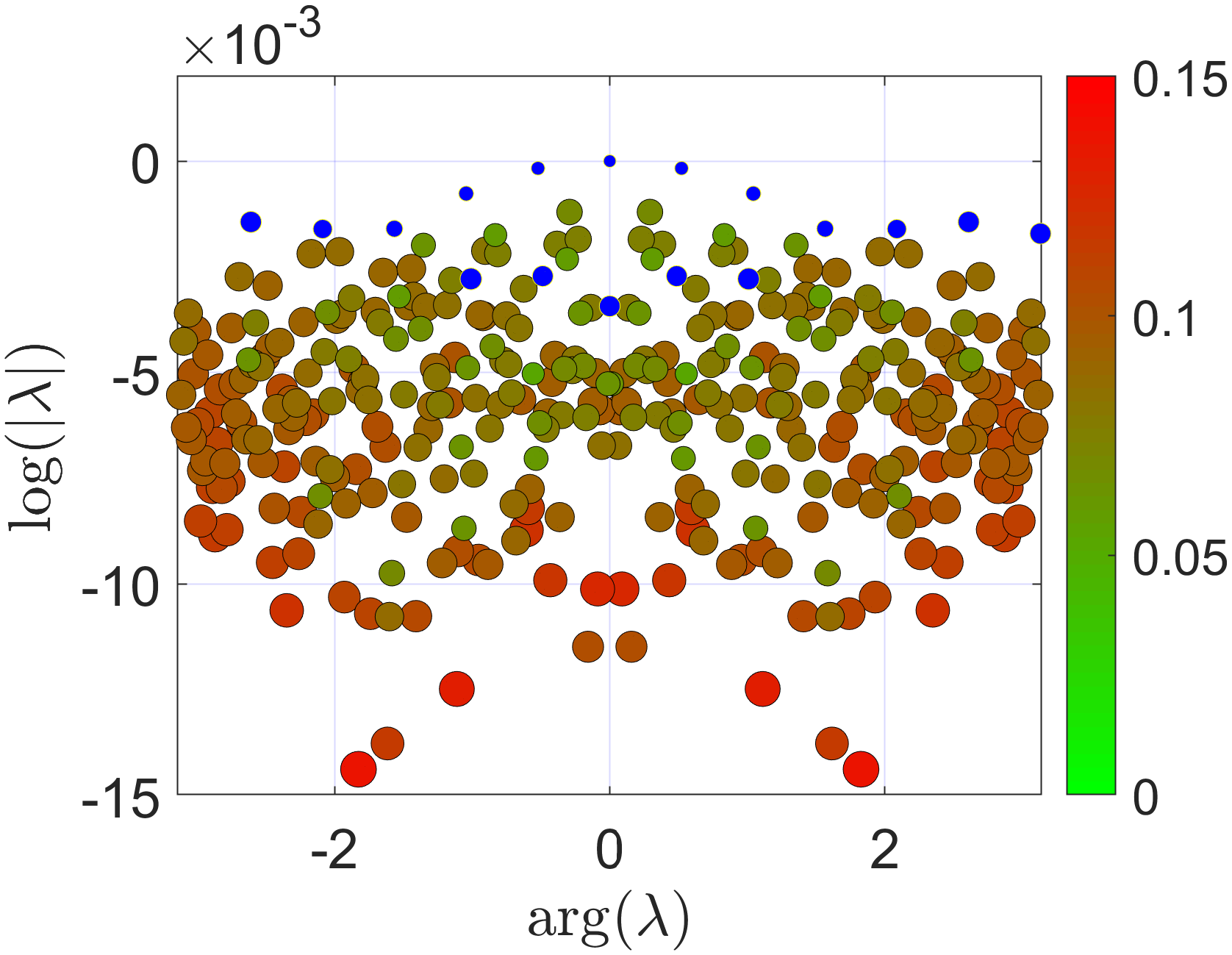}
\caption{Koopman spectrum for 2012--2021 Arctic sea ice data. The spectral features and reconstruction errors are similar to those obtained from the full dataset.}\label{arctic_extra3}
\end{figure}

\begin{figure}[t]
\centering
\includegraphics[width=0.4\textwidth,clip,trim={0mm 0mm 0mm 0mm}]{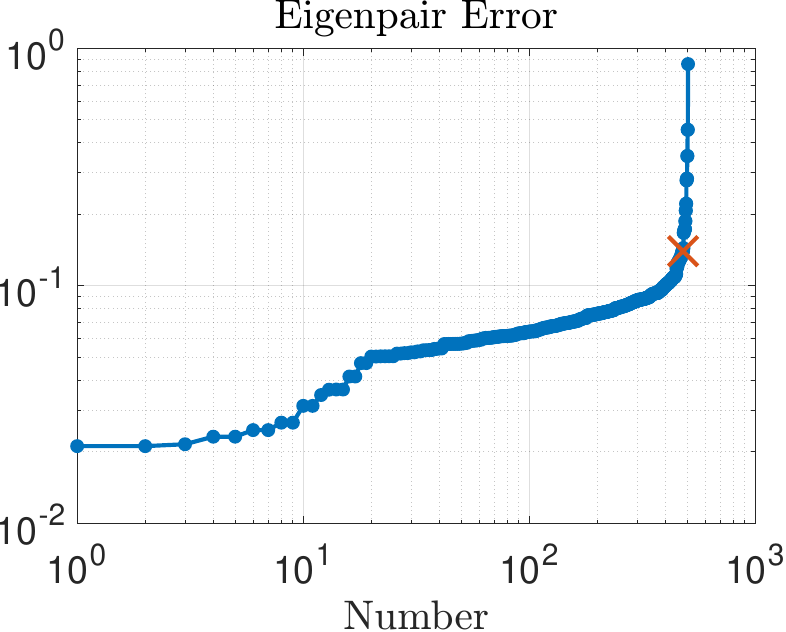}
\caption{Eigenpair errors for the Arctic sea ice example computed using the method outlined in the proof of \cref{thm:general_koopman_computation}. We have ordered the pairs according to their errors and the red cross shows a typical elbow point of the curve, which we use for optimal truncations in the main text.}\label{acrtic_elbow}
\end{figure}

We now plot the errors of the EDMD eigenvalues in Figure 2 of the main text. \cref{acrtic_elbow} presents the results, with eigenpairs ordered by their computed errors. To mitigate spurious eigenpairs in Koopman mode expansions, we truncate at the elbow of the curve, marked by the red cross. This provides a principled way to choose the truncation parameter $\varepsilon_0$ in the Koopman mode decomposition described in the Methods section.

\subsection{Other dynamical systems}
\label{SI_sec:others}

\begin{figure}[t]
\centering
\includegraphics[width=0.85\textwidth,trim={4mm 1mm 35mm 4mm},clip]{figure_7.pdf}
\caption{Spectral analysis of a range of analytic and real-world dynamical systems using our proposed algorithms. Each row corresponds to a different system: (i) periodic flow past a cylinder, (ii) Lorenz system, (iii) Rössler system, (iv) electrocardiogram (ECG) data (Image: \emph{Heart} by H. G. Wetselaar, Leiden University Libraries / Europeana, Public Domain), and (v) monthly mean sea surface height in the Northern Hemisphere (1950–present). Columns show: the system or dataset, a representative time series of observables, the computed pseudospectra (the sublevel sets of $\varepsilon$ for $\varepsilon$-approximate eigenfunctions) overlaid with EDMD eigenvalues (blue dots), and a convergence comparison between EDMD (blue) and our method. Except for the periodic flow, EDMD yields spurious eigenvalues and lacks convergence. In contrast, our approach produces qualitatively accurate and convergent spectral approximations. Notably, the pseudospectra reveal distinct spectral structures across systems: continuous spectra for chaotic systems (Lorenz, Rössler), spectral clustering near $\lambda = 1$ for ECG, and non-normal features and seasonal modes around $\lambda \approx \exp(2\pi i/12)$ in sea surface height data.\vspace{5mm}}
\label{fig:other_examples2}
\end{figure}

We now provide further details of the additional experiments in the main text, which we reproduce in \cref{fig:other_examples2} for the convenience of the reader. In each case, we specify the dynamical system, how the data were collected, and the choice of observables. This information fully specifies the input required for the EDMD algorithm (see main text) and our convergent algorithms, both of which use the same data and dictionary for each system.

\subsubsection{Flow past a cylinder}

We consider flow past a circular cylinder of diameter $D=1$ with Reynolds number $Re=100$, which exceeds the critical Reynolds number at which the flow undergoes a supercritical Hopf bifurcation, resulting in laminar vortex shedding \cite{jackson1987finite,zebib1987stability}.\index{Hopf bifurcation} The flow has a stable limit cycle representative of the three-dimensional flow \cite{noack1994global,noack2003hierarchy}. Due to its simplicity and relevance in engineering, this is one of the most studied examples in modal-analysis techniques \cite[Table 3]{rowley2017model}, \cite{chen2012variants,taira2020modal}. The Koopman operator of the post-transient flow is unitary with eigenvalues $\{\lambda_0^n:n\in\mathbb{Z}\}$ on the unit circle \cite{bagheri2013koopman}. We compute the vorticity field using an incompressible, two-dimensional lattice-Boltzmann solver \cite{jozsa2016validation,szHoke2017performance}, with a time step so that approximately 24 snapshots of the flow field correspond to the period of vortex shedding. The computational domain is $18D\times5D$, with a $800\times 200$ grid resolution and ambient state-space dimension $d=158624$. This is less than $800\times 200$ due to the presence of the cylinder, which is positioned $2D$ downstream of the inlet at the mid-height of the domain. The cylinder side walls are bounce-back and no-slip, with a parabolic velocity profile at the inlet of the domain and non-reflecting outflow at the outlet. After simulations converge to statistically stationary periodic vortex shedding, we collect $M=120$ snapshots.

Let $\Xv\in \mathbb{R}^{d\times 120}$ be the data matrix of the vorticity field at the snapshots and $\Yv\in \mathbb{R}^{d\times 120}$ the corresponding matrix after one time step. We compute a truncated SVD of the data matrix $\Xv \approx \Uv \mathbf{\Sigma}\Vv^*$, $\Uv\in\mathbb{R}^{d\times N}$, $\mathbf{\Sigma}\in\mathbb{R}^{N\times N}$, $\Vv\in\mathbb{R}^{M\times N}.$ The columns of $\Uv$ and $\Vv$ are orthonormal and $\mathbf{\Sigma}$ is diagonal. As our dictionary, we use the $N$ first left singular vectors of the data matrix of the vorticity field (POD modes \cite{berkooz1993proper}). The $j$th row of the POD matrix $\Uv^*\Xv$ is an affine function evaluated at the snapshot data. The pseudospectra plot was computed using $N=50$ observables. The trajectory data plotted in the figure in the paper (\cref{fig:other_examples2} in S.I.) correspond to the first five POD modes (note the periodic behavior).

\subsubsection{Lorenz system}

The famous Lorenz (63) system \cite{lorenz1963deterministic} is the following three coupled ordinary differential equations:
$$
\dot{u}_1=\sigma\left(u_2-u_1\right),\quad\dot{u}_2=u_1\left(\rho-u_3\right)-u_2,\quad \dot{u}_3=u_1u_2-\beta u_3.
$$
The system describes a truncated model of Rayleigh--B\'{e}nard convection, where the parameters $\sigma$, $\rho$, and $\beta$ are proportional to the Prandtl number, Rayleigh number, and the physical proportions of the fluid, respectively. We take the standard values $\sigma=10$, $\rho=28$, and $\beta=8/3$ and consider the dynamics of $x=(u_1,u_2,u_3)$ on the Lorenz attractor $\mathcal{X}\subset\mathbb{R}^3$. We consider a discrete-time dynamical system by sampling with a time-step $\Delta t=0.05$. Hence, $F(x)$ is the solution at $t=0.05$ to the above equations with initial value $x$. This system has a unique SRB measure $\omega$ on $\mathcal{X}$ \cite{Tucker2002}, meaning that for Lebesgue-almost every initial condition $x_0$ in the basin of attraction of $\mathcal{X}$ and for every compactly supported continuous function $g:\mathbb{R}^3\rightarrow\mathbb{C}$,
$$
\lim_{M\rightarrow\infty}\frac{1}{M}\sum_{m=0}^{M-1}\left[\mathcal{K}^{m}g\right](x_0)=\int_{\mathcal{X}}g(x)\dd \omega(x).
$$
We use a time-delay embedding dictionary. Specifically, for $j=1,2,3$, we define the observables
$$
g_{j,1}(x)=u_j,\quad g_{j,k+1}=\mathcal{K}g_{j,k},\quad k=1,\ldots,N-1.
$$
We consider $M$ snapshots $x^{(1)}=x_0$ and $x^{(m+1)}=F(x^{(m)})$ drawn from a single trajectory with quadrature weights $w_m=1/M$. Since $g_{j,k}(x^{(m)})=g_{j,1}(x^{(m+k-1)})$, we can evaluate the dictionary from the single trajectory. We use the \texttt{ode45} command in MATLAB to collect the data after an initial `burn-in' time to ensure that the initial point $x_0$ is (approximately) on the Lorenz attractor. The system is chaotic, so we cannot hope to numerically integrate accurately for long periods. However, convergence is still obtained as $M\rightarrow\infty$ due to an effect known as shadowing \cite{pilyugin2006shadowing}. The pseudospectra plot was computed using $N=100$ (300 observables) and $M=10^4$ snapshots, and the shown trajectory data is for $u_1$ (first coordinate).

\subsubsection{Rössler system}

The R\"ossler system \cite{rossler1976equation} is defined by the following three coupled ordinary differential equations:
$$
\dot{u}_1=-u_2-u_3,\quad\dot{u}_2=u_1+0.1u_2,\quad \dot{u}_3=0.1+u_3(u_1-14).
$$
We study the dynamics of $x=(u_1,u_2,u_3)$ on the R\"ossler attractor. Although often regarded as a simplified analogue of the Lorenz (63) system, the Rössler system—despite its relatively simple form—exhibits rich and complex dynamical behavior. The setup for data collection and dictionary construction is the same as in the Lorenz example, with one key difference: to illustrate the use of alternative dictionaries, we define the observables as
$$
g_{1}(x)=u_1,\quad g_{k+1}=\mathcal{K}g_{k},\quad k=1,\ldots,N-1,
$$
for our dictionary. The pseudospectra plot was computed using $N=100$ observables and $M=10^4$ snapshots, and the shown trajectory data is for $u_1$ (first coordinate).

\subsubsection{Electrocardiogram (ECG)}

An electrocardiogram (ECG) measures the electrical activity of the heart, producing the characteristic spiking pulses associated with each heartbeat. In this analysis, we use the ECG signal qtdb/sel102 from \cite{keogh2005hot} with $M=44499$ snapshots. We construct a dictionary consisting of $N-1$ time delays, yielding $N$ observables in total, following the same approach described in the Methods section for the Arctic sea ice example. This type of time-delay dictionary is commonly used in the analysis of ECG recordings \cite{schreiber1996nonlinear,richter1998phase}. The pseudospectra plot was computed using $N=50$ observables.

\subsubsection{Northern Hemisphere sea surface height}

This final example examines the dynamics of the monthly mean sea surface height in the Northern Hemisphere.
We use the OFES simulation of hindcast data from 1950 to the present day ($M=359$ snapshots), which was conducted on the Earth Simulator with the support of JAMSTEC \cite{jamstec_2009}. The dataset spans latitudes from $74.95$\textdegree S to $74.95$\textdegree N in intervals of $0.1$\textdegree (excluding Arctic regions), and longitudes from $0.05$\textdegree to $359.95$\textdegree, also in intervals $0.1$\textdegree. Since the Northern and Southern Hemispheres exhibit different dynamics, we restrict our analysis to latitudes between $0.05$\textdegree N and $74.95$\textdegree N. At each longitude-latitude point, the dataset provides monthly mean sea surface height, and we exclude any regions that are not ocean-covered.

We construct the dictionary using a kernelized version of the POD dictionary used for the cylinder flow. The description is given in \cite{williams2015kernel}, but amounts to kernel basis functions of the form
$$
\left(1+{\bf{x}}^\top{\bf{x}}^{(m)}\right)^{20}\exp(-\|{\bf{x}}-{\bf{x}}^{(m)}\|_{l^2}/\sigma),
$$
where the scaling parameter $\sigma$ is set to the average  $l^2$-norm of the snapshot data after centering it to have zero mean. The pseudospectra plot was computed using $N=250$ observables. The trajectory data plotted in the figure in the paper (\cref{fig:other_examples2} in S.I.) correspond to the first five functions.

\renewcommand{\arraystretch}{0.8}
\bibliographystyle{unsrt}

{\small\linespread{0.6}\selectfont{}
\bibliography{KoopFoundations}}

\end{document}